\documentclass{amsart}
\pdfoutput=1

\usepackage{amsthm, amsmath, amssymb}
\usepackage{amsfonts, enumerate, graphics, graphicx}
\usepackage{mathtools} 
\usepackage[all,cmtip]{xy} 
\usepackage{subcaption}
\usepackage[colorlinks=true,citecolor=black,linkcolor=black,urlcolor=blue]{hyperref}

\graphicspath{{figures/}}

\theoremstyle{plain}
\newtheorem{thm}{Theorem}[section]

\newtheorem{prop}[thm]{Proposition}
\newtheorem{lemma}[thm]{Lemma}

\theoremstyle{definition}
\newtheorem{defn}[thm]{Definition}
\newtheorem{ex}[thm]{Example}
\newtheorem{remark}[thm]{Remark}

\newcommand{\defemph}[1]{\emph{\color{blue} #1}} 

\newcommand{\Z}{\mathbb{Z}}
\newcommand{\R}{\mathbb{R}}

\newcommand{\N}{\mathbb{N}}
\newcommand{\T}{\mathbb{T}}
\newcommand{\Q}{\mathcal{Q}}
\renewcommand{\k}{\mathbf{k}}
\newcommand{\fund}{\mathbf{fund}}
\newcommand{\proj}{\mathbf{proj}}
\newcommand{\diag}{\operatorname{Diag}}
\newcommand{\Zodd}{\mathbb{Z}^3_{\rm odd}}
\newcommand{\F}{\mathcal{F}}
\renewcommand{\P}{\mathbb{P}}
\newcommand{\sgn}{\operatorname{sgn}}
\newcommand{\Trop}{\operatorname{Trop}}
\newcommand{\G}{\overline{G}}
\newcommand{\edge}{%
  \mathrel{-}
  \joinrel\joinrel 
  \mathrel{-}
}

\title{solutions to the T-systems with principal coefficients}
\author{Panupong Vichitkunakorn}
\address{Department of Mathematics, University of Illinois, Urbana, IL 61821 USA}
\address{vichitk1@illinois.edu}
\date{\today}

\begin{document}

\begin{abstract}
The $A_\infty$ T-system, also called the octahedron recurrence, is a dynamical recurrence relation. It can be realized as mutation in a coefficient-free cluster algebra (Kedem 2008, Di Francesco and Kedem 2009). We define T-systems with principal coefficients from cluster algebra aspect, and give combinatorial solutions with respect to any valid initial condition in terms of partition functions of perfect matchings, non-intersecting paths and networks. This also provides a solution to other systems with various choices of coefficients on T-systems including Speyer's octahedron recurrence (Speyer 2007), generalized lambda-determinants (Di Francesco 2013) and (higher) pentagram maps (Schwartz 1992, Ovsienko et al. 2010, Glick 2011, Gekhtman et al. 2014).
\end{abstract}

\maketitle 


\section{Introduction}\label{sec_Intro}

The $A_\infty$ T-system \cite{DFK13}, also called the octahedron recurrence, is a discrete dynamical system of formal variables $T_{i,j,k}$ for $i,j,k\in\Z$ satisfying
\[
	T_{i,j,k-1}T_{i,j,k+1} = T_{i-1,j,k}T_{i+1,j,k} + T_{i,j-1,k}T_{i,j+1,k}.
\]
This recurrence relation preserves the parity of $i+j+k$ and therefore there are two independent systems depending on the parity of $i+j+k.$
We will assume throughout the paper that $i+j+k\equiv 1 \bmod 2$.
Several combinatorial solutions have been considered including solutions in terms of alternating sign matrices \cite{RR86}, domino tilings \cite{RR86,EKLP}, perfect matchings \cite{Speyer} and networks \cite{DFK13}.

The $A_\infty$~T-system can also be interpreted \cite{DFK09} as mutation in an infinite-rank cluster algebra \cite{FZ} of geometric type without coefficients.
Its corresponding quiver is called the octahedron quiver \cite{DFK13}.

One can also consider cluster algebras with coefficients \cite{FZ4}.
For cluster algebras of geometric type, this is equivalent to adding frozen vertices to the quiver.
A quite general type of coefficients are principal coefficients.
It corresponds to adding one frozen vertex for each quiver vertex and an arrow pointing from it to the quiver vertex.
In some literature, this new quiver is called the coframed quiver associated with the octahedron quiver.
The reason why the principal coefficients are very important is due to the separation formula \cite[Theorem 3.7]{FZ4}, stating that a cluster algebras with any coefficients can be written in terms of one with principal coefficients.

Some generalizations of T-systems with coefficients have been suggested by Speyer in his work on Speyer's octahedron recurrence \cite{Speyer} and by Di Francesco in his work on the generalized lambda determinant \cite{DF13}.
In this paper, we consider $A_\infty$ T-systems with principal coefficients using cluster algebras definition. We then give combinatorial solutions in terms of perfect matchings, non-intersecting paths and networks.

The paper is organized as follows.
In Section \ref{sec_Cluster}, we review some basic definitions and results in cluster algebras from \cite{FZ,FZ4}.
In Section \ref{sec_T-system}, we define the T-system with principal coefficients, whose initial condition is in the form of initial data on a stepped surface.

The goal is to find, for an arbitrary point $(i_0,j_0,k_0)$ and a stepped surface $\k$, an expression of $T_{i_0,j_0,k_0}$ in terms of initial data on $\mathbf{k}$.
Laurent phenomenon for cluster variables \cite[Theorem 3.1]{FZ} guarantees that the expression is indeed a Laurent polynomial in the initial data and coefficients.
In this paper, we give explicit combinatorial expressions of $T_{i_0,j_0,k_0}$ in terms of initial data when the point $(i_0,j_0,k_0)$ is above $\mathbf{k}$ and $\mathbf{k}$ is above the fundamental stepped surface $\fund:(i,j)\mapsto(i+j \bmod 2)-1.$
Some other cases will be discussed in Section \ref{sec_General}.

Section \ref{sec_Dimer} is devoted to a perfect matching solution.
Following the construction in \cite{Speyer}, we first construct a finite bipartite graph with open faces $G$ depending on both $(i_0,j_0,k_0)$ and $\mathbf{k}$, then construct face-weight $w_f$ and pairing-weight $w_p$ on perfect matchings of $G.$
This leads to the perfect-matching solution (Theorem \ref{thm:main}):
\[ T_{i_0,j_0,k_0} = \sum_{M} w_p(M)w_f(M) \]
where the sum runs over all perfect matchings of $G$.
The weight $w_p(M)$ is a monomial in the cluster coefficients, while $w_f(M)$ is a Laurent monomial in the initial data (cluster variables).

In Section \ref{sec_edge-weight}, we define the closure $\overline{G}$ of the graph $G$ and transform our previous two weights to the edge-weight $w_e.$
This gives another form of the perfect-matching solution (Theorem \ref{thm:edgesol}):
\[
	T_{i_0,j_0,k_0} = \sum_{\overline{M}} w_e(\overline{M}) \Big/ w_e(\overline{M}_0)\big|_{c_{i,j}=1}.
\]
The sum runs over all perfect matchings of $\overline{G}$ with a certain boundary condition, see Definition \ref{defn:Mbar0}.

In Section \ref{sec_path}, we orient all the edges of $G$ and $\overline{G}$ and give an explicit bijection between perfect matchings of $G$ and non-intersecting paths on $G$ with certain sources and sinks.
This bijection can also be extended to $\overline{G}.$
Using the modified edge-weight $w'_e$ obtaining from $w_e$ together with the bijection, the perfect-matching solution for $\overline{G}$ gives the nonintersecting-path solution (Theorem \ref{thm_pathsol}):
\[
	T_{i_0,j_0,k_0} = \sum_{\overline{P}} w_e'(\overline{P}) \Big/ \prod_{\substack{\circ\text{---}\bullet\\b}\in \overline{M}_0}\bar{p}_b
\]
where the sum runs over all non-intersecting paths on $\overline{G}$ with certain sources and sinks, see Theorem \ref{thm_pathsol}.

In Section \ref{sec_network}, we first consider the network $N$, studied in \cite{DF10,DFK13} associated with $(i_0,j_0,k_0)$ and $\mathbf{k}$.
It is obtained from the shadow of the point $(i_0,j_0,k_0)$ on the lozenge covering on $\mathbf{k}$.
We point out that it can also be obtained from the graph $\overline{G}$ by tilting all the diagonal edges of $\overline{G}$ so that they become horizontal.
This allows us to pass the modified edge-weight $w'_e$ on $\overline{G}$ to a weight on the network $N$.
Paths on $\overline{G}$ then become paths on $N$.
From Theorem \ref{thm_pathsol}, we get the network solution (Theorem \ref{thm:network}) as a partition function of weighted non-intersecting paths on the network $N$, which can also be written as a certain minor of the network matrices (Theorem \ref{thm:networkmatrix}).

In Section \ref{sec_app}, we discuss other types of coefficients of the T-system related to Speyer's octahedron recurrence \cite{Speyer}, generalized lambda-determinants \cite{DF13} and (higher) pentagram maps \cite{Schwartz,OST,Glick,GSTV14}.


\subsection*{Acknowledgements}
The author would like to thank his advisors R. Kedem and P. Di Francesco for their helpful advice and comments. The author also thank G. Musiker and M. Glick for discussions. This work was partially supported in part by a gift to the Mathematics Department at the University of Illinois from Gene H. Golub, the NSF grant DMS-1100929, NSF grant DMS-1301636 and the Morris and Gertrude Fine endowment.


\section{Cluster algebras}\label{sec_Cluster}
In this section, we quote some basic definitions and important results in the theory of cluster algebras mainly from \cite{FZ,FZ4}.

\subsection{Finite rank cluster algebras}
Let $(\P,\oplus,\cdot)$ be a semifield, i.e., $(\P,\cdot)$ is an abelian group, and $\oplus$ is an auxiliary addition: commutative, associative and distributive with respect to the multiplication.
The following are two important examples of semifields.

\begin{defn}[Universal semifield]
  For a set of labels $J,$ the \defemph{universal semifield} on the set of variables $\{y_j\mid j\in J\}$ denoted by $\mathbb{Q}_{sf}(y_j :j\in J) := (\mathbb{Q}_{sf}(y_j :j\in J),+,\cdot)$ is the set of all subtraction-free expressions of rational functions in independent variables $\{y_j \mid j\in J \}$ over $\mathbb{Q}$ with the usual addition and multiplication.
\end{defn}

\begin{defn}[Tropical semifield]
  For a set of labels $J,$ the \defemph{tropical semifield} on the set of variables $\{y_j\mid j\in J\}$ denoted by $\Trop(y_j : j\in J)$ is the free multiplicative abelian group generated by $\{y_j \mid j\in J\}$ with the auxiliary addition defined by:
  \[
    \prod_j y_j^{a_j}\oplus\prod_j y_j^{b_j} = \prod_j y_j^{\min(a_j,b_j)}.
  \]
\end{defn}

Here are some notations that will be used throughout the paper:
\[
  [x]_+=\max(0,x),~ [1,n]=\{1,2,\dots,n\},~ \sgn(x) = 
  \begin{cases} 
    -1, &x<0\\
    0,	&x=0\\
    1,	&x>0
  \end{cases}.
\]

Let $n\in\N,$ we now state the main definitions of cluster algebras of rank $n.$
Let~$\P$ be a semifield, $\F=\mathbb{Q}\P(x_1,\dots,x_n)$ an ambient field, $\T_n$ the $n-$regular tree whose $n$ edges incident to each vertex have different labels from $1$ to $n.$

\begin{defn}[Cluster patterns and Y-patterns]\label{def:mutation}
  A \defemph{cluster pattern} (resp. \defemph{Y-pattern}) is an assignment $t\mapsto (\mathbf{x}_t,\mathbf{y}_t,B_t)$ (resp. $t\mapsto(\mathbf{y}_t,B_t)$) of any vertex $t\in\T_n$ to a \defemph{labeled seed} (resp. \defemph{labeled Y-seed}) such that:
  \begin{itemize}
    \item The \defemph{cluster} tuple $\mathbf{x_t}=(x_{1;t},\dots,x_{n;t})$ is an $n-$tuple of elements of $\F$ forming a free generating set.
    \item The \defemph{coefficient} tuple $\mathbf{y_t}=(y_{1;t},\dots,y_{n;t})$ is an $n-$tuple in $\P.$
    \item The \defemph{exchange matrix} $B_t=(b^{(t)}_{ij})\in M_{n\times n}(\Z)$ is a skew-symmetrizable matrix.
    \item $ t \stackrel{k}{\edge} t'$ in $\T_n$ if and only if $(\mathbf{x}_t,\mathbf{y}_t,B_t) \stackrel{\mu_k}{\longleftrightarrow} (\mathbf{x}_{t'},\mathbf{y}_{t'},B_{t'}),$ where $\mu_k$ is the \defemph{seed mutation} in direction k defined by:
    \begin{itemize}
      \item $B_{t'}=(b^{(t')}_{ij})$ where
      \[
        b^{(t')}_{ij}=
        \begin{cases}
          -b^{(t)}_{ij}, &i=k\text{ or }j=k,\\
          b^{(t)}_{ij}+\sgn(b^{(t)}_{ik}) [b^{(t)}_{ik} b^{(t)}_{kj}]_+, &\text{otherwise.}
        \end{cases}
      \]
      \item $\mathbf{y}_{t'}=(y_{1,t'},\dots,y_{n,t'})$ where
      \[
        y_{j,t'} = 
        \begin{cases}
          y_{k,t}^{-1},	& j=k,\\
          y_{j,t} (y_{k;t}^+)^{[b^{(t)}_{kj}]_+}(y_{k;t}^-)^{-[-b^{(t)}_{kj}]_+}, & j\neq k.
        \end{cases}
      \]
      \item $\mathbf{x}_{t'}=(x_{1;t'},\dots,x_{n;t'})$ where
      \[
        x_{j;t'} = 
        \begin{cases}
          x_{k,t}^{-1} \left( y_{k;t}^+ \prod_{i=1}^n x_{i;t}^{[b^{(t)}_{ik}]_+} + y_{k;t}^- \prod_{i=1}^n x_{i;t}^{[-b^{(t)}_{ik}]_+} \right),	& j=k,\\
          x_{j;t}, & j\neq k,
        \end{cases}
      \]
      where $y_{k;t}^+ = \dfrac{y_{k;t}}{(y_{k;t}\oplus 1)}$ and $y_{k;t}^- = \dfrac{1}{(y_{k;t}\oplus 1)}.$
    \end{itemize}
  \end{itemize}
\end{defn}

\begin{defn}[Cluster algebra]
  The \defemph{cluster algebra} $\mathcal{A}$ associated with a cluster pattern $t\mapsto (\mathbf{x}_t,\mathbf{y}_t,B_t)$ for $t\in\mathbb{T}_n$ is defined to be the $\mathbb{ZP}$-algebra generated by all cluster variables, i.e.
  $$ \mathcal{A} =  \mathbb{Z}\mathbb{P}[x_{i,t} : t\in\mathbb{T}_n,i\in[1,n]].$$
\end{defn}

In this paper, we will consider only cluster algebras with skew-symmetric exchange matrix.
In this case, we can think of the cluster mutation as a combinatorial rule performing on a quiver.

\begin{defn}[The quiver associated with a skew-symmetric $B$]
  When the exchange matrix $B=(b_{ij})_{i,j\in [1,n]}$ is skew-symmetric, we define $\mathcal{Q}_B$, the \defemph{quiver} associated with $B$, to be a directed graph with vertices $1,\dots,n$.
  There are $b_{ij}$ arrows from $i$ to $j$ if and only if $b_{ij}>0.$
\end{defn}

All the information of the exchange matrix $B$, skew-symmetric, is encoded in the quiver $\mathcal{Q}_B.$
The mutation $\mu_k$ will then act on $\mathcal{Q}_B$ by the following process:
\begin{enumerate}
  \item For every pair of arrows $i\rightarrow k$ and $k\rightarrow j$, add an arrow $i\rightarrow j$.
  \item Reverse all the arrows incident to $k$.
  \item Remove all oriented 2-cycles.
\end{enumerate}
One can easily check that $\mu_k(\mathcal{Q}_B) = \mathcal{Q}_{\mu_k(B)}.$
The following is an example of the quiver associated with an exchange matrix and the quiver mutation.
\[
  \begin{psmallmatrix}
    0&1&0&-1 \\
    -1&0&1&0 \\
    0&-1&0&1 \\
    1&0&-1&0 
  \end{psmallmatrix}
  \enspace\stackrel{\mu_1}{\longleftrightarrow}\enspace
  \begin{psmallmatrix}
    0&-1&0&1 \\
    1&0&1&-1 \\
    0&-1&0&1 \\
    -1&1&-1&0 
  \end{psmallmatrix}
  \enspace\stackrel{\mu_3}{\longleftrightarrow}\enspace
  \begin{psmallmatrix}
    0&-1&0&1 \\
    1&0&-1&0 \\
    0&1&0&-1 \\
    -1&0&1&0 
  \end{psmallmatrix}
\]

\begin{center}
  \includegraphics[scale=0.8]{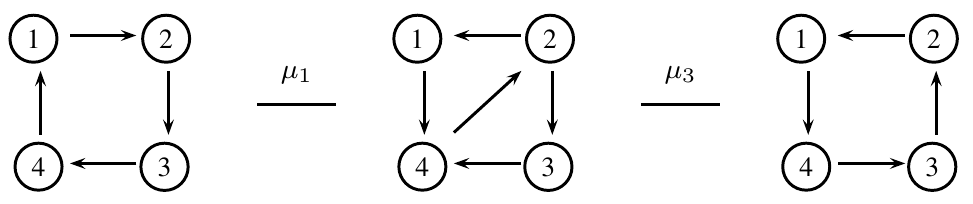}
\end{center}

Fixing $t_0\in\T_n,$ we can express variables $x_{1;t},\dots,x_{n;t},y_{1,t},\dots,y_{n;t}$ of the labeled seed at arbitrary $t\in\T_n$ in terms of variables $x_{1;t_0},\dots,x_{n;t_0},y_{1,t_0},\dots,y_{n;t_0}$ at $t_0.$
We will then call the labeled seed (resp. labeled Y-seed) at $t_0$ an \defemph{initial labeled seed} (resp. \defemph{initial labeled $Y$-seed}) and assign a simpler notations:
  \[
    \mathbf{x}=\mathbf{x}_{t_0},~\mathbf{y}=\mathbf{y}_{t_0},~B=B_{t_0}
  \]
where
  \[
    x_i = x_{i;t_0},~ y_i=y_{i,t_0},~b_{ij} = b_{ij}^{(t_0)} \quad(i,j\in[1,n]).
  \]

Unless stated otherwise, the initial labeled seed $(\mathbf{x},\mathbf{y},B)$ is always at $t_0$ and we denote $\mathcal{A}(\mathbf{x},\mathbf{y},B)$ the cluster algebra with an initial labeled seed $(\mathbf{x},\mathbf{y},B).$
Clearly, a cluster pattern (resp. Y-pattern) is completely determined by its initial labeled seed (resp. initial labeled Y-seed).
In addition, a cluster variable is a Laurent polynomial in the initial variables, as stated in the following theorem.

\begin{thm}[Laurent phenomenon {{\cite[Theorem 3.7]{FZ4}}}]\label{thm:Laurent}
  The algebra $\mathcal{A}(\mathbf{x},\mathbf{y},B)$ is contained in the Laurent polynomial ring $\Z\P[\mathbf{x}^{\pm 1}]$, i.e. every cluster variable is a Laurent polynomial over $\Z\P$ in the initial cluster seed $x_1,\dots,x_n$.
\end{thm}

If $y_i=1$ for all $i\in[1,n]$, we have $y_{i,t}=1$ for all $i\in[1,n],t\in\T_n.$
We call it a \defemph{coefficient-free} cluster algebra, and write just $(\mathbf{x}_t,B_t)$ for for its labeled seeds.
  
\begin{defn}[Frozen variables]
  For a cluster algebra (resp. cluster pattern)  of rank $m$ with initial seed $(\mathbf{x},\mathbf{y},B)$, we consider a subpattern $t\in \T_n\subseteq \T_m \mapsto (\mathbf{x}_t,\mathbf{y}_t,B_t)$.
  It is the pattern obtained by $\mu_1,\dots,\mu_n$.
  That means the directions $n+1,\dots,m$ are not mutated.
  We call it a cluster algebra (resp. cluster pattern) of rank $n$ with \defemph{frozen variables} $x_{n+1},\dots,x_m.$
\end{defn}

\begin{remark}
  In the cluster algebra of rank $n$ with frozen variables $x_{n+1},\dots,x_m,$ the cluster seeds are not mutated in directions $n+1,\dots,m$.
  So, the necessary information in the $m\times m$ matrix $B$ for mutations are only the columns $1$ to $n.$
  Hence we often use a \defemph{reduced exchange matrix} $\tilde{B}$ instead of the full exchange matrix $B$, where $\tilde{B}$ is the $m\times n$ submatrix of $B$ obtained by deleting columns $n+1$ to $m.$
\end{remark}


\begin{defn}[Geometric type]
  A cluster algebra (or cluster pattern, or Y-pattern) is of \defemph{geometric type} if $\P$ is a tropical semifield.
\end{defn}

\begin{remark}[Geometric type and Frozen variables]\label{rem:GeometricFrozen}
  For a cluster algebra or cluster pattern of geometric type, the notion of coefficients and frozen variables are interchangeable.
  Let $t\in\T_n\mapsto(\mathbf{x}_t,\mathbf{y}_t,B_t)$ be a cluster pattern of geometric type of rank $n$ where $\P = \Trop(x_{n+1},\dots,x_{m}).$ Since $x_{n+1},\dots,x_{m}$ generate $\P,$ we can choose the initial seed coefficients to be $y_j = \prod_{i=n+1}^m x_i^{{b}_{ij}}$ for all $j\in [1,n].$
  Then the pattern is equivalent to a coefficient-free cluster pattern: $t\in\T_m\mapsto ((x_{1;t},\dots,x_{m;t}),\tilde{B}_t)$ with frozen variables $x_{n+1},\dots,x_{m}$, where $\tilde{B} = (b_{ij})_{m\times n}.$
\end{remark}

\begin{ex}\label{ex:frozen}
  Consider a semifield $\P=\Trop(x_5,x_6)$ and a rank 4 cluster algebra of geometric type with an initial seed $(\mathbf{x},\mathbf{y},B)$ where
  \[
    \mathbf{x}=(x_1,x_2,x_3,x_4),\quad\mathbf{y}=\left(\frac{x_6}{x_5},\frac{1}{x_6},1,1\right),\quad B=
    \begin{psmallmatrix}
      0&1&0&-1 \\
      -1&0&1&0 \\
      0&-1&0&1 \\
      1&0&-1&0 
    \end{psmallmatrix}.
  \]
  We have $y_1^+=x_6$ and $y_1^-=x_5$.
  After the seed mutation $\mu_1$ (Definition \ref{def:mutation}), we have
  \[
    \mathbf{x}'=\left( \dfrac{x_6x_4+x_5x_2}{x_1} , x_2,x_3,\dots,x_6 \right), \quad \mathbf{y}' = \left( \dfrac{x_5}{x_2}, 1, 1, \dfrac{1}{x_5} \right).
  \]
  On the other hand, by Remark \ref{rem:GeometricFrozen}, we can think of $x_5,x_6$ as frozen variables and transform our cluster algebra with coefficients to the coefficient-free cluster algebra of rank 6 with the following cluster variables and reduced exchange matrix.
  \[
    \mathbf{x}=(x_1,x_2,\dots,x_6),\quad \tilde{B}=
    \begin{psmallmatrix}
      0&1&0&-1 \\
      -1&0&1&0 \\
      0&-1&0&1 \\
      1&0&-1&0 \\
      -1&0&0&0 \\
      1&-1&0&0
    \end{psmallmatrix}.
  \]
  The mutation $\mu_1$ gives the new cluster tuple $\mathbf{x}'=\left( \dfrac{x_4x_6+x_2x_5}{x_1} , x_2,x_3,\dots,x_6 \right)$ and the following quiver mutation.
  \begin{center}
    \includegraphics[scale=0.8]{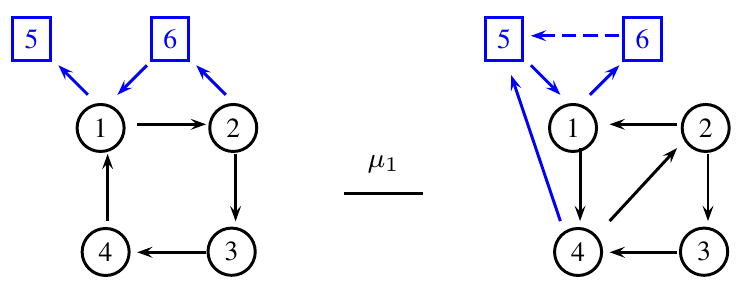}
  \end{center}
  We see that the mutated quiver encodes the information of $\mathbf{y}'= \left( \dfrac{x_5}{x_2}, 1, 1, \dfrac{1}{x_5} \right)$.
  Also notice that we can omit arrows between frozen variables because they will not effect any mutations at non-frozen variables.
\end{ex}

\begin{defn}[Principal coefficient]\label{defn:principalcoef}
  A cluster algebra (or cluster pattern, or Y-pattern) has a \defemph{principal coefficients} at $t_0\in\T_n$ if $\P=\Trop(y_1,\dots,y_n)$ where the initial coefficient tuple is $y_{t_0}=(y_1,\dots,y_n)$.
  We denote $\mathcal{A}_\bullet(B)$ for the cluster algebra with principal coefficients.
\end{defn}

\begin{remark}[Principal coefficients and Frozen variables]\label{rem:PrincCoefFrozen}
  From Remark \ref{rem:GeometricFrozen}, a cluster algebra with principal coefficients of rank $n$ with initial seed $(\mathbf{x},\mathbf{y},B)$ where $\mathbf{x}=(x_1,\dots,x_n)$ and $\mathbf{y}=(y_1,\dots,y_n)$ can be identified with a coefficient-free cluster algebra of rank $2n$ with an initial seed
  $$((x_1,\dots,x_n,y_1,\dots,y_n),\tilde{B})\quad\text{where}\enspace \tilde{B}=\begin{pmatrix} B \\ I_n \end{pmatrix}.$$
  
  The quiver $\mathcal{Q}_{\tilde{B}}$ is obtained from $\mathcal{Q}_B$ by adding one vertex $i'$ and an arrow $i'\rightarrow i$ for any vertex $i$ in the quiver $Q_{B}$.
  The new $\mathcal{Q}_{\tilde{B}}$ is called the \defemph{coframed quiver} associated with $\mathcal{Q}_B.$
\end{remark}

\begin{defn}[The functions $X_{l;t}^{(B)}$ and $F_{l;t}^{(B)}$]\label{def:XF}
  Given an exchange matrix $B,$ we consider the unique (up to isomorphism) cluster pattern $t\mapsto({\mathbf{X}}_t,{\mathbf{Y}}_t,B_t)$ with principal coefficients at $t_0$ and an initial seed $(\mathbf{X},\mathbf{Y},B).$
  For $l\in[1,n]$ and $t\in\T_n,$ we let 
  $$X_{l;t}^{(B)}\in\mathbb{Q}_{sf}(X_1,\dots,X_n;Y_1,\dots,Y_n)$$
  be the $l$-th component of the cluster tuple at $t$, and
  \begin{align*}
    F_{l;t}^{(B)} &:= X_{l;t}^{(B)}(1,\dots,1;Y_1,\dots,Y_n)\in\Z[Y_1,\dots,Y_n].
  \end{align*}
In short, $X_{l;t}^{(B)}$ is a cluster variable in the cluster algebra with principal coefficients.
For a fixed $B$, we often view it as a function on the initial variables $X_i$ and $Y_i$ for $i\in[1,n]$.
The function $F_{l;t}^{(B)}$ is a specialization of $X_{l;t}^{(B)}$ when $X_i=1$ for $i\in[1,n]$.
\end{defn}

The next theorem states that cluster variables of any cluster pattern can be written in terms of the functions $X_{l;t}^{(B)}$ and $F_{l;t}^{(B)}$ with some restriction.

\begin{thm}[Separation formula {\cite[Theorem 3.7]{FZ4}}]\label{thm:sep}
  Let $t\mapsto (\mathbf{x}_t,\mathbf{y}_t,B_t)$ be a cluster pattern over a semifield $\P$ with an initial seed $(\mathbf{x},\mathbf{y},B).$
  Then
  \[
    x_{l;t} = \frac{X_{l;t}^{(B)}(x_1,\dots,x_n;y_1,\dots,y_n)}{F_{l;t}^{(B)} |_\P(y_1,\dots,y_n)}.
  \] 
\end{thm}
The notation $F_{l;t}^{(B)} |_\P(y_1,\dots,y_n)$ means that we compute $F_{l;t}^{(B)}(y_1,\dots,y_n)$ in $\P$ by changing $+$ to $\oplus$.

\begin{ex}\label{ex:sep}
  Consider the cluster algebra with principal coefficients with the same exchange matrix as in Example \ref{ex:frozen}.
  Let $\P=\Trop(Y_1,Y_2,Y_3,Y_4)$, we can write an initial seed as $(\mathbf{X},\mathbf{Y},B)$ where
  \[
    \mathbf{X}=(X_1,X_2,X_3,X_4),\quad \mathbf{Y}=(Y_1,Y_2,Y_3,Y_4).
  \]
  By Remark \ref{rem:PrincCoefFrozen}, we think of $Y_i$'s as frozen variables and get a coefficient-free cluster algebra of rank 8 with the following quiver and exchange matrix, where $Y_i$ is the cluster variable on the vertex $i+4$.
  \[
    \raisebox{-.5\height}{\includegraphics[scale=0.8]{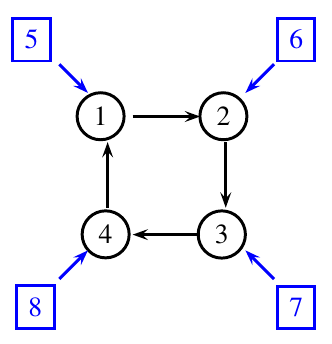}}
    \hspace{50pt}
    \scalebox{0.8}
    {$
    \begin{pmatrix}
      0&1&0&-1 \\
      -1&0&1&0 \\
      0&-1&0&1 \\
      1&0&-1&0 \\
      1&0&0&0 \\
      0&1&0&0 \\
      0&0&1&0 \\
      0&0&0&1
    \end{pmatrix}
    $}
  \]
  Then the mutation $\mu_1$ gives 
  \[
    X_1' = \dfrac{Y_1 X_4 + X_2}{X_1}.
  \]
  Let us try to compute $x'_1$ in Example \ref{ex:frozen} using the separation formula.
  From the formula, we think of $X'_1$ as a function $(X_1,\dots,X_4;Y_1,\dots,Y_4)\mapsto\dfrac{Y_1 X_4 + X_2}{X_1}$.
  Then
  \begin{align*}
    x'_1 &= \dfrac{ X'_1 \left(x_1,x_2,x_3,x_4;\frac{x_6}{x_5},\frac{1}{x_6},1,1 \right) }{ X'_1 \left(1,1,1,1;\frac{x_6}{x_5},\frac{1}{x_6},1,1 \right) \big|_{\Trop(x_5,x_6)}}\\
    &= \dfrac {( \frac{x_6}{x_5} x_4 + x_2 )/x_1 }{ ( \frac{x_6}{x_5} \oplus 1 )/1  } = \dfrac {( \frac{x_6}{x_5} x_4 + x_2 )/x_1 }{{1}/{x_5} } = \dfrac{x_4x_6+x_2x_5}{x_1}.
  \end{align*}
\end{ex}

\begin{defn}[The functions $Y_{l,t}^{(B)}$]
  Given an exchange matrix $B,$ we consider the unique (up to isomorphism) Y-pattern $t\mapsto(\mathbf{Y}_t,B_t)$ having an initial seed $(\mathbf{Y},B)$ in the semifield $\mathbb{Q}_{sf}(Y_1,\dots,Y_n).$
  Let $Y_{l;t}^{(B)}\in\mathbb{Q}_{sf}(Y_1,\dots,Y_n)$ be the $l$-th component in the coefficient tuple at $t$.
\end{defn}
Again, we think of $Y_{l;t}^{(B)}$ as a function on $Y_1,\dots,Y_n$.
The next theorem gives an expression of the function $Y_{j,t}^{(B)}$ in terms of the polynomials $F_{i,t}^{(B)}.$

\begin{thm}[{\cite[Proposition 3.13]{FZ4}}]\label{thm:YasF}
  Given an exchange matrix $B$ and a semifield $\P=\Trop(y_1,\dots,y_n)$, we get
  \[
    Y_{j;t}^{(B)}(y_1,\dots,y_n) = Y_{j;t}^{(B)}\big|_\P(y_1,\dots,y_n) \prod_{i=1}^n \left(F_{i;t}^{(B)}(y_1,\dots,y_n)\right)^{b_{ij}^{(t)}}.
  \]
  where $Y_{j;t}^{(B)}|_\P(y_1,\dots,y_n)$ can be interpreted as a cluster coefficient in the Y-pattern with principal coefficients with an initial coefficient tuple $(y_1,\dots,y_n)$.
\end{thm}

\begin{ex}
  Consider the Y-pattern with the following quiver and the initial coefficient tuple $(y_1,y_2,y_3,y_4)$ in $\P= \mathbb{Q}_{sf}(y_1,y_2,y_3,y_4)$.
  \begin{center}
    \includegraphics[scale=0.8]{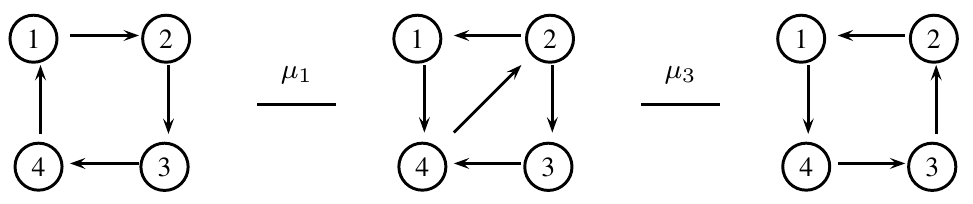}
  \end{center}
  \begin{align}
    \begin{aligned}
      \mathbf{y}= (y_1,\,y_2,\,y_3,\,y_4)\quad &\overset{\mu_1}{\line(1,0){20}}\quad
      \mathbf{y}'= \left( \frac{1}{y_1},\, y_2\frac{y_1}{1+y_1},\, y_3,\, y_4(1+y_1) \right)\\
      &\overset{\mu_3}{\line(1,0){20}} \quad
      \mathbf{y}''= \left( \frac{1}{y_1},\, y_2y_1\frac{1+y_3}{1+y_1},\, \frac{1}{y_3},\, y_4y_3\frac{1+y_1}{1+y_3} \right)
    \end{aligned}
    \label{eq:ypattern}
  \end{align}
  
  Consider a different Y-pattern with the same quiver but with principal coefficients.
  So with the same initial coefficients, we set $\P=\Trop(Y_1,Y_2,Y_3,Y_4)$.
  Using Remark \ref{rem:GeometricFrozen}, we can realize the coefficients as frozen variables.
  \begin{center}
    \includegraphics[scale=0.8]{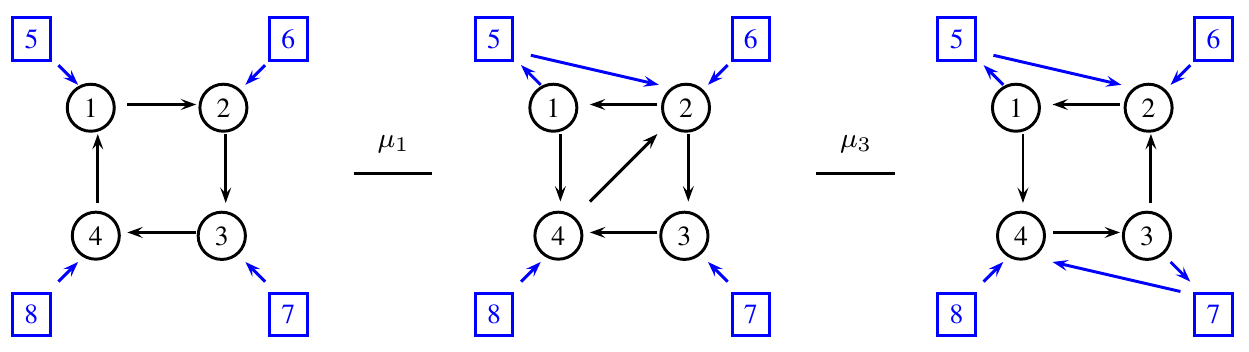}
  \end{center}
  \begin{align*}
    \mathbf{Y}= (Y_1,\,Y_2,\,Y_3,\,Y_4) \quad &\overset{\mu_1}{\line(1,0){20}}\quad
    \mathbf{Y}'= \left( \frac{1}{Y_1},\, Y_2Y_1,\, Y_3,\, Y_4(1+Y_1) \right)\\
    &\overset{\mu_3}{\line(1,0){20}}\quad
    \mathbf{Y}''= \left( \frac{1}{Y_1},\, Y_2 Y_1,\, \frac{1}{Y_3},\, Y_4Y_3 \right).
  \end{align*}
  
  In order to apply Theorem \ref{thm:YasF}, we also need to compute the cluster variables of the cluster algebra with principal coefficients of the same quiver.
  Let $(X_1,X_2,X_3,X_4)$ be the initial cluster tuple, we then get the following.
  \begin{align*}
    \mathbf{X} = (X_1,\,X_2,\,X_3,\,X_4) \quad &\overset{\mu_1}{\line(1,0){20}}\quad
    \mathbf{X}' = \left( \frac{Y_1X_4+X_2}{X_1} ,\,X_2,\,X_3,\,X_4 \right)\\
    &\overset{\mu_3}{\line(1,0){20}} \quad
    \mathbf{X}'' = \left( \frac{Y_1X_4+X_2}{X_1} ,\,X_2,\, \frac{Y_3X_2+X_4}{X_3},\,X_4 \right)
  \end{align*}
  We consider $\mathbf{Y}''_2$, $\mathbf{X}''_1$ and $\mathbf{X}''_3$ as functions on $X_i$'s and $Y_i$'s.
  By Theorem \ref{thm:YasF} we have
  \[
    y''_2 = \mathbf{Y}''_2(y_1,y_2,y_3,y_4) \frac{\mathbf{X}''_3(1,1,1,1;y_1,y_2,y_3,y_4)}{\mathbf{X}''_1(1,1,1,1;y_1,y_2,y_3,y_4)} = y_2y_1\frac{y_3+1}{y_1+1}.
  \]
  This is the same result as we computed directly in \eqref{eq:ypattern}.
\end{ex}


\subsection{Infinite rank cluster algebras}

We define infinite rank cluster algebras in a similar way. The cluster tuple, coefficient tuple and the exchange matrix now are infinite dimensional.
For the mutation to make sense, we only need the condition: For each $j,$ $b_{ij}=0$ for all but finitely many~$i.$
If~$B$ is also skew-symmetric, this condition is equivalent to saying that an infinite quiver $\mathcal{Q}_B$ has only finitely many arrows incident to each of its non-frozen vertex.

For the cluster pattern, although we think of it as an assignment from the infinite tree $\mathbb{T}$, we usually restrict the study to only those seeds obtainable from the initial seed by finitely many mutations.
In the next section, we will review the fact that the T-system can be realized as an infinite-rank coefficient-free cluster pattern, and then define the T-system with principal coefficients in the same way as we have already discussed cluster algebras with principal coefficients in this section.
We pick a specific cluster seed to put principal coefficients.
This choice generates a new recurrence relation which we will call the octahedron recurrence with principal coefficients.


\section{T-systems}\label{sec_T-system}

In this paper, we consider the $A_\infty$ T-system \cite{DFK13}, which is also known as the octahedron recurrence.
It is the infinite-rank version of $A_r$ T-system.
We will refer to $A_\infty$ T-system as just the T-system when there is no ambiguity.

The T-system can be realized as mutation in an infinite-rank coefficient-free cluster algebra of geometric type \cite{DFK09,DFK13}.
Its exchange matrix is skew-symmetric, so we can express it as a quiver, the octahedron quiver.
In this section, we review this connection and define T-systems with generic coefficients by inserting principal coefficient into the relation, as in Definition \ref{defn:principalcoef}, in the initial quiver.
We will also show that it is equivalent to the recurrence relation \eqref{eq:tsyscoef}.


\subsection{T-systems without coefficients}\label{subsec:Tsys}

Let $\Zodd = \{ (i,j,k)\in\Z^3 \mid i+j+k\equiv 1\bmod 2\}.$ The \defemph{T-system}, so called the \defemph{octahedron recurrence}, is a recurrence relation on the set of formal variables $\mathcal{T}=\{ T_{i,j,k} \mid (i,j,k)\in\Zodd\}$ defined by
\begin{align}
  T_{i,j,k-1}T_{i,j,k+1} = T_{i-1,j,k}T_{i+1,j,k} + T_{i,j-1,k}T_{i,j+1,k}.
  \label{eq:Tsys}
\end{align}

A \defemph{stepped surface} is a subset $\{(i,j,\k(i,j)) \mid  i,j\in\Z\}\subset\Zodd$ defined by a function $\k:\Z\times\Z\rightarrow\Z$ satisfying:
\[
  |\k(i,j)-\k(i',j')|=1 \quad \text{when} \quad |i-i'|+|j-j'|=1.
\]
We will also denote this surface by the function $\k$.
The condition $|i-i'|+|j-j'|=1$ is referred to as $(i,j)$ and $(i',j')$ are \defemph{lattice-adjacent}, and $\k(i,j)$ is called the \defemph{height} of $(i,j)$ with respect to $\k$.
There are three important stepped surfaces which we will use throughout the paper.
We define
\begin{align}
  \begin{aligned}
    \fund &:(i,j)\mapsto (i+j \bmod 2)-1,\\
    \proj_{(i',j',k')} &:(i,j)\mapsto k' - | i- i'| -| j- j'|,\\
    \k_{p} &:(i,j)\mapsto\min \left( \k(i,j), \proj_p(i,j) \right),
  \end{aligned}
  \label{eq:fun&topsurface}
\end{align}
and call them the \defemph{fundamental stepped surface}, the \defemph{stepped surface projected from a point} $(i',j',k')$ and the \defemph{adjusted stepped surface} associated with a surface $\k$ and a point $p$, respectively.
See Figure \ref{fig_surface} for examples.

To each $\mathbf{k},$ we can attach an \defemph{initial condition} $X_\mathbf{k}(\mathbf{t}):\{ T_{i,j,\k(i,j)}=t_{i,j} \mid i,j\in\Z\}$ for some formal variables $\mathbf{t}=\{ t_{i,j} \mid i,j\in\Z\}$, to which we refer as \defemph{initial data/values} along the stepped surface $\mathbf{k}$.

It is worth pointing out that for a point $(i_0,j_0,k_0)\in\Zodd$, not every initial data gives a finite solution to $T_{i_0,j_0,k_0}.$
In other words, an expression of $T_{i_0,j_0,k_0}$ in terms of $t_{i,j}$'s may not be finite.
We call an initial data on $\mathbf{k}$ that gives a finite expression for $T_{i_0,j_0,k_0}$ an \defemph{admissible initial data} with respect to $(i_0,j_0,k_0)$.

\begin{ex}
  The fundamental stepped surface is always admissible with respect to any point in $\Zodd.$
  The stepped surface $\proj_{(0,0,m)}$ is not admissible with respect to a point $(0,0,n)$ when $n>m$.
\end{ex}

\begin{figure}
  \begin{center}\includegraphics[scale=0.9]{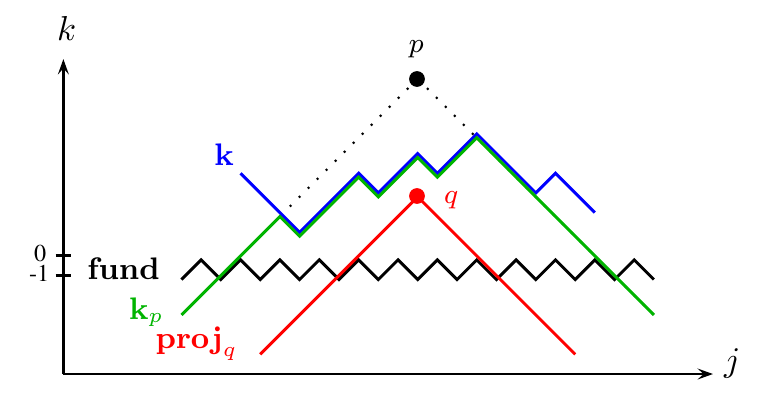}\end{center}
  \caption{
    The surfaces $\fund$, ${\proj_q}$ and $\k_p$ associated with a surface $\k$ in the section $i=i_0$ of the 3-dimensional lattice.}
  \label{fig_surface}
\end{figure}

The T-system can also be interpreted as an infinite-rank coefficient-free cluster algebra \cite{DFK09}.
Using $\mathbb{Z}^2$ as the index set, the initial seed is $(\mathbf{x},B) = \left( (x_{i,j})_{(i,j)\in\Z^2} , (b_{(i',j'),(i,j)}) \right)$ where 
\[
  x_{i,j}=T_{i,j,\fund(i,j)} \quad \text{and} \quad b_{(i',j'),(i,j)} = (-1)^{i+j}\left(\delta_{i',i\pm 1}\delta_{j',j}-\delta_{i',i}\delta_{j',j\pm 1}\right).
\]
The quiver $\mathcal{Q}_B$ associated with $B$ is shown in Figure \ref{fig:octquiver}.
We call it the \defemph{octahedron quiver}.
We embed the vertices of the quiver into the 3-dimensional lattice $\Zodd$ so that they lie on the fundamental stepped surface, i.e. the vertex $(i,j)$ of the octahedron quiver lies at the point $(i,j,\fund(i,j))\in\Zodd$.
The reason for picking this choice is to associate it to the index of the initial cluster variables $T_{i,j,\fund(i,j)}$ at $(i,j)$ for $(i,j)\in\Z^2$.

\begin{figure}
  \begin{center}\includegraphics{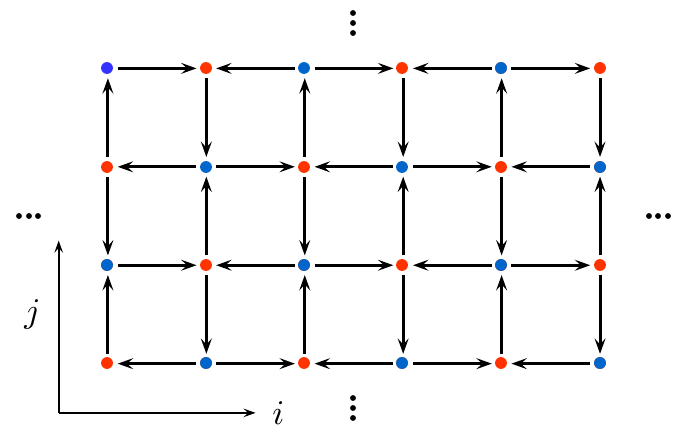}\end{center}
  \caption{
    The octahedron quiver.
    The red dots correspond to indices $(i,j)$ where $i+j$ is even, and the blue to the odd $i+j.$
    The quiver is infinite in both $i$ and $j$ directions.
  }
  \label{fig:octquiver}
\end{figure}

We then allow mutations only on vertices $(i,j,k)$ having the property that there are exactly two incoming and two outgoing arrows incident to $(i,j,k).$
This property is equivalent to saying that all four neighbors of $(i,j,k)$ have the same third coordinate in $\Zodd$, i.e., the four neighbors are all either $(i\pm 1,j\pm 1,k-1)$ or $(i\pm 1,j\pm 1,k+1).$

If the neighbors are $(i\pm 1,j\pm 1,k-1)$, after the mutation at $(i,j,k)$, we move the vertex that used to be at $(i,j,k)$ to $(i,j,k-2)$ and call the new cluster variable obtained by the mutation $T_{i,j,k-2}.$
We call this mutation a \defemph{downward mutation}.
On the other hand, when the neighbors are $(i\pm 1,j\pm 1,k+1)$, $(i,j,k)$ is moved to $(i,j,k+2)$ and the new cluster variable is called $T_{i,j,k+2}$.
We call it an \defemph{upward mutation}.
  
The set of vertices of a quiver $\mathcal{Q}$ obtained from the octahedron quiver by allowed mutations forms a stepped surface, denoted by $\mathbf{k}_\mathcal{Q}$.
On the other hand, we can create a quiver from a stepped surface by reading the arrangement of the quiver arrows from Table \ref{tab:quiver&k}, and call this quiver $\mathcal{Q}_\mathbf{k}.$
We notice that the quiver mutation at $(i,j)$ corresponds to moving $(i,j,\k(i,j))$ to $(i,j,\k(i,j)\pm 2)$, depending on the height of its neighbors as discussed above.
We say that $\k'$ is obtained from $\k$ by a \defemph{mutation} at $(i,j)$ if $\Q_{\k'}=\mu_{(i,j)}(\Q_{\k})$.


\subsection{T-systems with principal coefficients}

We define the \defemph{T-systems with principal coefficients} from the cluster algebra setting.
Instead of the coefficient-free cluster algebra with the octahedron quiver, we consider the cluster algebra with principal coefficients (Definition \ref{defn:principalcoef}) on the same quiver, where the initial coefficient at $(i,j)$ is $c_{i,j}.$
Due to Remark \ref{rem:PrincCoefFrozen}, it is the same as the coefficient-free cluster algebras on the coframed octahedron quiver, where the variables $c_{i,j}$ on the added vertices are frozen, see Figure \ref{fig_framedocta}.
We show that the cluster variables satisfy the recurrence relation \eqref{eq:tsyscoef} on $\{ T_{i,j,k} \mid (i,j,k)\in\Zodd \}$ with an extra set of coefficients $\{c_{i,j} \mid (i,j)\in\Z^2\}.$
We will use this recursion as an alternative definition of the T-system with principal coefficients.
  
\begin{figure}
  \begin{center}\includegraphics[scale=0.8]{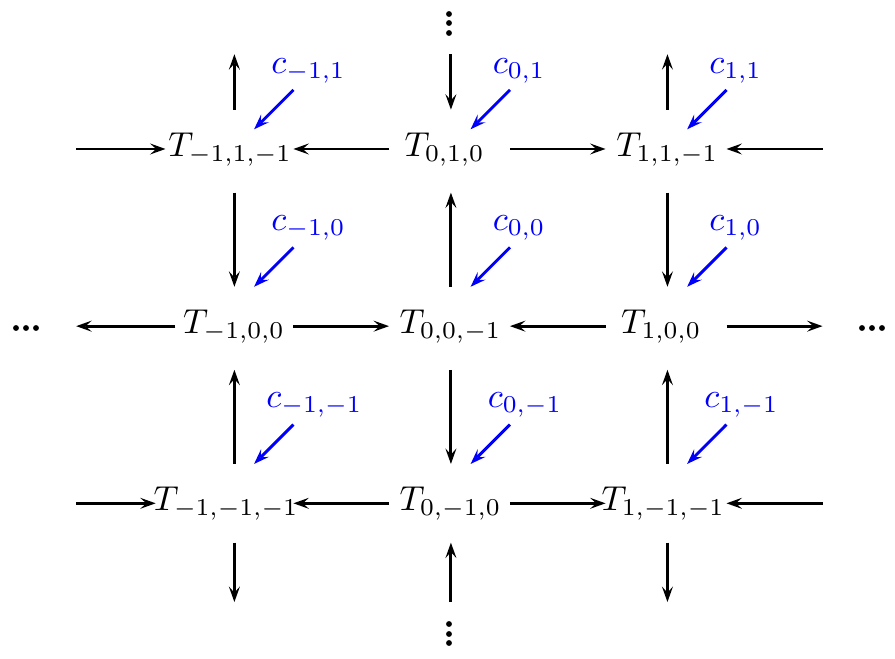}\end{center}
  \caption{A portion of the infinite framed octahedron quiver and its cluster variables including frozen variables $c_{i,j}$.}
  \label{fig_framedocta}
\end{figure}

\begin{thm}\label{thm:Tsyscoef}
  Let $\{ T_{i,j,k} \mid (i,j,k)\in\Zodd \}$ be the set of cluster variables obtained from the T-system with principal coefficients.
  Then
  \begin{align}
    T_{i,j,k-1}T_{i,j,k+1} = J_{i,j,k} T_{i-1,j,k}T_{i+1,j,k} + I_{i,j,k}T_{i,j-1,k}T_{i,j+1,k} 
    \label{eq:tsyscoef}
  \end{align}
  where
  \begin{align}
    I_{i,j,k} =
    \begin{cases}
      \prod_{a=k+1}^{-(k+1)} c_{i+a,j}, & k<0,\\
      1, & k\geq 0,
    \end{cases} \quad \text{and}\quad
    J_{i,j,k} =
    \begin{cases}
      1, &k<0,\\
      \prod_{a=-k}^k c_{i,j+a}, &k\geq 0.
    \end{cases}
    \label{eq:IJ}
  \end{align}
\end{thm}

We call the relation \eqref{eq:tsyscoef} the \defemph{octahedron recurrence with principal coefficients}.
The pictorial representation of $I$ and $J$ are shown in Figure \ref{fig_IJ}.

In order to prove Theorem \ref{thm:Tsyscoef}, we first compute the coefficients at the vertices in any quiver obtained by the octahedron quiver.
We note that unlike cluster variables, a coefficient at $(i,j)$ on a stepped surface $\k$ depends on the height of $(i,j)$ and its neighbors $(i\pm 1,j\pm 1)$.

\begin{prop}\label{prop:coefatvertex}
  Consider the T-system with principal coefficients.
  Let $\k$ be a stepped surface obtained from $\fund$ by a finite number of allowed mutations.
  Let $y_{(i,j),\k}$ be the coefficient at the vertex $(i,j)$ in $\Q_\k$.
  Then
  \begin{align}
    y_{(i,j),\k} = \frac{ I_{i,j,k-1} J_{i,j-1,k}^{[\epsilon_1]_+} J_{i,j+1,k}^{[\epsilon_2]_+} }{ J_{i,j,k-1} I_{i-1,j,k}^{[\epsilon_3]_+} I_{i+1,j,k}^{[\epsilon_4]_+}} = \frac{ J_{i,j,k+1}I_{i-1,j,k}^{[-\epsilon_3]_+} I_{i+1,j,k}^{[-\epsilon_4]_+}}{ I_{i,j,k+1}J_{i,j-1,k}^{[-\epsilon_1]_+} J_{i,j+1,k}^{[-\epsilon_2]_+}},
    \label{eq:coefatvertex}
  \end{align}
  when $\k(i,j)=k$, $\k(i,j-1)=k+\epsilon_1$, $\k(i,j+1)=k+\epsilon_2$, $\k(i-1,j)=k+\epsilon_3$ and $\k(i+1,j)=k+\epsilon_4$ where $ \epsilon_\ell\in\{-1,1\}$, as described as follows:
  \[
    \begin{matrix}
      & (i,j+1,k+\epsilon_2) &\\\\
      (i-1,j,k+\epsilon_3) & (i,j,k) & (i+1,j,k+\epsilon_4)\\\\
      & (i,j-1,k+\epsilon_1) &
    \end{matrix}
  \]
\end{prop}

\begin{ex}
  Consider a stepped surface $\k$ having height as the following.
  \[
    \begin{matrix}
      & (i,j+1,k+1) & \\\\
      (i-1,j,k+1) & (i,j,k) & (i+1,j,k-1) \\\\
      & (i,j-1,k-1) & 
    \end{matrix}
  \]
  The coefficient at the vertex $(i,j)$, $y_{(i,j),\k}$, computed by Proposition \ref{prop:coefatvertex} is
  \[
    y_{(i,j),\k}=\frac{ I_{i,j,k-1} J_{i,j+1,k} }{ J_{i,j,k-1} I_{i-1,j,k}} = \frac{ J_{i,j,k+1}I_{i+1,j,k}}{ I_{i,j,k+1}J_{i,j-1,k}}.
  \]
\end{ex}

\begin{proof}[Proof of Proposition \ref{prop:coefatvertex}]
  We fist show the second equality in \eqref{eq:coefatvertex}.
  Notice that $[\epsilon_i]_+ +[-\epsilon_i]_+ = 1$ for all $i$.
  So we only need to show
  \[
    \frac{ I_{i,j,k-1} J_{i,j-1,k} J_{i,j+1,k} }{ J_{i,j,k-1} I_{i-1,j,k} I_{i+1,j,k}} = \frac{ J_{i,j,k+1}}{ I_{i,j,k+1}},
  \]
  which can be easily derived from the definition of $I$ and $J$ in \eqref{eq:IJ}.
  
  We will then prove the proposition by induction on the number of mutations from the fundamental stepped surface.
  On $\fund$, the vertices are in the forms $(i,j,-1)$ or $(i,j,0)$ depending on the parity of $i+j$, and $y_{(i,j),\fund}=c_{i,j}$ for all $(i,j)$.
  When $i+j\equiv 0\bmod 2$, $\fund(i,j)=-1$ and the neighbors of $(i,j,-1)$ are $(i\pm 1,j\pm 1,0)$.
  So $\epsilon_\ell=1$ for all $\ell$ at $(i,j)$.
  We also have
  \[
    y_{(i,j),\k} = c_{i,j} = \frac{ J_{i,j,0}}{ I_{i,j,0}}.
  \]
  When $i+j\equiv 1\bmod 2$, $\fund(i,j)=0$ and the neighbors of $(i,j,0)$ are $(i\pm 1,j\pm 1,-1)$.
  So $\epsilon_\ell= -1$ for all $\ell$.
  We have
  \[
    y_{(i,j),\k} = c_{i,j} = \frac{ I_{i,j,-1}}{ J_{i,j,-1}}.
  \]
  Hence the proposition holds for the fundamental stepped surface.
  
  Next we assume that the proposition holds for a stepped surface $\k$.
  Consider a stepped surface $\k'$ obtained from $\k$ by a mutation at $(i,j)$.
  Then $\k=\k'$ on every point except at $(i,j)$.
  Also $y_{(a,b),\k}=y_{(a,b),\k'}$ for all $(a,b)$ but at most five points: $(i,j)$, $(i\pm 1, j\pm 1,)$.
  So we only need to consider the coefficients at these five points.

  Let us assume that $\k(i,j)=k$.
  Since $\k$ is mutable at $(i,j)$, we have two cases: $\k(i\pm 1,j\pm 1)=k-1$ or $\k(i\pm 1,j\pm 1) =k+ 1$, as shown in the following pictures.
  \[
    \scalebox{0.7}{
      $\xymatrix@=1em{
        & (i,j+1,k-1)\ar[d] & \\
        (i-1,j,k-1) & \ar[l](i,j,k)\ar[r] & (i+1,j,k-1) \\
        & (i,j-1,k-1)\ar[u] &.
      }
      \hspace{2em}
      \xymatrix@=1em{
        & (i,j+1,k+1) & \\
        (i-1,j,k+1)\ar[r] & (i,j,k)\ar[u]\ar[d] & \ar[l](i+1,j,k+1)\\
        & (i,j-1,k+1) &.
      }$
    }
  \]
  
  \begin{description}
  \item[Case 1]
    We know that $y_{(i,j),\k} = I_{i,j,k-1}/J_{i,j,k-1}$ by the induction hypothesis.
    After the mutation at $(i,j)$, the point $(i,j,k)$ becomes $(i,j,k-2)$.
    So on $\k'$, $\epsilon_i = 1$ for all $i$.
    We also get
    \[
      y_{(i,j),\k'} = \left(\frac{ I_{i,j,k-1}}{ J_{i,j,k-1}}\right)^{-1}=\frac{ J_{i,j,k'+1}}{ I_{i,j,k'+1}} = \frac{ J_{i,j,k'+1}I_{i-1,j,k'}^{[-1]_+} I_{i+1,j,k'}^{[-1]_+}}{ I_{i,j,k'+1}J_{i,j-1,k'}^{[-1]_+} J_{i,j+1,k'}^{[-1]_+}}
    \]
    where $k'=k-2$.
    Hence the expression of $y_{(i,j),\k'}$ agrees with the proposition.
    
    At $(i,j+1,k-1)$, the induction hypothesis gives
    \begin{align*}
      y_{(i,j+1),\k} &= \frac{ I_{i,j+1,k-2} J_{i,j,k-1}^{[1]_+} J_{i,j+2,k-1}^{[\epsilon_2]_+} }{ J_{i,j+1,k-2} I_{i-1,j+1,k-1}^{[\epsilon_3]_+} I_{i+1,j+1,k-1}^{[\epsilon_4]_+}}\\
      &= \frac{ I_{i,j+1,k-2} J_{i,j,k-1} J_{i,j+2,k-1}^{[\epsilon_2]_+} }{ J_{i,j+1,k-2} I_{i-1,j+1,k-1}^{[\epsilon_3]_+} I_{i+1,j+1,k-1}^{[\epsilon_4]_+}}.
    \end{align*}
    We know that $\epsilon_1=1$ since $\k(i,j)= k =\k(i,j+1)+1$.
    Then the mutation at $(i,j)$ gives
    \begin{align*}
      y_{(i,j+1),\k'} &= y_{(i,j+1),\k} (1\oplus y_{(i,j),\k})\\
      &= y_{(i,j+1),\k}\frac{1}{J_{i,j,k-1}}\\
      &= \frac{ I_{i,j+1,k-2} J_{i,j,k-1}^{[-1]_+} J_{i,j+2,k-1}^{[\epsilon_2]_+} }{ J_{i,j+1,k-2} I_{i-1,j+1,k-1}^{[\epsilon_3]_+} I_{i+1,j+1,k-1}^{[\epsilon_4]_+}},
    \end{align*}
    which agrees to the proposition.
    By the similar argument, we can show that all four of the $y_{(i\pm 1,j\pm 1),\k'}$ agree to the proposition.
  \item[Case 2]
  We know that $y_{(i,j),\k} = J_{i,j,k+1}/I_{i,j,k+1}$ by the induction hypothesis.
  After the mutation at $(i,j)$, the point $(i,j,k)$ becomes $(i,j,k+2)$.
  So on $\k'$, $\epsilon_i = -1$ for all $i$.
  We also get
    \[
      y_{(i,j),\k'} = \left(\frac{ J_{i,j,k+1}}{ I_{i,j,k+1}}\right)^{-1}=\frac{ I_{i,j,k'-1}}{ J_{i,j,k'-1}} = \frac{ I_{i,j,k'-1}J_{i,j-1,k'}^{[-1]_+} J_{i,j+1,k'}^{[-1]_+}}{ J_{i,j,k'-1}I_{i-1,j,k'}^{[-1]_+} I_{i+1,j,k'}^{[-1]_+}}
    \]
    when $k'=k+2$.
    Hence the expression of $y_{(i,j),\k'}$ agrees with the proposition.
    
    At $(i,j+1,k-1)$, the induction hypothesis gives
    \begin{align*}
      y_{(i,j+1),\k} &= \frac{ I_{i,j+1,k-2} J_{i,j,k-1}^{[-1]_+} J_{i,j+2,k-1}^{[\epsilon_2]_+} }{ J_{i,j+1,k-2} I_{i-1,j+1,k-1}^{[\epsilon_3]_+} I_{i+1,j+1,k-1}^{[\epsilon_4]_+}}\\
      &= \frac{ I_{i,j+1,k-2} J_{i,j,k-1} J_{i,j+2,k-1}^{[\epsilon_2]_+} }{ J_{i,j+1,k-2} I_{i-1,j+1,k-1}^{[\epsilon_3]_+} I_{i+1,j+1,k-1}^{[\epsilon_4]_+}}.
    \end{align*}
    We know that $\epsilon_1=-1$ since $\k(i,j)= k =\k(i,j+1)-1$.
    Then the mutation at $(i,j)$ gives
    \begin{align*}
      y_{(i,j+1),\k'} &= y_{(i,j+1),\k} \frac{y_{(i,j),\k}}{(1\oplus y_{(i,j),\k})}\\
      &= y_{(i,j+1),\k}J_{i,j,k-1}\\
      &= \frac{ I_{i,j+1,k-2} J_{i,j,k-1}^{[1]_+} J_{i,j+2,k-1}^{[\epsilon_2]_+} }{ J_{i,j+1,k-2} I_{i-1,j+1,k-1}^{[\epsilon_3]_+} I_{i+1,j+1,k-1}^{[\epsilon_4]_+}},
    \end{align*}
    which agrees to the proposition.
    By the similar argument, we can show that all four of the $y_{(i\pm 1,j\pm 1),\k'}$ agree to the proposition.
  \end{description}
  By both cases, we proved the proposition.
\end{proof}

\begin{proof}[Proof of Theorem \ref{thm:Tsyscoef}]
  To show \eqref{eq:tsyscoef}, it is enough to show that it is the mutation rule at $(i,j,k-1)$ on $\k$ such that $\k(i\pm 1,j\pm 1)=k$ and $\k(i,j)=k-1$.
  The quiver at $(i,j)$ will look like the following:
  \[
    \scalebox{0.8}{$
      \xymatrix@=1em{
        & (i,j+1,k) & \\
        (i-1,j,k)\ar[r] & (i,j,k-1)\ar[u]\ar[d] & \ar[l](i+1,j,k) \\
        & (i,j-1,k) &
      }
    $}
  \]
  So it is equivalent to show that $y_{(i,j),\k} = J_{i,j,k}/I_{i,j,k}$, which comes from Proposition \ref{prop:coefatvertex}.
\end{proof}

\begin{figure}
  \begin{center}\includegraphics{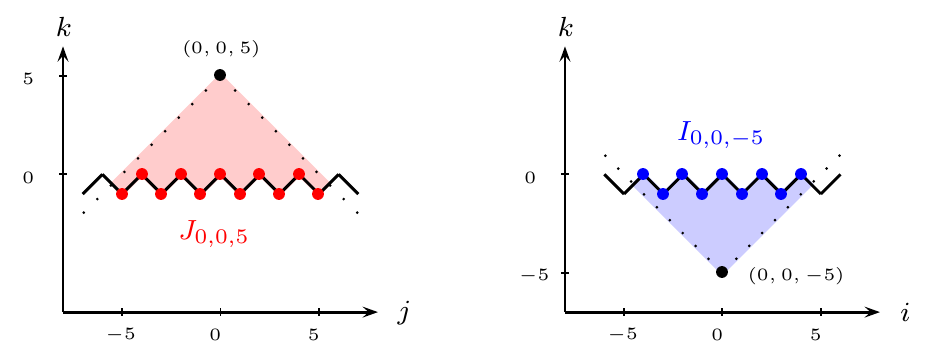}\end{center}
  \caption{
    $J_{0,0,5} = c_{0,-5}c_{0,-4}\dots c_{0,5}$ can be realized as a shadow shaded down from $(0,0,5)$ as depicted by the red dots in the section $i=0$.
    $I_{0,0,-5} = c_{-4,0}c_{-3,0}\dots c_{4,0}$ can be realized as a shadow shaded up from $(0,0,-5)$ as depicted by the blue dots in the section $j=0$.
    }
  \label{fig_IJ}
\end{figure}

Due to Theorem \ref{thm:Tsyscoef}, we view the T-system with principal coefficients as a recurrence relation on $T_{i,j,k}$, $(i,j,k)\in\Zodd$ with extra coefficient variables $c_{i,j}$, $(i,j)\in\Z^2.$
Fixing a point $p=(i_0,j_0,k_0)\in\Zodd$ and an admissible initial data on a stepped surface $\mathbf{k}$, Theorem \ref{thm:Laurent} guarantees that the expression of $T_{i_0,j_0,k_0}$ is a Laurent polynomial in the initial data $\{ t_{i,j}=T_{i,j,\k(i,j)}\mid (i,j)\in\Z^2 \}$ and coefficients $\{c_{i,j}\mid (i,j)\in\Z^2 \}$.
The goal is to give combinatorial interpretation for this expression.
In this paper, we study the case when $p$ is above the $\k$ and $\k$ is above $\fund$, i.e.
\begin{align}\label{eq:surfacecond}
  k_0 \geq \k(i_0,j_0)\quad\text{and}\quad\k(i,j) \geq \fund(i,j) = (i+j\bmod 2)-1.
\end{align}
In this case, we have explicit combinatorial solutions in terms of perfect matchings in Sections \ref{sec_Dimer} and \ref{sec_edge-weight}, non-intersecting paths in Section \ref{sec_path} and networks in Section \ref{sec_network}.


\section{Perfect-matching solution}\label{sec_Dimer}

The goal of this section is to give an expression of $T_{i_0,j_0,k_0}$ in terms of a partition function of weighted perfect matchings of a certain graph.
There are previous works \cite{Speyer,MS10,JMZ13} on expressing cluster variables by using perfect matchings of certain weighted graphs.
Regarding only cluster variables, the weight studied in \cite{Speyer} coincides with the ``face-weight'' in Definition \ref{defn:faceweight}, while the weight in \cite{MS10,JMZ13} coincides with the ``edge-weight'' in Definition \ref{defn:edgeweight}.


\subsection{Graphs from stepped surfaces}\label{subsec_graph}

We fix a point $p=(i_0,j_0,k_0)\in\Zodd$, an admissible stepped surface $\k$ and an initial data $X_\mathbf{k}(\mathbf{t}):\{ T_{i,j,\k(i,j)}=t_{i,j} \mid i,j\in\Z\}$ on $\mathbf{k}$.
Also assume that $k_0 \geq \k(i_0,j_0)$ and $\k \geq \fund$, i.e. $\k(i,j)\geq\fund(i,j)$ for all $(i,j)\in\Z^2$.

From the stepped surface $\mathbf{k}$, we follow the construction in \cite{Speyer} and define, using Table \ref{tab:quiver&k}, an infinite bipartite graph $G_\k$ associated with $\k$.
This graph can also be realized as the dual of the quiver $\Q_\k$ associated with $\k$ with vertex bi-coloring, see the end of Section \ref{subsec:Tsys}.
Faces of $\mathcal{Q}_{\k}$ become vertices of $G_\k.$
Since all faces of $\Q_k$ are always oriented, we color a vertex of the graph in white if the arrows around its corresponding face of the quiver are oriented counter-clockwise and black if they are oriented clockwise.
Vertices of $\mathcal{Q}_\k$ become faces of $G_\mathbf{k}$.
Since the vertices of the quiver are indexed by $\Z^2$, we will use $(i,j)\in\Z^2$ to represent a face of the graph.
Arrows of $\Q_\k$ gives edges of $G_\k$.
There are three types of edges in the graphs: horizontal, vertical and diagonal, which came from vertical, horizontal and diagonal arrows of the quiver, respectively.
See Figure \ref{fig:3edges} for an example.

\begin{table}
  \begin{center}{\begin{tabular}{|c|c|c|}
\hline
  \quad
  $\begin{matrix}
  \k(D)	&	\k(C) \\
  \k(A) & \k(B)
  \end{matrix}$
  \quad
&
  \quad A part in $\mathcal{Q}_\k$ \quad
& 
  \quad A part in $G_\mathbf{k}$ \quad
\\[2ex]
\hline 
  $\begin{matrix}
  k	&	k+1 \\
  k+1   &k
  \end{matrix}$
&
  $\begin{matrix}
  D		& \leftarrow &	C \\
  \downarrow	& 						&	\uparrow		\\
  A			& \rightarrow 	&	B
  \end{matrix}$
&
  \raisebox{-.5\height}{\includegraphics[scale=0.8]{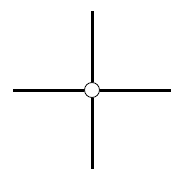}}
\\
\hline
  $\begin{matrix}
  k+1 & k\\
  k & k+1
  \end{matrix}$
&
  $\begin{matrix}
  D		&	\rightarrow	&	C \\
  \uparrow&							& \downarrow\\
  A			&	\leftarrow	&	B
  \end{matrix}$
&
  \raisebox{-.5\height}{\includegraphics[scale=0.8]{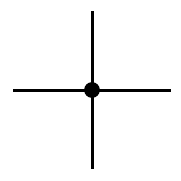}}
\\
\hline
  $\begin{matrix}
  k-1 & k \\ k &k+1 
  \end{matrix}$
&
  $\begin{matrix}
  D		& \leftarrow &	C \\
  \downarrow	& \nearrow		&	\downarrow		\\
  A			& \leftarrow 	&	B
  \end{matrix}$
&
  \raisebox{-.5\height}{\includegraphics[scale=0.8]{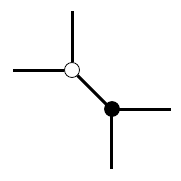}}
\\
\hline
	$\begin{matrix}
	k+1 & k \\ k & k-1
	\end{matrix}$
&
	$\begin{matrix}
	D		& \rightarrow &	C \\
	\uparrow	& \swarrow		&	\uparrow		\\
	A			& \rightarrow 	&	B
	\end{matrix}$
&
	\raisebox{-.5\height}{\includegraphics[scale=0.8]{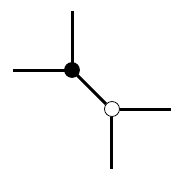}}
\\
\hline
	$\begin{matrix}
	k & k+1 \\ k-1 & k
	\end{matrix}$
&
	$\begin{matrix}
	D		& \leftarrow &	C \\
	\uparrow	& \searrow		&	\uparrow		\\
	A			& \leftarrow 	&	B
	\end{matrix}$
&
	\raisebox{-.5\height}{\includegraphics[scale=0.8]{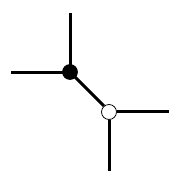}}
\\
\hline
	$\begin{matrix}
	k & k-1 \\ k+1 & k
	\end{matrix}$
&
	$\begin{matrix}
	D		& \rightarrow &	C \\
	\downarrow	& \nwarrow		&	\downarrow		\\
	A			& \rightarrow 	&	B
	\end{matrix}$
&
	\raisebox{-.5\height}{\includegraphics[scale=0.8]{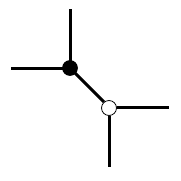}}
\\
\hline
\end{tabular}}\end{center}\bigskip
  \caption{All six local pictures of $\mathcal{Q}_\k$ and $G_\k$ for four points $A=(i,j),$ $B=(i+1,j),$ $C=(i+1,j+1),$ and $D=(i,j+1)$ on a stepped surface $\k$.}
  \label{tab:quiver&k}
\end{table}

\begin{figure}
  \centering
  \begin{subfigure}[b]{0.3\textwidth}
    \centering 
    $\begin{matrix} 0 && 1 && 2 &&3 \\\\ 1 && 2 && 1 && 2 \\\\ 2& & 1 && 0 && 1 \\\\ 1 && 2 && 1 && 2 \end{matrix}$
    \caption{A portion of $\mathbf{k}$}
  \end{subfigure}
  \begin{subfigure}[b]{0.3\textwidth}
    \centering$\xymatrix@=1.45em{
    \bullet\ar[d] & \bullet\ar[l]\ar[d] & \bullet\ar[l]\ar[rd] & \bullet\ar[l] \\
    \bullet\ar[ur]\ar[d] &\bullet\ar[l]\ar[r] & \bullet\ar[u]\ar[ld]\ar[rd] & \bullet\ar[l]\ar[u] \\
    \bullet\ar[r] & \bullet\ar[u]\ar[r]\ar[d] & \bullet\ar[u]\ar[d] & \bullet\ar[u]\ar[l]\ar[d] \\
    \bullet\ar[u] & \bullet\ar[l]\ar[r] & \bullet\ar[ul]\ar[ur] & \bullet\ar[l]
    }$
    \caption{A portion of $\mathcal{Q}_\k$}
  \end{subfigure}
  \begin{subfigure}[b]{0.3\textwidth}
    \centering\includegraphics[scale=1.2]{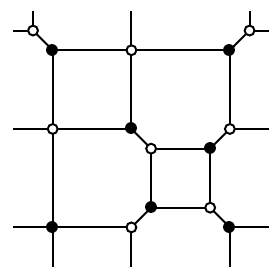}
    \caption{A portion of $G_\mathbf{k}$}
  \end{subfigure}
  \caption{An example of $\mathbf{k}$ and its corresponding $\mathcal{Q}_k$ and $G_\mathbf{k}.$}
  \label{fig:3edges}
\end{figure}

If $\k'$ is obtained from $\k$ by a mutation at $(i,j)$, then we can see from Table \ref{tab:quiver&k} that the face $(i,j)$ in $G_k$ must be a square.
In addition, $G_{\k'}$ can be obtained \cite{Ciucu,Speyer} from $G_\k$ by the following steps. 
\begin{enumerate}
  \item Apply \defemph{urban renewal} at the face $(i,j)$, see Figure \ref{fig:urban}.
  \item \defemph{Collapse} any degree-2 vertices created by the previous step, see Figure \ref{fig:deg2vertex}.
\end{enumerate}

\begin{figure}
  \begin{center}\includegraphics{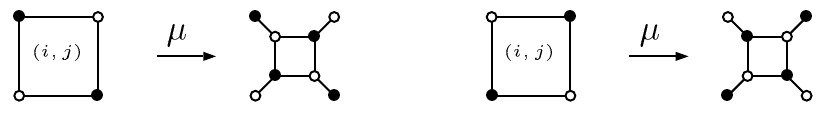}\end{center}
  \caption{The urban renewal at the face $(i,j)$.}
  \label{fig:urban}
\end{figure}

\begin{figure}
  \begin{center}\includegraphics[scale=1.3]{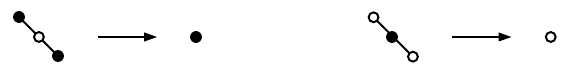}\end{center}
  \caption{A degree-2 vertex and its two adjacent vertices collapse into one vertex.}
  \label{fig:deg2vertex}
\end{figure}

We use the notations $F(G)$, $V(G)$ and $E(G)$ for the set of faces, vertices and edges of a graph $G$, respectively.
We then define two subsets $\mathring{F}=\mathring{F}(p,G_\k)$ and $\partial F = \partial F(p,G_\k)$ of $F(G_\mathbf{k})=\Z^2$ depending on $p$ and $\k$ as follows.
\begin{align}
\begin{aligned}
  \mathring{F}&=\left\{ (i,j)\in \Z^2 ~\big|~ |i-i_0|+|j-j_0|<k_0-\k(i,j) \right\},\\
  \partial F &= \left\{ (i',j')\in\Z^2 \setminus \mathring{F} ~\Big|~ |i'-i|+|j'-j|=1 \text{ for some }(i,j)\in \mathring{F} \right\}.
\end{aligned}\label{eq:opencloseface}
\end{align}
We also assume that $\partial F=\{(i_0,j_0)\}$ when $k_0=\k(i_0,j_0)$.
The set $\mathring{F}$ can be illustrated as the set of points inside (excluding boundary) the shadow projecting from $p$ onto $\k$, while $\partial F$ is the boundary of the projection.
The following example shows elements of $\mathring{F}$ in blue and elements of $\partial F$ in red when $p=(0,0,3)$ and $\k:(i,j)\mapsto |i+j|-1$.
\begin{center}
  \includegraphics[scale=0.65]{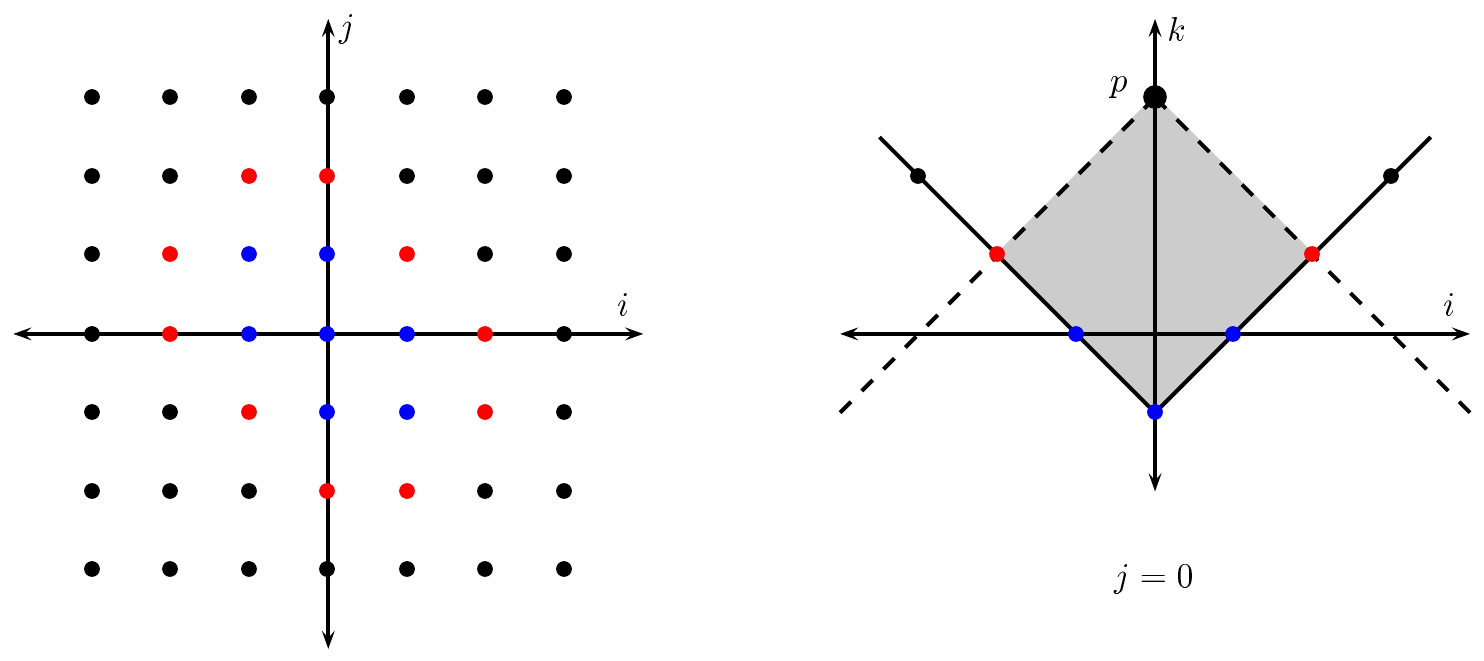}
\end{center}
The picture on the left shows the faces on $(i,j)$-plane, discarding the $k$-direction.
The picture on the right shows the projection in the section $j=0$ of the whole 3-dimensional space.

We will see later from the solution to the T-system (Theorem \ref{thm:main}) that the expression of $T_{i_0,j_0,k_0}$ depends only on $t_{i,j}$'s where $(i,j)\in \mathring{F}\cup\partial F$.
Due to this reason, we will work on a finite subgraph $G_{p,\mathbf{k}}$ of $G_\mathbf{k}$ generated by the faces in $\mathring{F},$ while considering faces in $\partial F$ as ``open faces'' as in the following definition.

\begin{defn}[Graph with open faces {\cite[Section 2.2]{Speyer}}]
  The \defemph{graph with open faces} associated with $p$ and $\k$ is defined to be a pair $\left(G,\partial F(G)\right)$ where $G:=G_{p,\mathbf{k}}$ is a finite subgraph of $G_\mathbf{k}$ generated by the faces in $\mathring{F}$, and $\partial F(G):= \partial F$ is the set of \defemph{open faces}.
\end{defn}

Since we can always determine $\partial F(G)$ from $F(G),$ we will omit $\partial F(G)$ by writing just $G$ instead of $(G,\partial F(G)).$
The faces in $F(G)=\mathring{F}$ are called \defemph{closed faces}, while the faces in $\partial F(G)=\partial F$ are called \defemph{open faces}.

Later in the paper, some other solutions to the T-systems with principal coefficients will look nicer if written in terms of the ``closure" of $G$ instead of $G$.
This will be a graph with no open faces.

\begin{defn}[The closure $\overline{G}$ of $G$]\label{defn:Gclosure}
  For a point $p$ and a surface $\k$, let $\k_p$ be the adjusted stepped surface associated with $\k$ and $p$ defined in \eqref{eq:fun&topsurface} and $G_\infty:=G_{\k_p}$ be the graph associated to $\k_p$.
  We define the closure $\overline{G}$ of $G$ to be the finite subgraph of $G_\infty$ generated by $\mathring{F}\cup\partial F$, and we think of it as a graph with no open face.
\end{defn}

We note that $\k(i,j)=\k_p(i,j)$ for all $(i,j)\in F(G)\cup\partial F(G).$
So the graphs with open faces $G_{p,\k}$ and $G_{p,\k_p}$ are exactly the same except for the shape of the open faces.
Due to the following proposition, we can obtain $\overline{G}$ directly from $G$ by closing all the open faces of $G$ in a certain way.

\begin{prop}\label{prop:modifiedopenface}
  All 16 types of the faces of $\G$ in $F(\overline{G})\setminus F(G) = \partial F$ are shown in Figure \ref{fig_modifiedopenface} where dotted lines indicate edges in $E(\overline{G})\setminus E(G).$
\end{prop}
\begin{proof}
  At each open face of $G$, we consider the height of its neighboring faces.
  The shape of the faces are obtained from Table \ref{tab:quiver&k}.
  The proposition then easily follows.
\end{proof}

\begin{ex}\label{ex:gbar}
  Let $\k(i,j)=|i+j|-1$ and $p=(0,0,3)$.
  Then the infinite graphs $G_\mathbf{k}$, $G_\infty=G_{\k_p}$ and the finite graphs $G=G_{p,\k}$, $\overline{G}$ are shown in Figure \ref{fig_Gbar}.
\end{ex}

\begin{figure}
  \begin{center}\includegraphics[scale=0.8]{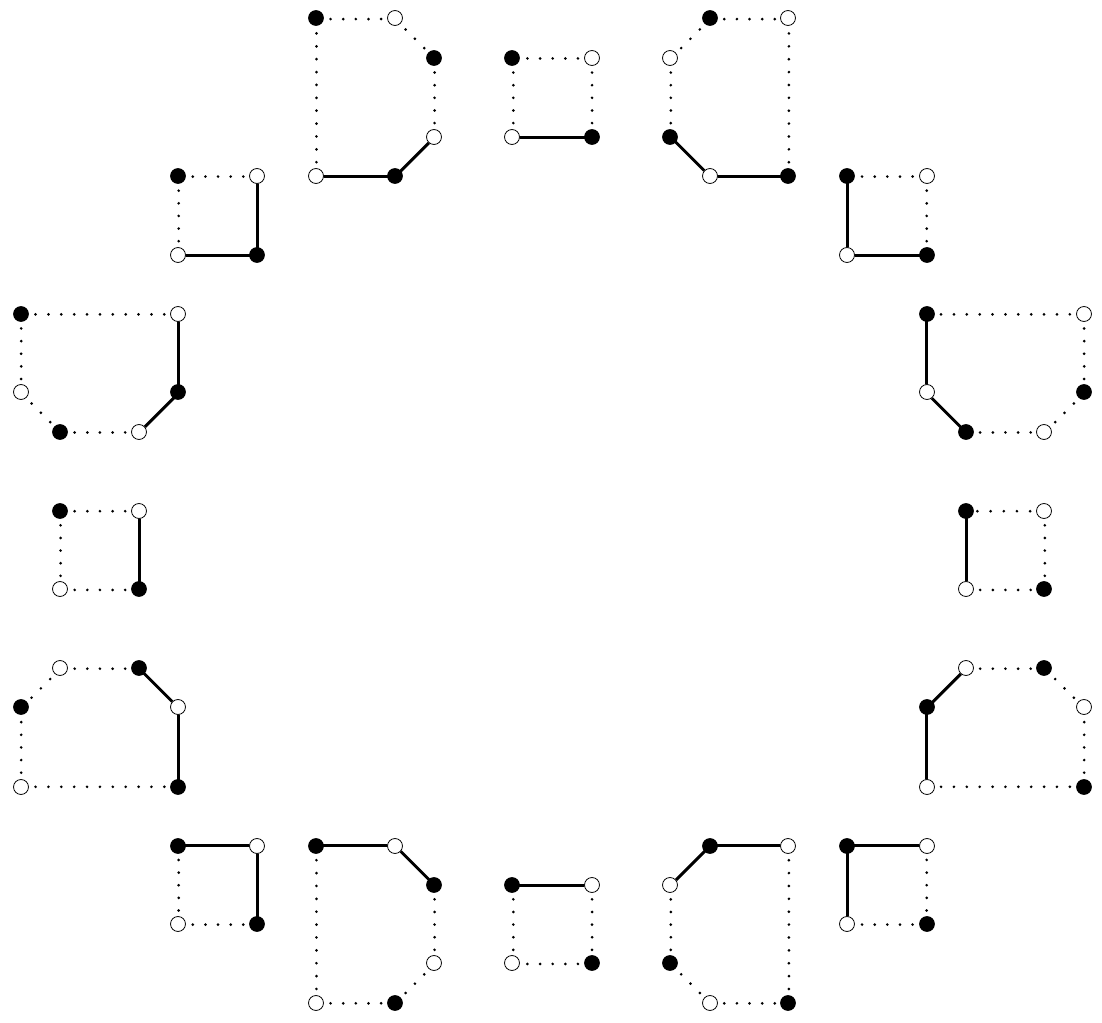}\end{center}
  \caption{Faces in $F(\overline{G})\setminus F(G) = \partial F$ of $\overline{G}.$}
  \label{fig_modifiedopenface}
\end{figure}

\begin{figure}
  \begin{center}\includegraphics[scale=0.5]{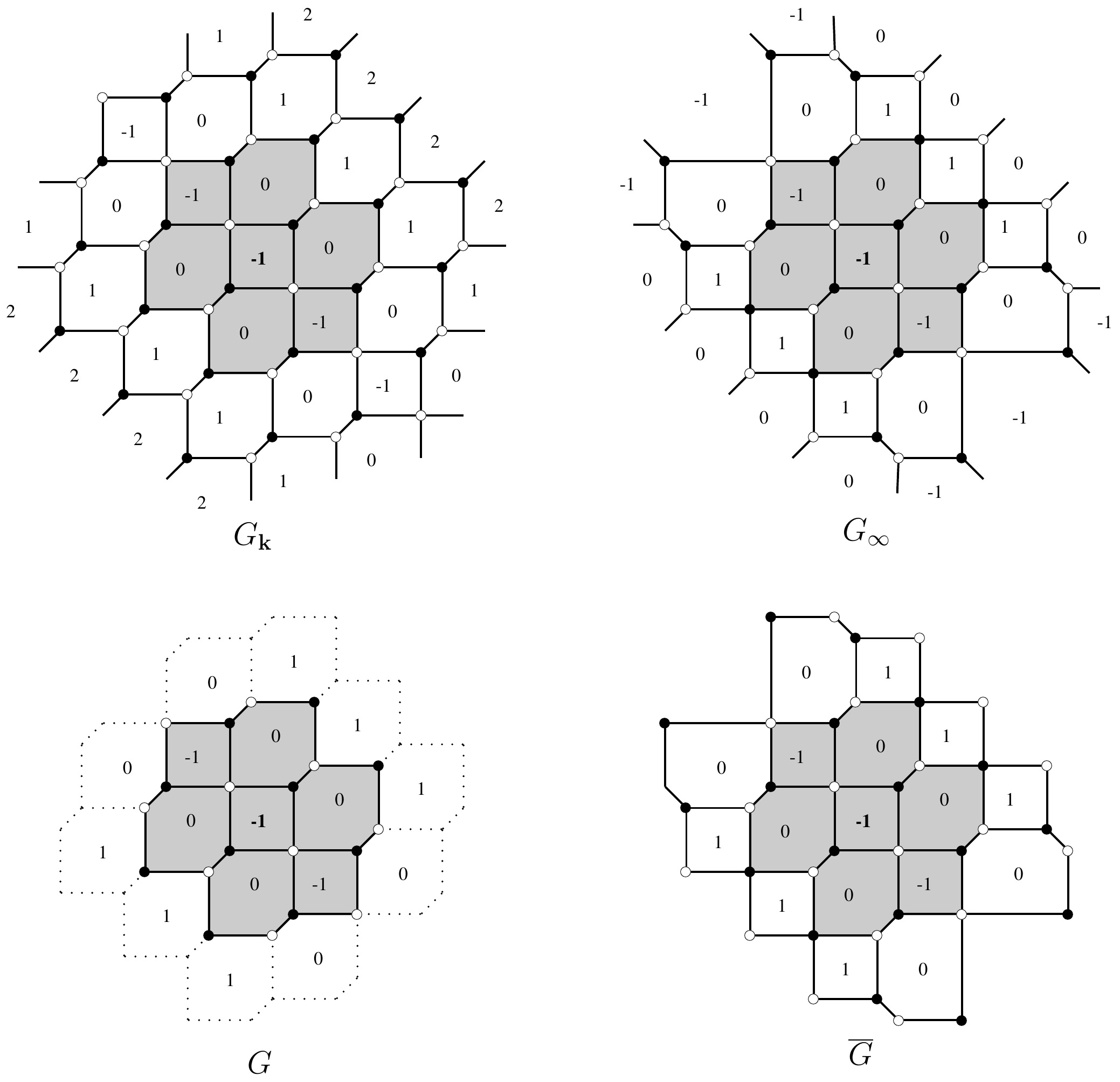}\end{center}
  \caption{
    $G_\mathbf{k}$, $G_\infty$, $G$ and $\overline{G}$ when $\k(i,j)=|i+j|-1$ and $p=(0,0,3)$.
    The shaded faces are the faces in $\mathring{F}$.
  }
  \label{fig_Gbar}
\end{figure}


\subsection{Face-weight and pairing-weight}\label{subsec:weight}

From this point onward, we let $G:=G_{p,\k}$ as a graph with open faces.
Let $\mathcal{M}$ be a set of all perfect matchings, a.k.a. dimer configurations, of $G.$
We recall that a perfect matching of $G$ is a subset $M\subseteq E(G)$ such that each $v\in V(G)$ is incident to exactly one edge in $M.$

We define the face-weight $w_f$ and the pairing-weight $w_p$ on $G$, which contribute cluster variables/initial data $t_{i,j}$'s and coefficients $c_{i,j}$'s, respectively, to the expression of $T_{i_0,j_0,k_0}$.

\begin{defn}\label{defn:faceweight}
  For a face $(i,j)\in \mathring{F}\cup \partial F,$ we define the \defemph{face-weight} depending on a perfect matching $M$ of $G$ as:
  \[
    w_f(M) := \prod_{x\in \mathring{F}\cup\partial F}w_f(x),
  \]
  where a contribution of a face to the product is defined as:
  \[
    w_f(i,j) := 
    \begin{cases}
      t_{i,j}^{\left\lceil {\frac{b-a}2}\right\rceil-1}, 	& (i,j)\in \mathring{F},\\
      t_{i,j}^{\left\lceil {\frac{b-a}2}\right\rceil}, 		& (i,j)\in \partial F,
    \end{cases}
  \]
  where $a$ is the number of sides of $(i,j)$ in the matching $M$ and $b$ the number of sides in $E(G)\setminus M.$
\end{defn}

The pairing-weight will be defined on pairs of horizontal edges in $M.$
We first note that there are exactly two types of horizontal edges in $G$ as follows.
\begin{center}
  \includegraphics{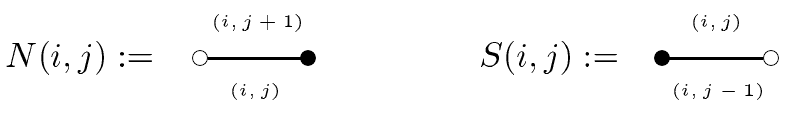}
\end{center}
\begin{itemize}
  \item a \defemph{white-black} horizontal edge, an edge joining a white vertex on the left and a black vertex on the right. We will call it $N(i,j)$, indexing by the face $(i,j)$ below it (the north side of the face $(i,j)$).
  \item a \defemph{black-white} horizontal edge, an edge joining a black vertex on the left and a white vertex on the right. We will call it $S(i,j)$, indexing by the face $(i,j)$ above it (the south side of the face $(i,j)$).
\end{itemize}

Let an \defemph{allowed pair} be a pair of $S(i,j_1)$ and $N(i,j_2)$ when $j_1\leq j_2$ in the same column of the graph.
In the other words, an allowed pair is a pair of a white-black horizontal edge above a black-white horizontal edge in the same column.
We denote $\binom{N(i,j_2)}{S(i,j_1)}$ for an allowed pair. Since $F(G)\subset\Z^2$, for each $i$ we can consider a subgraph of $G$ generated by the faces in $F(G)\cap (\{i\}\times\Z).$
In this column subgraph, we read from the bottom to the top and get a sequence of horizontal edges in the matching $M$.
We then pair these edges into allowed pairs by the following steps.
\begin{enumerate}
  \item If $S(i,j_1)$ and $N(i,j_2)$ where $j_1\leq j_2$ are consecutive in the sequence, we pair the two.
  \item Remove both $S(i,j_1)$ and $N(i,j_2)$ from the sequence, and repeat the first step until the sequence is empty.
\end{enumerate}
We do this to all of the columns of $G$.
The set $P$ of all allowed pairs obtained by this process is called the \defemph{perfect pairing} of $M$.
Proposition \ref{prop:main} will guarantee that the process works and the perfect pairing always exists.
Now the pairing-weight is defined in the following definition.

\begin{defn}\label{defn:pairweight}
  Let $P$ be the perfect pairing of a perfect matching $M$ of $G$.
  The \defemph{pairing-weight} on $M$ is defined to be:
  \[
    w_p(M):=\prod_{x\in P}w_p(x),
  \]
  where a contribution of an allowed pair in the product is defined as:
  \begin{align*}
    w_p\binom{ N(i,j_2)}{S(i,j_1)} := J_{i,j',k'}= \prod_{a=j_1-\k(i,j_1)-1}^{j_2+\k(i,j_2)+1}c_{i,a},
  \end{align*}
  and
  \begin{align*}
    j' &= j_1-\k(i,j_1)+k'-1 = j_2+\k(i,j_2)-k'+1,\\
    k' &= \dfrac{\k(i,j_1)+\k(i,j_2)-j_1+j_2}{2}+1.
  \end{align*}
  A contribution of an allowed pair in the perfect pairing to the pairing-weight can be illustrated by Figure \ref{fig_pairing}.
\end{defn}

\begin{figure}
  \begin{center}\includegraphics{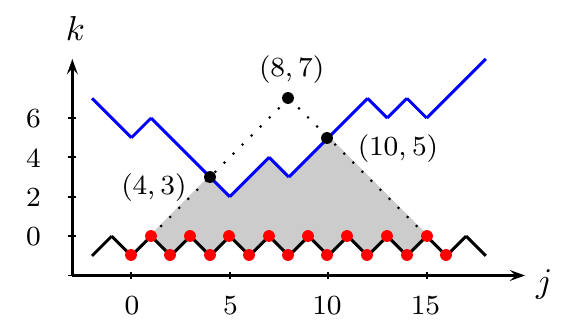}\end{center}
  \caption{
    If $\k$ is the surface depicted in blue, then $w_p\binom{ N(0,10)}{S(0,4)}= c_{0,0}c_{0,1}\dots c_{0,15} = J_{0,8,8}$ is shown in red.
    The picture is drawn in the section $i=0$ of the 3-dimensional lattice.
  }
  \label{fig_pairing}
\end{figure}

\begin{ex}
  Consider the following perfect matching $M$ of the graph $G$ from Example \ref{ex:gbar}.
  \begin{center}
    \includegraphics[scale=0.9]{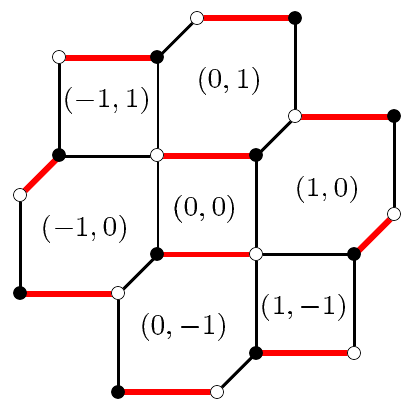}
  \end{center}
  The perfect pairing is
  \[
    P=\left\{ \binom{N(-1,1)}{S(-1,0)}, \binom{N(0,0)}{S(0,0)}, \binom{N(0,1)}{S(0,-1)}, \binom{N(1,0)}{S(1,-1)} \right\}.
  \]
  We then have
  \[
    w_p(M) = (c_{-1,-1}c_{-1,0}c_{-1,1})(c_{0,0})(c_{0,-2}c_{0,-1}\dots c_{0,2})(c_{1,-1}c_{1,0}c_{1,1}).
  \]
  Also, the face-weight of $M$ is $w_f(M) = t_{-2,0}t_{0,0}^{-1}t_{2,0}.$
\end{ex}

\begin{prop}\label{prop:main}
  Let $M$ be a perfect matching of $G$.
  Then the following holds.
  \begin{enumerate}
    \item
      If all four adjacent faces of a face $(i,j)\in\mathring{F}$ have the same height, then the face $(i,j)$ is a square.
      Also, the coloring around the face depends on the height of $(i,j)$ and its neighbors as shown in Figure \ref{fig:mutatable}.
    \item
      For each $i\in\Z,$ we get $|\{ S(i,j)\in M \mid j\in\Z\}| = |\{ N(i,j)\in M \mid j\in\Z\}|.$
      That means the number of black-white horizontal edges in $M$ and the number of white-black horizontal edges in $M$ in the same column are equal.
    \item
      For each $i,j\in\Z,$ $|\{ S(i,b)\in M \mid b\leq j\}| \geq |\{ N(i,b)\in M \mid b\leq j\}|.$
      That means in any column of $G$ the number of black-white horizontal edges in $M$ dominates the number of the white-black edges in $M$ when counting from bottom to top.
  \end{enumerate}
\end{prop}

\begin{proof}
  (1) follows directly from Table \ref{tab:quiver&k}.
  For (2) and (3), we first notice that if the point $p=(i_0,j_0,k_0)$ lies on the stepped surface $\k$ then the graph with open face associated to $p$ and $\k$ is $(G,\partial F)$ where $G$ is empty and $\partial F = \{ (i_0,j_0) \}.$
  So (2) and (3) automatically hold.
  
  If $k_0 > \k(i_0,j_0)$, then $G:=G_{p,\k} = G_{p,\k_p}$, see \eqref{eq:fun&topsurface}.
  Without loss of generality, we can then assume that $\k = \k_p$.
  Then $\k$ is obtainable from the stepped surface $\proj_p$  by a finite number of downward mutations.
  We will show (2) and (3) using induction on the number of downward mutations from $\proj_p.$
  
  When $\k$ is away from $\proj_p$ by only one downward mutation, $G$ is a square of type (S1) in Figure \ref{fig:mutatable}.
  There are only two perfect matchings of the graph, which both satisfy (2) and (3).
  
  Next, we assume that the claims hold for any surfaces which are away from $\proj_p$ by less than $n$ mutations.
  Let $\k$ be a surface away from $\proj_p$ by $n$ downward mutations.
  There must be an intermediate surface $\k'$ such that $\k'$ is obtained from $\proj_p$ by $n-1$ downward mutations and $\k$ is obtained from $\k'$ by one downward mutation, says at $(i,j)$.
  We have two cases:
  \begin{center}
    \includegraphics{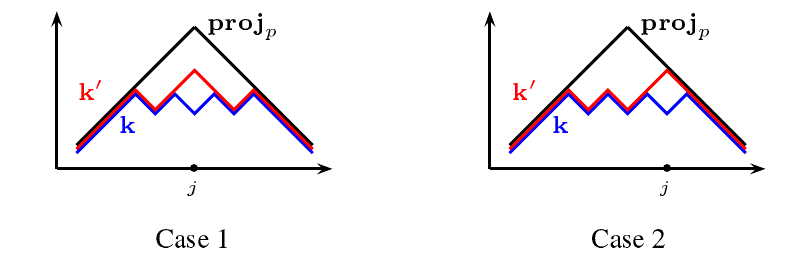}
  \end{center}
  
  \textbf{Case 1.}
  If $(i,j)$ is a closed face of $G_{p,\k'}$, then $(i,j)$ is also a closed face of $G_{p,\k}$.
  $G_{p,\k}$ is obtained from $G_{p,\k'}$ by applying the urban renewal action at the face $(i,j)$ then collapsing all degree-2 vertices created by the urban renewal.
  For any perfect matching $M$ of $G_{p,\k}$, there exists a perfect matching $M'$ of $G_{p,\k'}$ differing from $M$ only at the face $(i,j)$, see Figure \ref{fig:matching-mutation}.
  Since $M'$ satisfies (2) and (3) by the induction hypothesis, we see from Figure \ref{fig:matching-mutation} that $M$ also satisfy (2) and (3).
  So they hold for any matchings of $G_{p,\k}$.

  \textbf{Case 2.}
  If $(i,j)$ is an open face of $G_{p,\k'}$, then $(i,j)$ becomes a closed face of $G_{p,\k}$.
  We first consider the case when $i>i_0$ and $j>j_0$.
  $G_{p,\k}$ is obtained from $G_{p,\k'}$ by applying the urban renewal action at the face $(i,j)$ and collapsing all degree-2 vertices created by the urban renewal.
  This yields the correspondence of the matchings of $G_{p,\k'}$ and $G_{p,\k}$ via Figure \ref{fig_urbanopen}.
  With the same argument as for a closed face, (2) and (3) hold for any perfect matchings of $G_{p,\k}$.
  Similarly, if $i>i_0$ and $j=j_0$, the correspondence is shown in Figure \ref{fig_urbaneast}, which implies (2) and (3) for any perfect matchings of $G_{p,\k}$.
  The other cases can be treated similarly.
  
  From both cases, the statements (2) and (3) hold for every perfect matching of $G_{p,\k}.$
  By induction, we proved (2) and (3).
\end{proof}

\begin{figure}
  \begin{center}\includegraphics{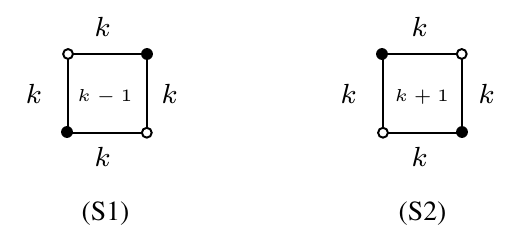}\end{center}
  \caption{
    The only two possibilities of square faces.
    The values on the faces indicate their height.
  }
  \label{fig:mutatable}
\end{figure}

\begin{table}
  \begin{center}\includegraphics[scale=0.9]{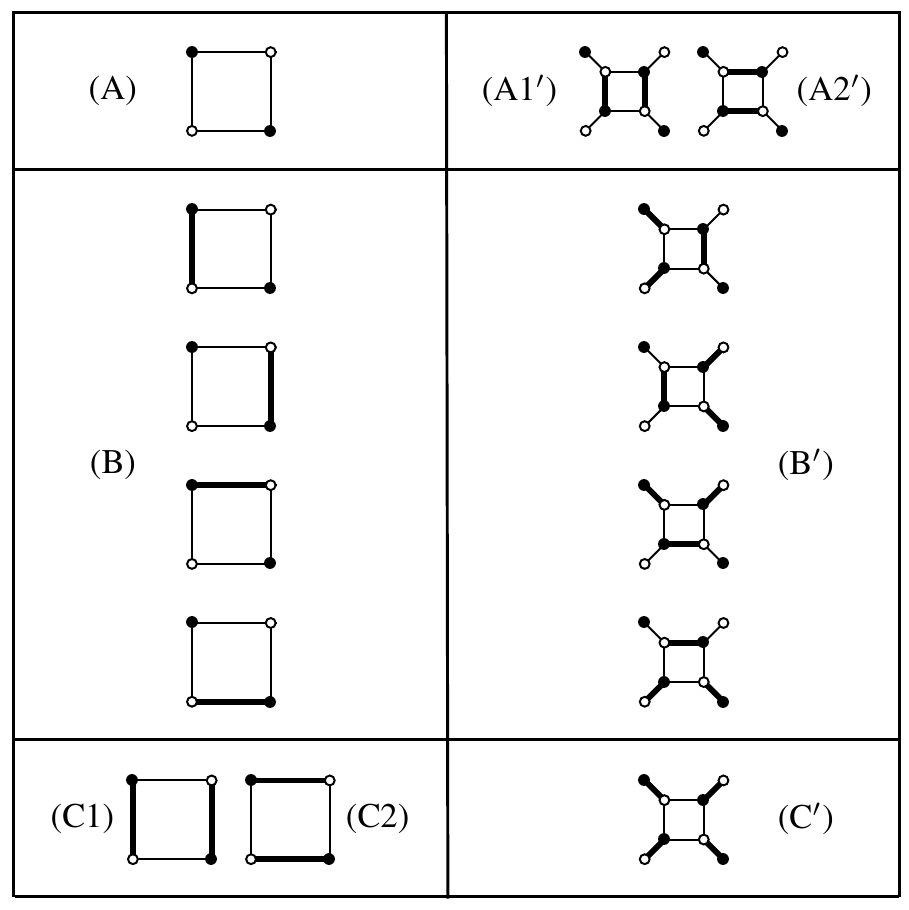}\end{center}
  \caption{The list of all correspondences between matchings before and after a single downward mutation at a closed face.}
  \label{fig:matching-mutation}
\end{table}

\begin{table}%
  \begin{center}\includegraphics[scale=0.9]{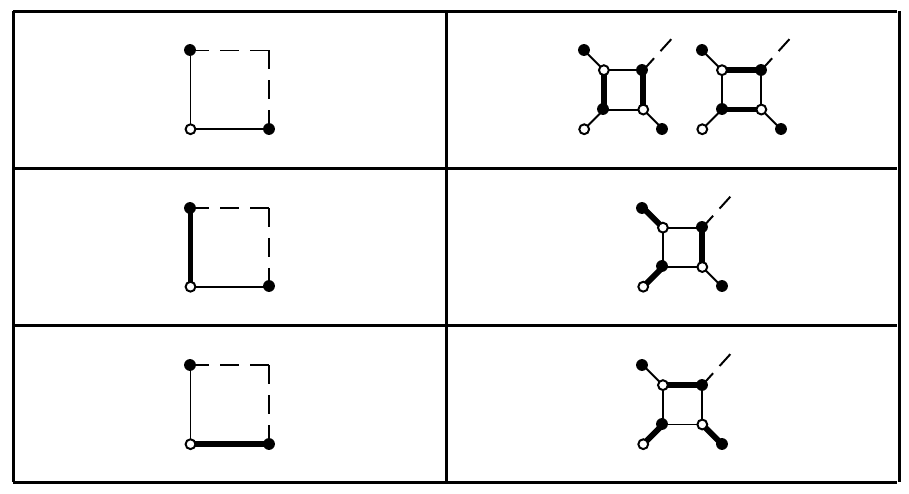}\end{center}
  \caption{The list of all correspondences between matchings before and after a single downward mutation at an open face $(i,j)$ where $i>i_0$ and $j>j_0$.}%
  \label{fig_urbanopen}%
\end{table}

\begin{table}%
  \begin{center}\includegraphics[scale=0.9]{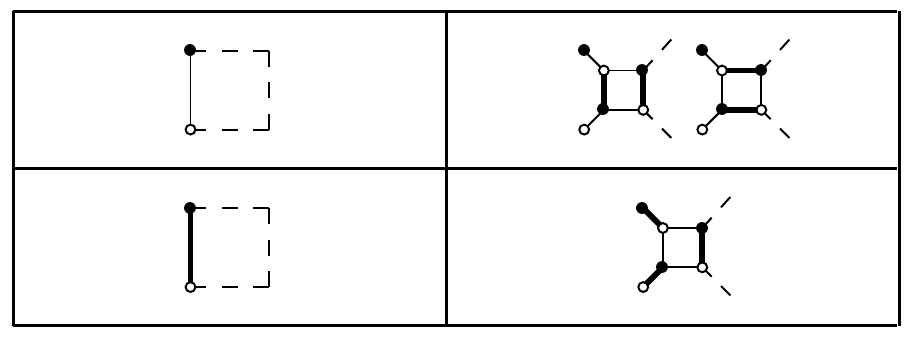}\end{center}
  \caption{The list of all correspondences between matchings before and after a single downward mutation at an open face $(i,j)$ where $i>i_0$ and $j=j_0$.}%
  \label{fig_urbaneast}%
\end{table}

We have defined face-weight and pairing-weight for perfect matchings of $G$.
The previous proposition ensures that the pairing-weight is well-defined.
We are now ready to state the main theorem.


\subsection{Perfect-matching solution}

\begin{thm}[Perfect-matching solution]\label{thm:main}
    Let $p=(i_0,j_0,k_0)$ and $\mathbf{k}$ be an admissible initial data stepped surface with respect to $p$ where $k_0\geq \k(i_0,j_0)$ and $\k\geq \fund$.
    Then
    \begin{align}
      T_{i_0,j_0,k_0} = \sum_{M\in\mathcal{M}} w_p(M)w_f(M)
    \label{eq:dimersol}
    \end{align}
    where $\mathcal{M}$ is the set of all the perfect matchings of $G=G_{p,\k}.$
  \end{thm}

This solution specializes to the solution in \cite{Speyer} for the coefficient-free T-system \cite[The Aztec Diamonds theorem]{Speyer} when $c_{i,j}=1$ for all $(i,j)\in\Z^2$.

The proof of the theorem follows from the proof in \cite{Speyer} using the ``infinite completion" of $G=G_{p,\k}$, which is the same thing as $G_\infty = G_{\k_p}$ in our setup.
To do so, we need to make sense of perfect matchings of $G_\infty$ and weight on them.

\begin{defn}[Acceptable perfect matching of $G_\infty$]
  We call a perfect matching $M_\infty$ of $G_\infty$ \defemph{acceptable} if $M_\infty \setminus E(G)$ is exactly the set of all the diagonal edges in $E(G_\infty)\setminus E(G)$.
\end{defn}

We then extend the definition of the face-weight and the pairing-weight to acceptable perfect matchings of $G_\infty$.
Notice that $G_\infty$ has no open faces.
Also the weight of $M$ and $M_\infty$ are equal, i.e.
\begin{align}
  w_p(M_\infty)w_f(M_\infty) = w_p(M)w_f(M).
  \label{eq:weightpreserve}
\end{align}
The following proposition gives a bijection between the perfect matchings of $G$ and the acceptable perfect matching of $G_\infty$.	

\begin{prop}[{{\cite[Proposition 6]{Speyer}}}]\label{prop:infcomp}
  
  There exists a bijection between the set of all perfect matchings of $G$ and the set of all acceptable perfect matchings of $G_\infty$, which maps a perfect matching $M$ of $G$ to an acceptable perfect matching $M_\infty$ of $G_\infty$ where
  \[
    M_\infty = M \cup \{ \text{diagonal edges in }E(G_\infty)\setminus E(G) \}.
  \]
\end{prop}

\begin{ex}
  Figure \ref{fig_exinfcomp} shows an example of a perfect matching $M$ of $G$ and its corresponding acceptable perfect matching $M_\infty$ of $G_\infty$ from the bijection in Proposition \ref{prop:infcomp}.
  An edge in $M_\infty$ is either an edge in $M$ (described in red) or a diagonal edge in $E(G_\infty)\setminus E(G)$ (described in blue).
 \end{ex}

\begin{figure}
  \begin{center}\includegraphics[scale=0.6]{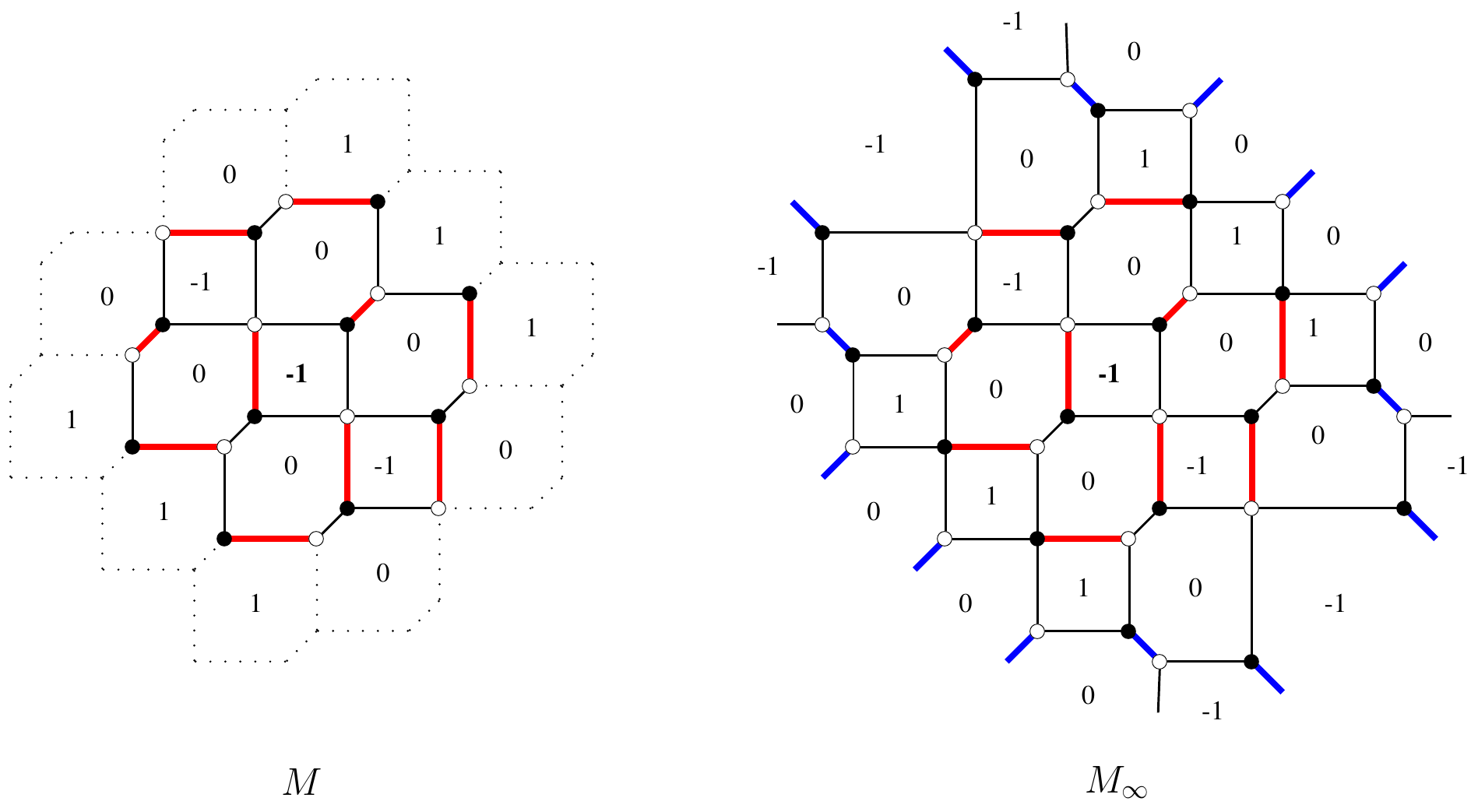}\end{center}
  \caption{A perfect matching $M$ of $G$ and its corresponding acceptable perfect matching $M_\infty$ of $G_\infty$. }
  \label{fig_exinfcomp}
\end{figure}

From \eqref{eq:weightpreserve} and Proposition \ref{prop:infcomp}, we can see that Theorem \ref{thm:main} is equivalent to the following theorem.

\begin{thm}[Perfect matching solution for $G_\infty$]\label{thm:maininf}
  Let $p$ and $\k$ be as in the assumption of Theorem \ref{thm:main} and $\k_p$ be defined as in \eqref{eq:fun&topsurface}.
  We have
  \[
    T_{i_0,j_0,k_0} = \sum_{M} w_p(M)w_f(M)
  \]
  where the sum runs over all the acceptable perfect matchings of $G_\infty=G_{\k_p}.$
\end{thm}
\begin{proof}
  Since $T_{i_0,j_0,k_0}$ depends only on $G_{p,\k}=G_{p,\k_p}$, we can assume without loss of generality that $\k = \k_p$.
  We will prove the theorem by using induction on the number of downward mutations from the top-most stepped surface $\proj_p$ to $\k$.
  
  The base case is when $\k = \proj_p$.
  The graph $G_{\proj_p}$ is shown in Figure \ref{fig:p}.
  There is only one acceptale perfect matching and its weight is $t_{i_0,j_0}=T_{i_0,j_0,k_0}$.
  So the theorem holds for the base case.
  
  Assuming that the theorem holds for any stepped surfaces away from $\proj_p$ by less than $n$ downward mutations, we let $\k$ be a surface obtained from $\proj_p$ by $n$ downward mutations.
  Then we can find an intermediate surface $\k'$ such that it is obtained from $\proj_p$ by $n-1$ downward mutations and $\k$ is obtained from $\k'$ by one downward mutation at $(i,j)$.
  We also assume that $\k(i,j)=k-1$ and $\k'(i,j)=k+1$ for some $k\in\Z$.
  By the induction hypothesis we have
  \[
    T_{i_0,j_0,k_0} = \sum_{\text{acceptable }M\text{ of }G_{\k'}} w_p(M)w_f(M).
  \]
  
  Let $M$ be any acceptable perfect matching of $G_{\k'}$.
  By Proposition \ref{prop:main}, the face $(i,j)$ of $G_{\k'}$ is a square of type (S2) in Figure \ref{fig:mutatable}.
  Then the matching $M$ at the face $(i,j)$ must be one of the 7 cases in the first column of Figure \ref{fig:matching-mutation}.
  
  If $M$ is of type $(A)$ at $(i,j)$, there are two matchings $M_{A1'}$ and $M_{A2'}$ of $G_{\k}$ of type $(A1')$ and $(A2'),$ respectively, such that the matchings are the same except locally at the face $(i,j)$.
  We then have
  \begin{align*}
    w_p(M_{A1'})w_f(M_{A1'}) &= \frac{T_{i,j-1,k}T_{i,j+1,k}}{T_{i,j,k-1}T_{i,j,k+1}}w_p(M)w_f(M),\\
    w_p(M_{A2'})w_f(M_{A2'}) &= \frac{J_{i,j,k}T_{i-1,j,k}T_{i+1,j,k}}{T_{i,j,k-1}T_{i,j,k+1}}w_p(M)w_f(M).
  \end{align*}
  The term $J_{i,j,k}$ in the second equation came from an extra pair $\binom{ N(i,j)}{ S(i,j)}$ in $M_{A2'}$ which gives an extra term $J_{i,j,k}$ to the pairing-weight.
  By \eqref{eq:tsyscoef}, we have
  \begin{align}
    w_p(M_{A1'})w_f(M_{A1'})+w_p(M_{A2'})w_f(M_{A2'}) = w_p(M_A)w_f(M_A).
  \label{eq:a}
  \end{align}

  If $M$ is of type $(B)$, there exists a unique corresponding matching $M_{B'}$ of $G_\k$ of type $(B')$.
  We see that the weight of $M$ and $M_{B'}$ are equal.
  That is
  \begin{align}
    w_p(M_{B'})w_f(M_{B'}) = w_p(M)w_f(M).
  \label{eq:b}
  \end{align}

  If $M$ is of type $(C1)$ (resp. $(C2)$), let $M'$ be another matching of $G_{\k'}$ of type $(C2)$ (resp. $(C1)$) which is the same as $M$ except for the two edges at the face $(i,j)$.
  Without loss of generality, we assume that $M$ is of type $(C1)$ and $M'$ is of type $(C2)$.
  Then there exists a corresponding perfect matching $M_{C'}$ of $G_\k$ of type $(C')$.
  We then have
  \[
    w_p(M)w_f(M) = \frac{T_{i,j-1,k}T_{i,j+1,k}}{T_{i,j,k-1}T_{i,j,k+1}}w_p(M_{C'})w_f(M_{C'}).
  \]

  To write $w_p(M')$ in terms of $w_p(M_{C'})$, we first notice that there must be two other edges $S(i,j_1),N(i,j_2)\in M'$ where $j_1<j<j_2$ such that both pairs $\binom{N(i,j-1)}{ S(i,j_1)}$ and $\binom{N(i,j_2)}{S(i,j+1)}$ are in the perfect pairing of $M'$, while in its corresponding $M_{C'}$ there is only $\binom{N(i,j_2)}{S(i,j_1)}$.
  Thus we have
  \[
    w_p(M') = \frac{\prod_{b=j_1-k_1-1}^{j-1+k+1} c_{i,b} \prod_{b=j+1-k-1}^{j_2+k_2+1} c_{i,b}}{\prod_{b=j_1-k_1-1}^{j_2+k_2+1} c_{i,b}}w_p(M_{C'})= J_{i,j,k}w_p(M_{C'}),
  \]
  and so
  \[
    w_p(M')w_f(M') = \frac{J_{i,j,k}T_{i-1,j,k}T_{i+1,j,k}}{T_{i,j,k-1}T_{i,j,k+1}}w_p(M_{C'})w_f(M_{C'}).
  \]
  Hence,
  \begin{align}
    w_p(M)w_f(M)+w_p(M')w_f(M') = w_p(M_{C'})w_f(M_{C'}).
  \label{eq:c}
  \end{align}
  By \eqref{eq:a}, \eqref{eq:b}, \eqref{eq:c} and the induction hypothesis, we can conclude that
  \[  
    \sum_{\text{acceptable }M\text{ of }G_{\k}} w_p(M)w_f(M) = \sum_{\text{acceptable }M\text{ of }G_{\k'}} w_p(M)w_f(M) = T_{i_0,j_0,k_0}. 
  \]
  So the statement holds for $\k$. By the induction, we proved the theorem.
\end{proof}

\begin{figure}
  \begin{center}\includegraphics{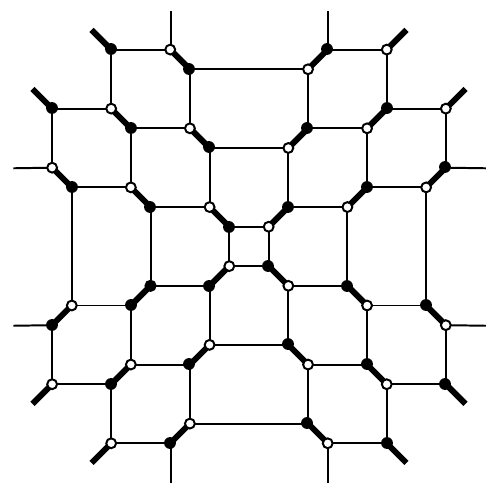}\end{center}
  \caption{
    The only one acceptable perfect matching of $G_{\proj_p}$.
    The center face of the graph is $(i_0,j_0)$.
  }
  \label{fig:p}
\end{figure}

Now we have proved Theorem \ref{thm:main} and Theorem \ref{thm:maininf}.
In the proof of Theorem \ref{thm:maininf}, we notice that for any acceptable perfect matching $M_\infty$ of $G_\infty$, any face outside $\mathring{F}\cup\partial F$ always gives $1$ for its face-weight.
Also edges in $M_\infty\setminus E(G)$ are all diagonal, so they will not contribute any weight to the partition function.
We then have the following perfect-matching solution for the closure $\overline{G}$ of $G$.

\begin{thm}[Perfect matching solution for $\overline{G}$]\label{thm:mainclosure}
  Let $p$ and $\k$ be as in the assumption of Theorem \ref{thm:main} and $\overline{G}$ be the closure of $G=G_{p,\k}$ defined as in Definition \ref{defn:Gclosure}.
  We have
  \[
    T_{i_0,j_0,k_0} = \sum_{\overline{M}\in\overline{\mathcal{M}}} w_p(\overline{M})w_f(\overline{M})
  \]
  where $\overline{\mathcal{M}} := \left\{ \overline{M}= M\cup \diag(E(\overline{G})\setminus E(G)) \mid M\in\mathcal{M} \right\}$, 
  $\diag(A)$ is the set of all diagonal edges in $A$, and $\mathcal{M}$ is the set of all perfect matchings of $G$.
\end{thm}

We now have a combinatorial expression of $T_{i_0,j_0,k_0}$ as a partition function of face-weight and pairing-weight over all perfect matchings of a graph.
In the next section, we will combine the two weights together and construct an edge-weight.
This will be the first step toward our next aim to construct a solution in terms of networks, analog to \cite{DF14} for the coefficient-free T-system.


\section{Perfect-matching solution via edge-weight} \label{sec_edge-weight}
	Now that we have multiple versions of the perfect matching solution to the T-system with principal coefficients in Theorem \ref{thm:main}, Theorem \ref{thm:maininf} and Theorem \ref{thm:mainclosure}, our next goal is to find a network solution analog to the network solution for coefficient-free T-systems studied in \cite{DFK13,DF14}.
	One big advantage of the network solution is its explicit solution via the network matrices.
	In the perfect matching solution, we need to enumerate all the perfect matchings of the graph $G$ in order to compute the solution.
	For the network solution, we associate an initial data stepped surface with a product of network matrices.
	Then the solution is just a certain minor of the product.
	
	In order to get the network solution, we first transform the face-weight and the pairing-weight studied in the last section to the edge-weight $w_e$ (Definition \ref{defn:edgeweight}) on edges of the closure $\G$ of $G$, which also gives us a new perfect-matching solution but with the edge-weight $w_e$ (Theorem \ref{thm:edgesol}).
	This solution will be used to construct a nonintersecting-path solution in the next section.
	We also note that our edge-weight coincide with the weight studied in \cite{MS10,JMZ13} in the case when all $c_{i,j}=1$.

	\begin{defn}[Edge-weight $w_e$]\label{defn:edgeweight}
		Let $\k$ be an admissible initial stepped surface with respect to $p$, $\G$ be the closure of the graph $G=G_{p,\k}$.
		For each edge of $\G$ we assign the \defemph{edge-weight} $w_e$ as follows:
		\begin{center}
			\includegraphics[trim=1cm 0cm 1cm 0cm, scale=0.9]{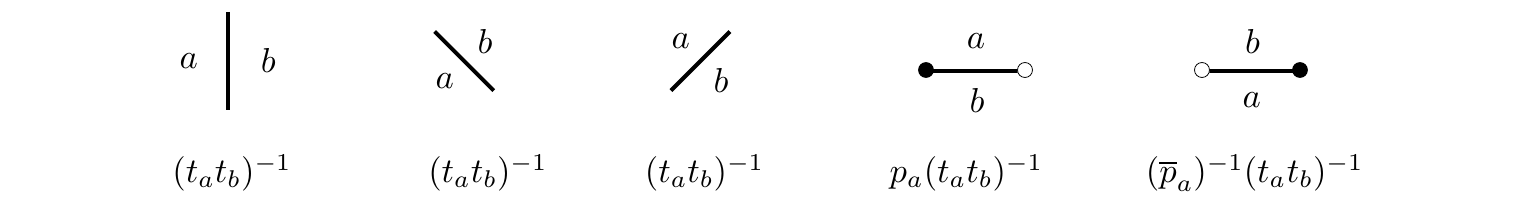}
		\end{center}
		where $p_a$ and $\overline{p}_a$ are the following formal products
		\[
			p_a = \prod_{\alpha=j-k-1}^\infty c_{i,\alpha} \quad\text{and} \quad \overline{p}_a = \prod_{\alpha=i+k+2}^\infty c_{i,\alpha}
		\]
		when $a=(i,j)$ and $k=\k(i,j).$ We also assume that $t_a=1$ when $a\notin F(\G)$ and $p_a=\overline{p}_a = 1$ when $a\notin F(G)$.
	\end{defn}

	\begin{defn}\label{defn:Mbar0}
		Let $\mathcal{M}$ be the set of all perfect matching of $G.$ For a matching $M\in\mathcal{M}$, we let
		\[
			\overline{M} := M\cup\diag(E(\G)\setminus E(G))
		\]
		be its corresponding (not necessary perfect) matching of $\G$ where $\diag(A)$ is the set of all diagonal edges in $A$ for $A\subseteq E(\G).$
		Also let
		\begin{align*}
			\overline{\mathcal{M}} &:= \left\{ \overline{M} \mid M\in\mathcal{M}\right\},\\
			\overline{M}_0 &:= \left\{\text{all white-black horizontal and diagonal edges of }\G\right\}.
		\end{align*}
		Then the \defemph{edge-weight} of a matching $\overline{M}\in\overline{\mathcal{M}}$ is
		\begin{align}
			w_e(\overline{M}) := \prod_{x\in \overline{M}} w_e(x).
			\label{eq:edgeweight}
		\end{align}
	\end{defn}

	We note that $\overline{M}$ and $\overline{M}_0$ are not necessary perfect matchings of $\G.$
	Also for $j_1\leq j_2$,
	\[
		p_{(i,j_1)} (\overline{p}_{(i,j_2)})^{-1} = \prod_{\alpha = j_1-\k(i,j_1)-1}^{j_2+\k(i,j_2)+1} = w_p\binom{N(i,j_2)}{S(i,j_1)}.
	\]
	By Proposition \ref{prop:main}, the product in \eqref{eq:edgeweight} is indeed a finite product of pairing-weights, hence a finite product of $c_{i,j}$'s.
	
	The following lemma interprets the face-weight of a matching $\overline{M}$ as a function of $\overline{M}$ and our special matching $\overline{M}_0$.

	\begin{lemma}\label{lem:faceweightedgeweight}
		For $x\in F(\G)$ and $\overline{M}\in\overline{\mathcal{M}}$, we have
		\[
			w_f(x) = t_x^{N_x-D_x}
		\]
		where $N_x=\left| \{ e\in \overline{M}_0 \mid e\text{ is a side of }x \} \right|$ and $D_x = \left| \{ e\in \overline{M} \mid e\text{ is a side of }x \} \right|$.
	\end{lemma}
	\begin{proof}
		We can easily check that for all $x\in F(\G)=\mathring{F}\cup\partial F,$ we get $w_f(x) = t_x^{\lceil \frac{N-D}{2}\rceil-1}$ where $N$ and $D=D_x$ are the numbers of sides of $x$ which are not in $\overline{M}$ and are in $\overline{M}$, respectively.
		Let $S\in\{4,6,8\}$ be the number of sides of $x$. Then $N=S-D.$
		So 
		\[
			w_f(x) = t_x^{\lceil \frac{S-2D}{2}\rceil-1}=t_x^{S/2-D-1}.
		\]
		Since $x$ must be one of the cases in Figure \ref{fig_468}, we have $S/2-1 = N_x.$
		Hence $w_f(x) = t_x^{N_x-D_x}$.
	\end{proof}

	\begin{figure}
		\begin{center}\includegraphics[scale=0.8]{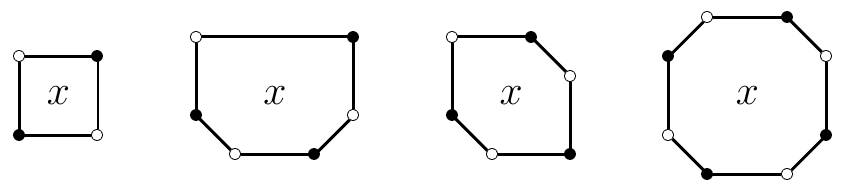}\end{center}
		\caption{All possible faces of $\G$ up to $\frac{\pi}{2}$ rotation and $i$-axis reflexion.}
		\label{fig_468}
	\end{figure}

	\begin{thm}[Perfect-matching solution for $\G$ with the edge-weight $w_e$]\label{thm:edgesol}
		Let~ $\k$ be an admissible initial data stepped surface with respect to $p=(i_0,j_0,k_0)$, $\overline{M}$ and $\overline{M}_0$ be defined as in Definition \ref{defn:Mbar0}.
		Then
		\[
			T_{i_0,j_0,k_0} = \sum_{\overline{M}\in\overline{\mathcal{M}}} w_e(\overline{M}) \Big/ w_e(\overline{M}_0)\big|_{c_{i,j}=1}
		\]
		where $w_e(\overline{M}_0)\big|_{c_{i,j}=1}$ denotes the substitution $c_{i,j}=1$ for any $(i,j)\in\Z^2.$
	\end{thm}
	\begin{proof}
		Let $\overline{M}\in\overline{\mathcal{M}}$.
		By Lemma \ref{lem:faceweightedgeweight}, we get
		\begin{align*}
			w_f(\overline{M})= \prod_{x\in F(\G)}t_x^{N_x-D_x}= \prod_{x\in F(\G)}t_x^{N_x}\prod_{x\in F(\G)}t_x^{-D_x}.
    \end{align*}
    By Definition \ref{defn:edgeweight}, it equals to
    \begin{align*}
			w_f(\overline{M}) = {\Big(w_e(\overline{M}_0)\big|_{c=1}\Big)}^{-1} \prod_{x\in F(\G)}t_x^{-D_x}.
		\end{align*}
		Thus we have
		\begin{align}
			\prod_{x\in F(\G)}t_x^{-D_x} = w_f(\overline{M}) w_e(\overline{M}_0)\big|_{c=1}.
		\label{eq:1}
		\end{align}
		By the definition of $w_e$, we consider
		\begin{align*}
			w_e(\overline{M}) = \prod_{y\in \overline{M}}w_e(y) =	\prod_{\bullet\overset{a}{\text{---}}\circ\in\overline{M}}p_a \prod_{\circ\underset{b}{\text{---}}\bullet\in\overline{M}}\overline{p}_b \prod_{x\in F(\G)}t_x^{-D_x}.
    \end{align*}
    From \eqref{eq:1} and the fact that $\overline{M}\setminus M$ contains only diagonal edges, we then get
    \begin{align*}
			w_e(\overline{M})=\prod_{\bullet\overset{a}{\text{---}}\circ\in M}p_a \prod_{\circ\underset{b}{\text{---}}\bullet\in M}\overline{p}_b ~ w_f(\overline{M}) w_e(\overline{M}_0)\big|_{c=1}.
		\end{align*}
		Since
		\[
			\prod_{\bullet\overset{a}{\text{---}}\circ\in M}p_a \prod_{\circ\underset{b}{\text{---}}\bullet\in M}\overline{p}_b = w_p(\overline{M}),
		\]
		we conclude that $w_f(\overline{M})w_p(\overline{M}) = w_e(\overline{M})\big/w_e(\overline{M}_0)\big|_{c=1}$ for any $\overline{M}\in\overline{\mathcal{M}}.$
		By Theorem \ref{thm:mainclosure}, we have $T_{i_0,j_0,k_0} = \sum_{\overline{M}\in\overline{\mathcal{M}}} w_e(\overline{M}) \big/ w_e(\overline{M}_0)|_{c=1}.$
	\end{proof}
	
	Now we have a combinatorial expression of $T_{i_0,j_0,k_0}$ in terms of a partition function of edge-weight over all matchings $\G$.
	In the next section, we give an explicit bijection between perfect matchings and non-intersecting paths (with certain sources and sinks) in both $G$ and $\G$.
	Using this bijection, we are able to transform the perfect-matching solutions to a solution in terms of non-intersecting paths.


\section{Non-intersecting path solution}\label{sec_path}
	In this section, we provide an explicit bijection (Proposition \ref{prop:bijection}) between the perfect matchings of $G$ and the non-intersecting paths in the oriented graph $G$ with certain sources and sinks.
  It can be extended to a bijection between the matchings in $\overline{\mathcal{M}}$ of $\overline{G}$ and the non-intersecting paths in the oriented graph $\overline{G}$ with certain sources and sinks (Proposition \ref{prop:bijectionbar}).
	Using this bijection and a new weight $w_e'$ modified from the edge-weight $w_e$, we can write the solution to the T-system in terms of non-intersecting paths in $\overline{G}$ (Theorem \ref{thm_pathsol}).
	
	\subsection{Some setup}
		We first show some properties of the graph $G$ and $\overline{G}$.
		\begin{prop}
			$G$ and $\overline{G}$ are bipartite and connected.
		\end{prop}
		\begin{proof}
			It was proved in \cite[Section 3.5]{Speyer} that $G$ is bipartite and connected. The extended result to $\overline{G}$ follows from Proposition \ref{prop:modifiedopenface}.
		\end{proof}

		\begin{defn}\label{defn:leftright}
			For two vertices $v,v'$ of a graph, we say that $v$ (resp. $v'$) is on the \defemph{left} (resp. \defemph{right}) of $v'$ (resp. $v$) if there is a sequence of vertices of the graph $v=v_0,v_1,v_2,\dots,v_n=v'$ such that any two consecutive vertices are connected by one of the following edges.
			\begin{center}
				\includegraphics{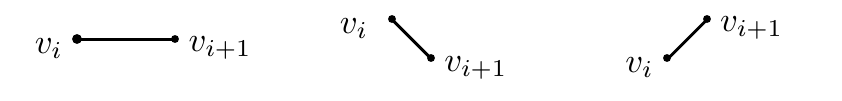}
			\end{center}
		\end{defn}


		We center the graph at the face $(i_0,j_0)$.
		Then the notions of the North-West, North-East, South-West and South-East of the center are well-defined.

		\begin{defn}\label{defn:direction}
			We denote $V_{\rm NW}(G)$, $V_{\rm NE}(G)$, $V_{\rm SW}(G)$ and $V_{\rm SE}(G)$ to be the set of all the vertices of a graph $G$ in the North-West, North-East, South-West and South-East of the center face $(i_0,j_0)$, respectively.
			Also let $V_{\rm leftNW}(G)$, $V_{\rm right NE}(G)$, $V_{\rm left SW}(G)$, $V_{\rm right SE}(G)$ be the set of all left-most NW vertices, right-most NE, left-most SW and right-most SE vertices of $G$, respectively.
			See Figure \ref{fig_ex} for an example.
		\end{defn}

		\begin{figure}
			\begin{center}\includegraphics[scale=0.8]{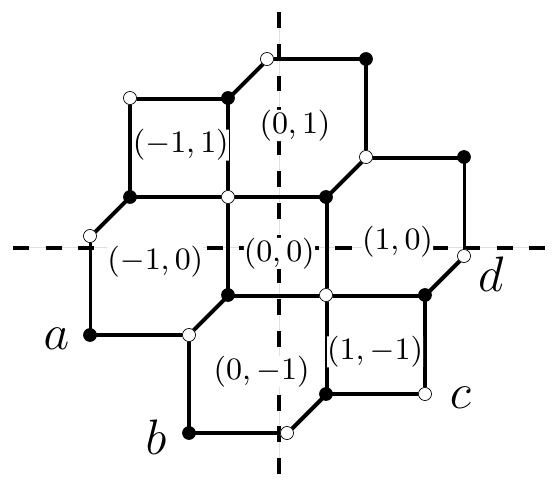}\end{center}
			\caption{With $(0,0)$ as the center face, we have $|V_{SW}(G)|=4,$ $|V_{SE}(G)|=6,$ $V_{\rm left SW}(G)=\{a,b\}$ and $V_{\rm right SE}(G)=\{c,d\}$.}
			\label{fig_ex}
		\end{figure}

		\begin{prop}\label{prop_g}
			$G$ has the following properties.
			\begin{enumerate}
				\item Vertices in $V_{\rm left SW}(G)\cup V_{\rm right NE}(G)$ are black.
				\item Vertices in $V_{\rm right SE}(G)\cup V_{\rm left NW}(G)$ are white.
				\item $|V_{\rm left SW}(G)|=|V_{\rm right SE}(G)|$.
				\item $|V_{\rm left NW}(G)|=|V_{\rm right NE}(G)|$.
			\end{enumerate}
		\end{prop}
		\begin{proof}
			To show (1), we use the same analysis as in the proof of Proposition 5 in \cite{Speyer}.
			It is clear that a left-most vertex must be on the boundary of $G.$
			We then consider an open face $(i,j)\in \partial F$ on the South-West of $(i_0,j_0)$ with height $k\in\mathbb{Z}.$
			All the eight possibilities are shown in Figure \ref{fig_SW}, where the face of height $k$ in the circle is $(i,j)$.
			We see that all the left-most South-West vertices must be black.
			A similar argument can be used to show the other case of (1) and (2).
			In order to show (3), we consider a maximal sequence $v_0,v_1,\dots,v_n$ of vertices of $G$ such that $v_i$ is on the left of $v_{i+1}$ for all $i$.
			We have $v_0\in V_{\rm left SW}(G)$ if and only if $v_n\in V_{\rm right SE}(G)$.
			This gives a bijection between $V_{\rm left SW}(G)$ and $V_{\rm right SE}(G)$.
			So we proved (3).
			The similar argument also shows (4).
		\end{proof}

		\begin{figure}
			\begin{center}
				\includegraphics[scale=0.7]{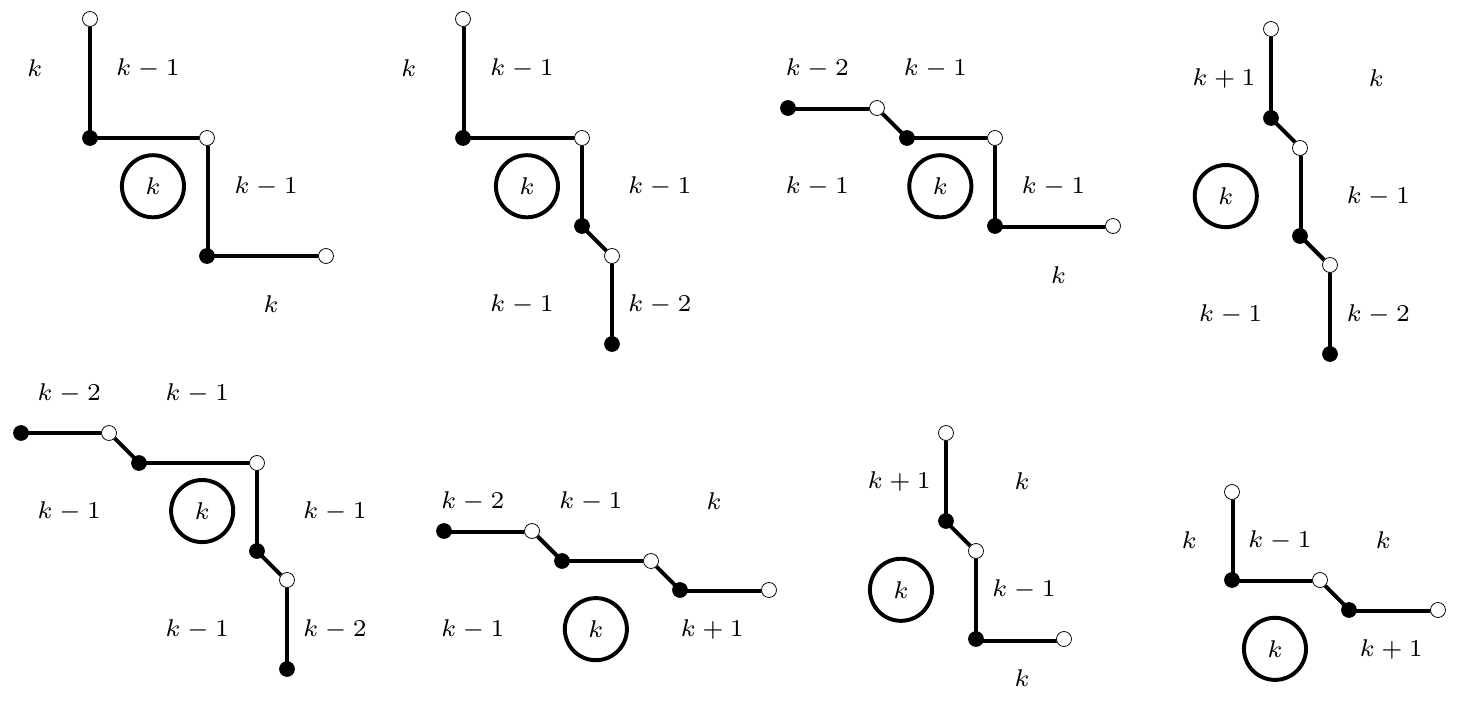}
			\end{center}
			\caption{
				All the eight possibilities of the boundary near an open face in the South-West of the center.
				The face is of height $k$ and is marked by a circle.
			}
			\label{fig_SW}
		\end{figure}
		
	\subsection{Non-intersecting paths and perfect matchings}

		We will give an explicit bijection between the perfect matchings of $G$ and the non-intersecting paths from $V_{\rm left SW}(G)$ to $V_{\rm right SE}(G)$ in $G$ (Proposition \ref{prop:bijection}).
		This bijection can be extended to $\overline{G}$ in Proposition \ref{prop:bijectionbar}.
		
		In order to define a path in $G$ and $\G$, we need an orientation of the graphs.
		
		\begin{defn}[Edge orientation of $G$ and $\overline{G}$]\label{defn:edgeoreient}
			Let the orientation of the edges of $G$ and $\overline{G}$ be such that it goes from left to right for diagonal and horizontal edges and from a black vertex to a white vertex for the vertical as follows.
			\medskip
			\begin{center}
				\includegraphics[scale=0.8]{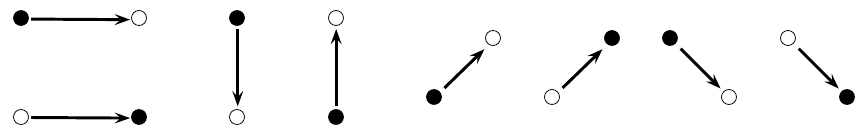}
			\end{center}
		\end{defn}
		
		\begin{defn}\label{defn:sourcesink}
			Let $G=G_{p,\k}$ and $\G$ be as in Definition \ref{defn:Gclosure}. With the notations from Definition \ref{defn:direction}, we define
			\begin{align*}
				\begin{aligned}
					M_0 &:= \{\text{white-black horizontal and diagonal edges of }G\},\\
					\overline{M}_0 &:= \{\text{white-black horizontal and diagonal edges of }\overline{G}\},\\
					\mathcal{M} &:= \{\text{perfect matchings of }G\},\\
					\overline{\mathcal{M}} &:= \{\overline{M}= M\cup\diag(E(\overline{G})\setminus E(G)) \mid M\in\mathcal{M}\},\\
					\mathcal{P} &:= \{\text{non-intersecting paths from }V_{\rm left SW}(G)\text{ to }V_{\rm right SE}(G)\text{ on }G\},\\
					\overline{\mathcal{P}} &:= \{\text{non-intersecting paths from }V_{\rm left SW}(\overline{G})\text{ to }V_{\rm right SE}(\overline{G})\text{ on }\overline{G}\}.
				\end{aligned}
			\end{align*}
		\end{defn}
		
		The map $\Phi$ (resp. $\overline{\Phi}$) in the following definition is indeed the bijection between $\mathcal{M}$ and $\mathcal{P}$ (resp. between $\overline{\mathcal{M}}$ and $\overline{\mathcal{P}}$).
		They will be key ingredients to construct our nonintersecting-path solution.
			
		\begin{defn}\label{defn:barphi}
			We define a map $\Phi:\mathcal{M}\rightarrow\mathcal{P}$ by
			\begin{align*}
				\Phi(M) :=M\triangle M_0,\quad\text{for }M\in\mathcal{M}.
			\end{align*}
			This map can be extended to a map $\overline{\Phi}:\overline{\mathcal{M}}\rightarrow\overline{\mathcal{P}}$ defined by
			\begin{align*} 
				\overline{\Phi}(\overline{M}) := \overline{M}\triangle\overline{M}_0,\quad\text{for }\overline{M}\in\overline{\mathcal{M}},
			\end{align*}
			where $A\triangle B := A\cup B \setminus A\cap B = (A\setminus B)\sqcup(B\setminus A)$ is the symmetric difference of $A$ and $B$.
			The notation $\sqcup$ represents the disjoint union.
		\end{defn}

		\begin{remark} It is worth mentioning that if we consider $M\in\mathcal{M}$ and $P\in\mathcal{P}$ as sets of dimers on edges of $G,$ then the action of $\Phi$ can be interpreted as superposing with $M_0$ and counted the number of dimers on each edge modulo 2.
		Similarly, $\overline{\Phi}$ is the superposition with $\overline{M}_0$ with number of dimers being counted modulo 2.
		\end{remark}

		\begin{ex}
			This example shows an interpretation of the map $\Phi$ on $\mathcal{M}$ as the superposition with $M_0$ modulo 2.
			\begin{center}
				\includegraphics[scale=0.7]{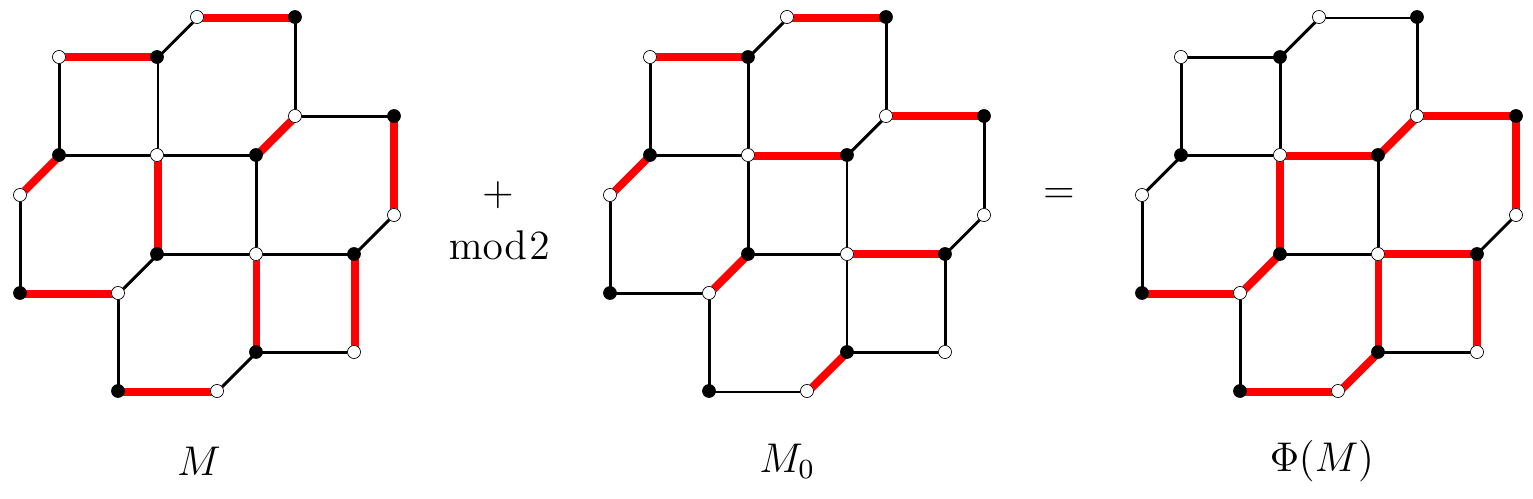}
			\end{center}
			We notice that the image of the map is indeed an element in $\mathcal{P}$.
		\end{ex}

		Using the following lemmas, Proposition \ref{prop:bijection} shows that $\Phi$ is a well-defined bijection from $\mathcal{M}$ to $\mathcal{P}$.

		\begin{lemma}\label{lem1}
			Let $M\in\mathcal{M}$. Then $\Phi(M)$ has the following properties.
			\begin{enumerate}
				\item For any vertex in $V_{\rm left SW}(G)\cup V_{\rm right SE}(G)$, exactly one of its incident edges is in $\Phi(M)$.
				\item For the other vertices, either none or two of their incident edges are in $\Phi(M).$
			\end{enumerate}
		\end{lemma}
		\begin{proof}
			To show (1), we first consider a vertex in $V_{\rm left SW}(G)$.
			By the fact that $M$ is a perfect matching, exactly one of its incident edges is in $M$.
			By Proposition \ref{prop_g}, the vertex is black.
			Since it is the left-most, none of its incident edges is in $M_0$.
			This is because there cannot be a white vertex on its left.
			So, exactly one of its incident edges is in $\Phi(M)$.
			Similarly, exactly one incident edge of a vertex in $V_{\rm right SE}(G)$ is in $M$.
			There is none of its edges is in $M_0$ because the vertex is white and the right-most. Hence (1) holds.

			For (2), if a vertex is not in $V_{\rm left SW}(G)\cup V_{\rm right SE}(G)$, it must be either in $V_{\rm left NW}(G)\cup V_{\rm right NE}(G)$ or it has both left and right adjacent vertices.
			For a vertex in $V_{\rm left NW}(G)\cup V_{\rm right NE}(G)$, there is one of its incident edges in $M,$ and also one in $M_0.$
      They can be the same or different.
			So, either none or two of its incident edges are in $\Phi(M)$.
			For a vertex having left and right adjacent vertices, one of its incident edges must be in $M_0$.
			This is because the vertex is either black and is on the right of a white vertex, or white and on the left of a black vertex.
			Also its one of the incident edges must be in $M$ because $M$ is a perfect matching.
			So either none or two of its incident edges are in $\Phi(M)$.
			Hence (2) holds.
		\end{proof}

		\begin{lemma}\label{lem2}
			For a black vertex of $G$ having two incident edges in $\Phi(M),$ the two edges must be of the form A, B or C in Figure \ref{fig_pathatblack} where the horizontal edges in the figure represent horizontal or diagonal edges of $G$.
		\end{lemma}
		\begin{proof}
			It is easy to see that the six cases in Figure \ref{fig_pathatblack} are all configurations of two incident edges of a vertex of $G$.
			Since $M$ is a perfect matching, one of the two edges must come from $M_0$.
			Since the cases (D), (E) and (F) contain no edge from $M_0$, they cannot happen.
		\end{proof}
		
		\begin{lemma}\label{lem3}
			For a white vertex of $G$ having two incident edges in $\Phi(M),$ the two edges must be of the form A, B or C in Figure \ref{fig_pathatwhite} where the horizontal edges in the figure represent horizontal or diagonal edges of $G$.
		\end{lemma}
		\begin{proof}
			Similar to Lemma \ref{lem2}.
		\end{proof}

		\begin{figure}
			\begin{center}
				\includegraphics[scale=0.8]{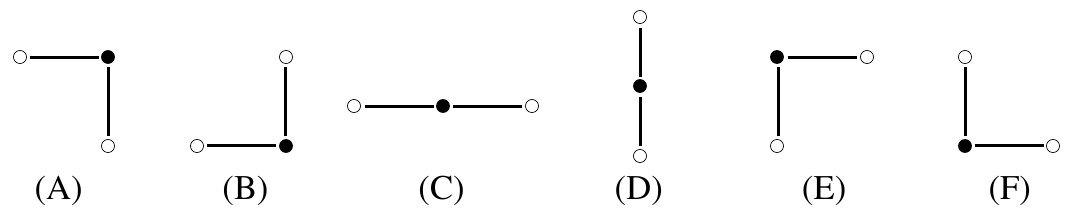}
			\end{center}
			\caption{
				All configurations of two incident edges of a black vertex of $G$.
				The horizontal edges in the picture represent horizontal/diagonal edges of $G$.
			}
			\label{fig_pathatblack}
		\end{figure}
	
		\begin{figure}
			\begin{center}
				\includegraphics[scale=0.8]{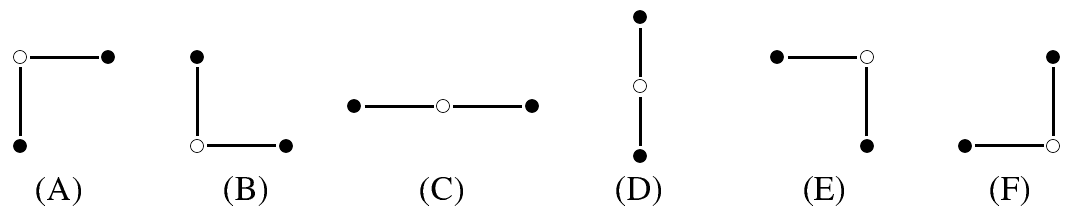}
			\end{center}
			\caption{
				All configurations of two incident edges of a white vertex of $G$.
				The horizontal edges in the picture represent horizontal/diagonal edges of $G$.
			}
			\label{fig_pathatwhite}
		\end{figure}

		\begin{prop}\label{prop:bijection}
			$\Phi:\mathcal{M}\rightarrow\mathcal{P}$ is a bijection.
		\end{prop}
		\begin{proof}
			To show that $\Phi(M)$ is well-defined, we need to show that $\Phi(M)$ is a collection of non-intersecting paths from $V_{\rm left SW}(G)$ to $V_{\rm right SE}(G)$ in $G$, where $G$ is oriented as in Definition \ref{defn:edgeoreient}.
			From Lemma \ref{lem2} and Lemma \ref{lem3}, only the case (A), (B) or (C) can happen at any vertex.
			From this, all the vertices incident to two edges in $\Phi(M)$ are neither source nor sink.
			So $\Phi(M)$ is a set of non-intersecting paths in the oriented graph $G$, where a source or a sink is incident to exactly one edge in $\Phi(M)$.
			From Lemma \ref{lem1}, the sources and the sinks must in $V_{\rm left SW}(G) \cup V_{\rm right SE}(G)$.
			But (1) and (2) of Proposition \ref{prop_g} guarantees that $V_{\rm left SW}(G)$ must be the source and $V_{\rm right SE}(G)$ must be the sink.
			It remains to show that there is no loop in $\Phi(M)$.
			Since every vertex in a loop are neither source nor sink, there must be an horizontal/diagonal edge oriented from right to left.
			This contradicts the orientation of $G$.
			So $\Phi(M)$ contains no loop.

			To show that $\Phi$ is a bijection, we construct another map $\Psi:\mathcal{P}\rightarrow\mathcal{M}$ by letting $\Psi(P)$ be the superposition of the path $P$ with $M_0$ and counting dimers on each edge modulo 2.
			In other words $\Psi(P)=(P\cup M_0)\setminus (P\cap M_0).$
			It is obvious that $\Psi\circ\Phi = \text{id}_\mathcal{M}$ and $\Phi\circ\Psi = \text{id}_\mathcal{P}.$ So, $\Psi = \Phi^{-1}$ provided that $\Psi$ is well-defined.
			To show that $\Psi(P)$ is a perfect matching of $G,$ it suffices to show that any vertex has exactly one incident edge in $\Psi(P)$.

			We first consider a black vertex.
			\begin{itemize}
				\item
					If it is in $V_{\rm left SW}(G)$, it has one incident edge in $P$ because it is a source.
					Also, it has no edge in $M_0$ since it is the left-most.
					So, it still has one incident edge in $\Psi(P)$ after the superposition.
				\item
					If it is not in $V_{\rm left SW}(G)$, it has either none or two edges in $P.$
					In both cases, since the vertex is not the left-most, there must be a white vertex on its left.
					So it has an incident edge in $M_0.$
				\begin{itemize}
					\item If it has no edge in $P,$ it will receive an edge from $M_0$ after mapping by $\Psi$.
					\item
						If it has two edges in $P,$ it is either of the case A, B or C in Figure \ref{fig_pathatblack}.
						So exactly one of the edges gets removed after superposing with $M_0.$
				\end{itemize}
				From both cases, the vertex is incident to exactly one edge in $\Psi(P)$.
			\end{itemize}
			A similar argument holds for a white vertex.
			So we conclude that $\Psi(P)\in \mathcal{M}$.
		\end{proof}
		
		\begin{prop}\label{prop:iff} 
			Let $e\in M_0.$
			Then $e\in M$ if and only $e \notin \Phi(M)$
		\end{prop}
		\begin{proof}
			Let $e\in M_0.$ If $e\in M$, then $e\in M\cap M_0$.
			Hence $e\notin (M\cup M_0)\setminus(M\cap M_0)=\Phi(M)$.
			On the other hand, if $e\notin M$, then $e\notin M\cap M_0$.
			Hence $e\in (M\cup M_0)\setminus(M\cap M_0)=\Phi(M).$
		\end{proof}

		Now we have analogs to Proposition \ref{prop:bijection} and Proposition \ref{prop:iff} for $\overline{G}.$
		
		\begin{prop}\label{prop:bijectionbar}
			Let $\overline{\Phi}$ be defined as in Definition \ref{defn:barphi}.
			Then we have the following:
			\begin{enumerate}
				\item $\overline{\Phi}:\overline{\mathcal{M}}\rightarrow\overline{\mathcal{P}}$ is a bijection,
				\item For $e\in \overline{M}_0,$ $e\in \overline{M}$ if and only $e \notin \overline{\Phi}(\overline{M}).$
			\end{enumerate}
		\end{prop}
		\begin{proof}
			We recall that the symmetric difference $\triangle$ is commutative and associative.
			Let $\overline{D} = \diag(E(\overline{G})\setminus E(G)).$
			Since $M\cap \overline{D}=\emptyset$ for $M\in\mathcal{M}$, $\mathcal{M}$ and $\overline{\mathcal{M}}$ are in bijection via the map:
			\[
				f:M\mapsto M\cup\overline{D} = M\triangle \overline{D},
			\]
			with the inverse map:
			\[
				f^{-1}:\overline{M}\mapsto \overline{M}\triangle \overline{D}.
			\]
			
			Consider $\overline{P}\in\overline{\mathcal{P}},$ we can see from Proposition \ref{prop:modifiedopenface} that the paths from $V_{\rm left SW}(\overline{G})$ must go to the right via horizontal or diagonal edges in $E(\G)\setminus E(G).$
			Similarly, the paths will arrive $V_{\rm right SE}(\overline{G})$ via horizontal or diagonal edges in $E(\G)\setminus E(G).$
			Since there is a unique choice of these edges, we have $ \overline{P} =  P\cup \overline{P}_0 $ where $\overline{P}_0\subset E(\overline{G})\setminus E(G)$ is the set of all horizontal and diagonal edges in the South-West, South or South-East.
			Note that $\overline{P}_0$ is also equal to the set of white-black horizontal and black-white diagonal edges in $E(\overline{G})\setminus E(G).$
			Since $P\cap \overline{P}_0=\emptyset$ for $P\in\mathcal{P}$, $\mathcal{P}$ and $\overline{\mathcal{P}}$ are in bijection via the map:
			\[
				g:P\mapsto P\cup \overline{P}_0 =P\triangle \overline{P}_0
			\]
			with the inverse map:
			\[
				g^{-1}:\overline{P}\mapsto \overline{P}\triangle \overline{P}_0.
			\]

			If we can show that $\overline{\Phi}=g\circ\Phi\circ f^{-1}$, (1) will automatically follows.
			To show this, we first show that that $\overline{M}_0 = \overline{D}\triangle M_0 \triangle \overline{P}_0$.
			Consider
			\[
				\overline{D}\triangle M_0 \triangle \overline{P}_0
				= \left( \overline{D} \triangle \overline{P}_0 \right) \triangle M_0 
				= (\overline{M}_0\setminus M_0) \triangle M_0 
				= (\overline{M}_0\setminus M_0) \cup M_0 
				= \overline{M}_0.
			\]
			From $\overline{M}_0 = \overline{D}\triangle M_0 \triangle \overline{P}_0$, we then have
			\[
				\overline{\Phi}(\overline{M}) = \overline{M} \triangle \overline{M}_0 =  \overline{M} \triangle (\overline{D}\triangle M_0 \triangle \overline{P}_0) = ((\overline{M}\triangle \overline{D}) \triangle M_0)\triangle \overline{P}_0 = (g\circ\Phi\circ f^{-1})(\overline{M})
			\]
			for all $\overline{M}\in\overline{\mathcal{M}}$.
			So, $\overline{\Phi}=g\circ\Phi\circ f^{-1}:\overline{\mathcal{M}}\rightarrow\overline{\mathcal{P}}$ is a bijection.

			For~(2), since $\overline{\Phi}:\overline{M}\mapsto\overline{M}\triangle \overline{M}_0,$ we have that for any $e\in\overline{M}_0,$
			\[
				e\in \overline{M}\iff e\in\overline{M}\cap\overline{M}_0\iff e \notin (\overline{M}\cup\overline{M}_0)\setminus(\overline{M}\cap\overline{M}_0)=\overline{\Phi}(\overline{M}).
			\]
			Hence we proved (2).
		\end{proof}

	\subsection{Modified edge-weight and nonintersecting-path solution}
	
		We recall the edge-weight $w_e$ in Definition \ref{defn:edgeweight}.
		It is compatible with the perfect-matching solution (Theorem \ref{thm:edgesol}).
		In order to construct a nonintersecting-path solution from the bijection $\overline{\Phi}$, the edge-weight $w_e$ requires some modification.
		Due to Proposition \ref{prop:bijectionbar}, we only need to inverse the weight of all edges in $\overline{M}_0$.

		\begin{defn}[Modified edge-weight $w_e'$]\label{defn:modifiededgeweight}
			For $x\in E(\overline{G})$ we define the \defemph{modified edge-weight} as follows:
			\[
				w_e'(x):=
				\begin{cases}
					w_e(x)^{-1}, & x\in\overline{M}_0,\\
					w_e(x), & \text{otherwise,}
				\end{cases}
			\]
			where $w_e$ is the edge-weight defined in Definition \ref{defn:edgeweight}.
		\end{defn}
		
		\begin{defn}[Modified edge-weight for paths in $\overline{G}$]
			For a path $p=x_1 x_2 \dots x_n$ in $\overline{G}$, its \defemph{modified edge-weight} is defined to be the following product
			\[
				w_e'(p) := \prod_{i=1}^n w_e'(x_i) .
			\]
			Then the weight for a non-intersecting path is defined by
			\[
				w_e'(\overline{P}) := \prod_{p\in\overline{P}} w_e'(p)\quad\text{for}\enspace\overline{P}\in\overline{\mathcal{P}}.
			\]
		\end{defn}
		
	Now we are ready for a nonintersecting-path solution for $\overline{G}$.
	Note that we can also consider the solution for $G$, but it turns out to be more complicated due to the present of open faces of $G$, which needs to be treated separately.
	
	\begin{thm}[Nonintersecting-path solution for $\overline{G}$]\label{thm_pathsol}
		Let $\k$ be an admissible initial data stepped surface with respect to $p=(i_0,j_0,k_0)$, $\overline{M}_0$ be defined as in Definition \ref{defn:sourcesink}.
		Then
		\begin{align*}
			T_{i_0,j_0,k_0} =\sum_{\overline{P}\in\mathcal{\overline{P}}} w_e'(\overline{P}) \Big/ \prod_{\circ\underset{b}{\text{---}}\bullet\in \overline{M}_0}\overline{p}_b,
		\end{align*}
		where $\overline{\mathcal{P}}$ is the set of all non-intersecting paths in $\G$ from $V_{\rm left SW}(\overline{G})$ to $V_{\rm right SE}(\overline{G})$.
	\end{thm}
	\begin{proof}
			Let $\overline{M}\in\overline{\mathcal{M}}$.
			From Definition \ref{defn:modifiededgeweight}, we have
			\[
				w_e'(\overline{M}\setminus\overline{M}_0)=w_e(\overline{M}\setminus\overline{M}_0) \text{ and } w_e'(\overline{M}_0\setminus\overline{M})=w_e(\overline{M}_0\setminus\overline{M})^{-1}.
			\]
			Since we can write the symmetric difference as the disjoint union $A\triangle B = (A\setminus B) \sqcup (B\setminus A)$, we have
			\begin{align*}
				w_e(\overline{M})
				&=w_e(\overline{M}\setminus\overline{M}_0)w_e(\overline{M}\cap\overline{M}_0) \\
				&=w_e(\overline{M}\setminus\overline{M}_0)w_e(\overline{M}_0\setminus\overline{M})^{-1}w_e(\overline{M}_0\setminus\overline{M})w_e(\overline{M}\cap\overline{M}_0) \\
				&=w_e'(\overline{M}\setminus\overline{M}_0)w_e'(\overline{M}_0\setminus\overline{M})w_e(\overline{M}_0) \\
				&=w_e'(\overline{M}\triangle \overline{M}_0)w_e(\overline{M}_0) \\
				&=w_e'(\overline{\Phi}(\overline{M}))w_e(\overline{M}_0).
			\end{align*}
			From Theorem \ref{thm:edgesol} and the previous equality, we have
			\begin{align*}
				T_{i_0,j_0,k_0}
				&= \sum_{\overline{M}\in\overline{\mathcal{M}}} w_e(\overline{M}) \Big/ w_e(\overline{M}_0)\big|_{c_{i,j}=1} \\
				&= \sum_{\overline{M}\in\overline{\mathcal{M}}} w_e'\left(\overline{\Phi}(\overline{M})\right) \frac{w_e(\overline{M}_0)}{ w_e(\overline{M}_0)\big|_{c_{i,j}=1}} \\
				&= \sum_{\overline{P}\in\overline{\mathcal{P}}} w_e'(\overline{P}) \Big/  \prod_{\circ\underset{b}{\text{---}}\bullet\in \overline{M}_0}\overline{p}_b.
			\end{align*}
			Hence we proved the theorem.
		\end{proof}
		
		The nonintersecting-path solution obtained in this section gives us a hint that it is possible to have a network solution analog to \cite{DFK13,DF14}.
		In the next section, we will transform the oriented graph $\G$ to a network.
		The modified edge-weight will be used on the network as well.
		This leads to a network solution and a network-matrix solution.


\section{Network solution}\label{sec_network}
	In this section, we will construct a weighted directed network $N$ associated with the oriented graph $\overline{G}$ and the modified edge-weight $w_e'$.
	We then decompose $N$ into network chips and their associated elementary network matrices (adjacency matrices).
	The product of all the elementary matrices associated with the network chips, according to an order of the chips, is then called the network matrix associated to $N$.
	The nonintersecting-path solution for $\G$ (Theorem \ref{thm_pathsol}) can then be interpreted as a path solution on this network, and can also be computed from a certain minor of the network matrix.
	We also show that our network and elementary network matrices coincide with the objects studied in \cite{DFK13} in the case of coefficient-free T-systems ($c_{i,j}=1$ for all $(i,j)\in\Z^2$).
	
	\subsection{Network associated with a graph}
		We construct the \defemph{directed network} $N$ associated with the oriented graph $\G$ by tilting all the diagonal edges so that they become horizontal, and tilting all the vertical edges so that the vertex on the left is black as shown in Figure \ref{fig_flatG}.
		Also, the network is directed from left to right.
		So a path in the oriented graph $\overline{G}$ (Definition \ref{defn:edgeoreient}) corresponds to a path in $N$.
		
		We then introduce the notion of ``row" for vertices of $N$.
		Two vertices are said to be in the same row if they are joined by a connected path of horizontal or diagonal edges.
		In other words, they are on the left and right of each other in $\G$, see Definition \ref{defn:leftright}.
		We number the rows so that they increase by one from bottom to top and the center face $(i_0,j_0)$ of $N$ lies between the row $-1$ and $0$. See Figure \ref{fig_flatG}.
		The precise definition is as the following.
		
		\begin{defn}\label{defn:row}
			The vertex $v\in V(\overline{G})$ is in \defemph{row} $r$ if two of its incident faces are $(i,j_0+r)$ and $(i,j_0+r+1)$ for some $i\in\Z$.
		\end{defn}
		
		\begin{figure}%
			\begin{center}\includegraphics[scale=0.5]{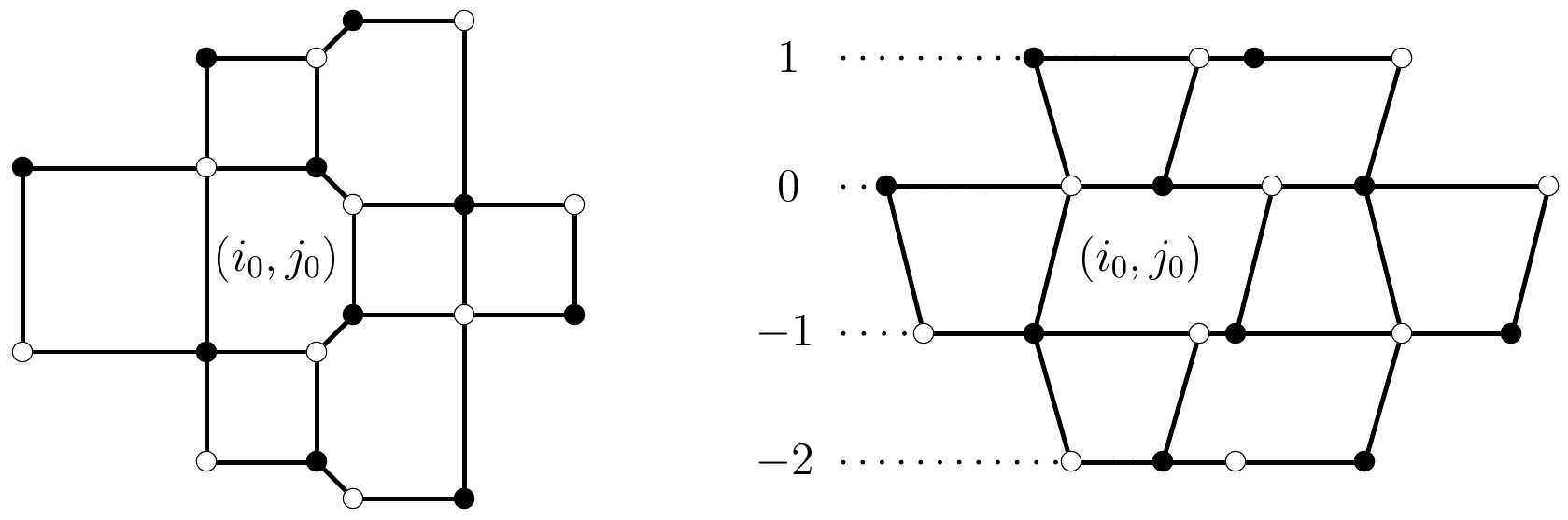}\end{center}
			\caption{An example of $\overline{G}$ and its associated network $N$.}
			\label{fig_flatG}%
		\end{figure}
		
		Let $N$ be the directed network associated with $\G$, where $r_{\rm min}$ and $r_{\rm max}$ are the smallest and the largest row numbers of $N$.
		We then put weight on the network locally around each black vertex as shown in Figure \ref{fig_tablenetwork}.
		The weight comes directly from the modified edge-weight $w_e'$ on $\G$ (Definition \ref{defn:modifiededgeweight}), so we will carry the notation $w_e'$ for the weight on $N$.
		
		\begin{remark}
			When a black vertex in Figure \ref{fig_tablenetwork} is on the boundary, we will assume the following.
			\begin{align*}
				t_a =1, \text{ when }a\notin F(\G),\\
				p_b =\overline{p}_b =1, \text{ when }b\notin F(G).
			\end{align*}
			This assumption comes directly from Definition \ref{defn:edgeweight}.
		\end{remark}
		
		\begin{table}%
			\begin{center}\begin{tabular}{ccc}
\hline
Part of $\G$ & Network chip & Network matrix \\
\hline
\raisebox{-.5\height}{\includegraphics[scale=1]{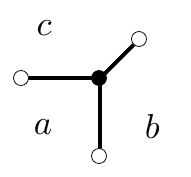}} & \raisebox{-.5\height}{\includegraphics[scale=1.1]{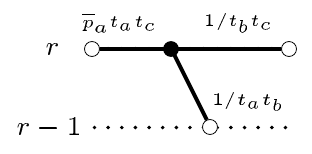}} & $U_r(t_a,t_b,t_c)$\\
\hline
\raisebox{-.5\height}{\includegraphics[scale=1]{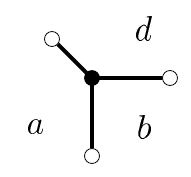}} & \raisebox{-.5\height}{\includegraphics[scale=1.1]{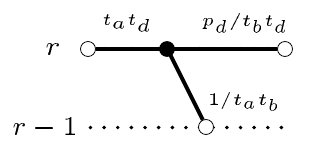}} & $U'_r(t_a,t_b,t_d)$\\
\hline
\raisebox{-.5\height}{\includegraphics[scale=1]{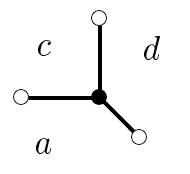}} & \raisebox{-.5\height}{\includegraphics[scale=1.1]{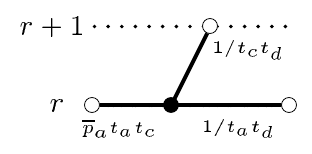}} & $V_r(t_a,t_c,t_d)$\\
\hline
\raisebox{-.5\height}{\includegraphics[scale=1]{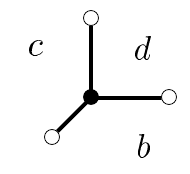}} & \raisebox{-.5\height}{\includegraphics[scale=1.1]{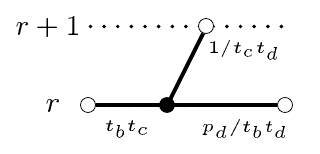}} & $V'_r(t_b,t_c,t_d)$\\
\hline
\raisebox{-.5\height}{\includegraphics[scale=1]{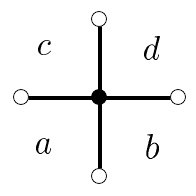}} & \raisebox{-.5\height}{\includegraphics[scale=1.1]{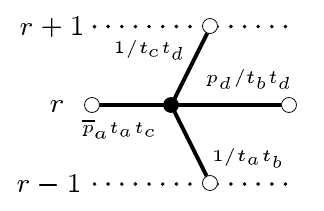}} & $W_r(t_a,t_b,t_c,t_d)$\\
\hline
\end{tabular}\end{center}
			\caption{
				A table comparing a part of $\G$ around a black vertex, its corresponding elementary network chip in $N$ and its corresponding elementary network matrix defined in Definition \ref{defn:elementarymatrix}
			}
			\label{fig_tablenetwork}%
		\end{table}
		
	\subsection{Nonintersecting-path solution for the network}
		We have already defined the weighted directed network $N$ associated with $\G$, where the weight on $N$ comes directly from the weight on $\G$.
		The nonintersecting-path solution for $\G$ (Theorem \ref{thm_pathsol}) can then be interpreted in terms of nonintersecting paths in $N$.
		The left-most SW vertices of $\G$ correspond to the left-most vertices in the rows $[r_{\rm min},-1]:=\{r_{\rm min},\dots,-2,-1 \}$ of $N$ while the right-most SE vertices of $\G$ corresponds to the right-most vertices in rows $[r_{\rm min},-1]$ of $N$.
		We then get the following theorem.
		
		\begin{thm}[Nonintersecting-path solution for $N$]\label{thm:network}
			Let $N$ be the weighted directed network associated with $\G$.
			Then
			\[
				T_{i_0,j_0,k_0} = Q^{-1} \sum_{P} w_e'(P),
			\]
			where the sum runs over all the non-intersecting paths on $N$ starting from the rows $[r_{\rm min},-1]$ on the left to $[r_{\rm min},-1]$ on the right and $Q :=\displaystyle\prod_{\substack{\circ\text{---}\bullet\\b}\in \overline{M}_0}\overline{p}_b$.
		\end{thm}
		
	\subsection{Network matrix}\label{subsec:networkmatrix}
		We will now decompose $N$ into \defemph{elementary network chips}, which are small pieces of $N$ around black vertices illustrated in Figure \ref{fig_tablenetwork}.
		This can be done by breaking the network at all of its white vertices.
		See Figure \ref{fig_networkdecomp} for an example.
		We notice that we will have the following choice when there are two incident vertical edges at a white vertex.
		\begin{center}
			\includegraphics[scale=0.8]{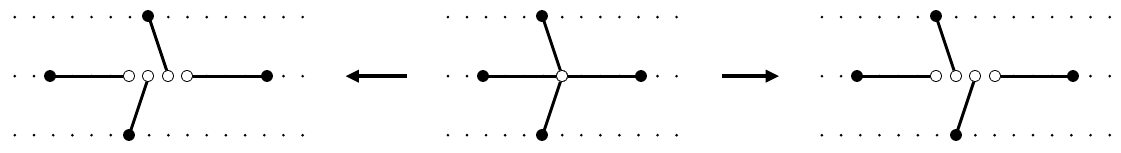}
		\end{center}
		Since the weight of a path on $N$ is independent of this choice, the nonintersecting-path solution is independent of this ambiguity.
		
		\begin{figure}%
			\begin{center}\includegraphics[scale=0.9]{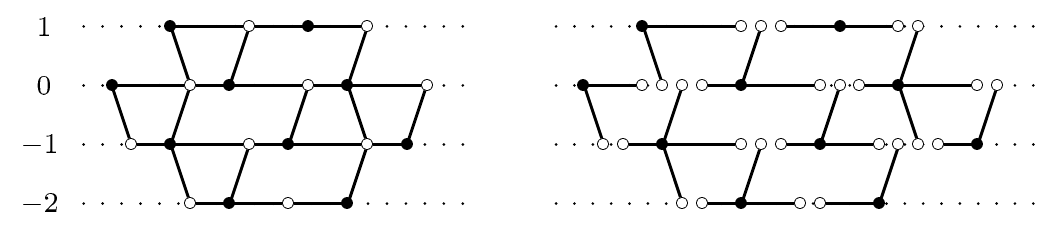}\end{center}
			\caption{A network is decomposed into elementary network chips}%
			\label{fig_networkdecomp}%
		\end{figure}
				
		A decomposition gives a partially-order on the network chips, we then pick a ``finer'' totally-order which does not contradict the partially-order.
		This totally-order can be thought as the order in which to pull chips out of the network from the left.
		
		\begin{ex}
			From the example in Figure \ref{fig_flatG}, we can pick a network decomposition as in Figure \ref{fig_networkdecomp}, and then pick a totally-orders as in Figure \ref{fig_sixorder}.
		\end{ex}
		
		\begin{figure}%
			\begin{center}\includegraphics[scale=0.7]{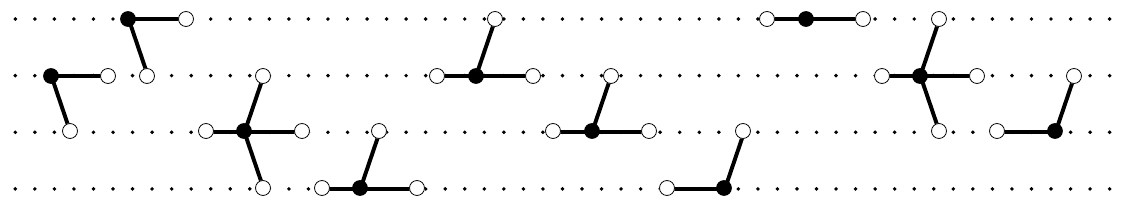}\end{center}
			\caption{A totally-order from a decomposition in Figure \ref{fig_networkdecomp}}%
			\label{fig_sixorder}%
		\end{figure}
		
		The next step is to associate each chip with an elementary network matrix shown in Figure \ref{fig_tablenetwork}.
    The matrices $W_r$, $V_r$, $V'_r$, $U_r$ and $U'_r$ are defined in the following definition.
    \begin{defn}\label{defn:elementarymatrix}
    We define an \defemph{elementary network matrix} associated with a network $N$, depending on its configuration around a black vertex in Figure \ref{fig_tablenetwork}, to be an square matrix of size $r_{\rm max}-r_{\rm min}+1$ with entries:
    \begin{align*}
      (U_r(x,y,z))_{\alpha,\beta}&:=
        \begin{cases}
          (U(x,y,z))_{\alpha-r+2,\beta-r+2}, &\text{if }\alpha,\beta\in\{r-1,r\},\\
          \delta_{\alpha,\beta}, &\text{otherwise,}
        \end{cases}\\
      (V_r(x,y,z))_{\alpha,\beta}&:=
        \begin{cases}
          (V(x,y,z))_{\alpha-r+1,\beta-r+1}, &\text{if }\alpha,\beta\in\{r,r+1\},\\
          \delta_{\alpha,\beta}, &\text{otherwise,}
        \end{cases}\\
      (W_r(w,x,y,z))_{\alpha,\beta}&:=
        \begin{cases}
          (W(w,x,y,z))_{\alpha-r+2,\beta-r+2}, &\text{if }\alpha,\beta\in\{r-1,r,r+1\},\\
          \delta_{\alpha,\beta}, &\text{otherwise,}
        \end{cases}
    \end{align*}
    and $U'_r$ (resp. $V'_r$) is defined in the same way as $U_r$ (resp. $V_r$) where
    \begin{align*}
      U(t_a,t_b,t_c)&:=
        \begin{pmatrix}
          1&0 \\ 
          \overline{p}_a\frac{t_c}{t_b} & \overline{p}_a \frac{t_a}{t_b}
        \end{pmatrix},\qquad
      U'(t_a,t_b,t_d):=
        \begin{pmatrix}
          1&0 \\ 
          \frac{t_d}{t_b} & p_d \frac{t_a}{t_b}
        \end{pmatrix},\\
      V(t_a,t_c,t_d)&:=
        \begin{pmatrix}
          \overline{p}_a\frac{t_c}{t_d} & \overline{p}_a \frac{t_a}{t_d}\\
          0&1
        \end{pmatrix},\qquad
      V'(t_b,t_c,t_d):=
        \begin{pmatrix}
          p_d\frac{t_c}{t_d} &  \frac{t_b}{t_d} \\ 
          0&1
        \end{pmatrix},\\
      W(t_a,t_b,t_c,t_d)&:=
        \begin{pmatrix}
          1&0&0\\ 
          \overline{p}_a \frac{t_c}{t_b} & \overline{p}_ap_d\frac{t_at_c}{t_bt_d} & \overline{p}_a\frac{t_a}{t_d} \\ 
          0&0&1
        \end{pmatrix}.
    \end{align*}
    The $(i,j)$ entry of a network matrix is just the weight of the edge from row $i$ to row $j$ in the network chip, see Figure \ref{fig_tablenetwork}.
    \end{defn}

		If two network chips do not have an order from the decomposition, then their corresponding elementary network matrices commute.
		For a totally-order of the network chips, we define the product of all the elementary network matrices with the same order of the chips.
		This product is independent of a choice of totally-order and decomposition.
		We call it the \defemph{network matrix} associated with $N$.
		
		\begin{remark}\label{rem:collapsingvscommuting}
			The network matrix $W_r$ defined in Definition \ref{defn:elementarymatrix} can be factored as follows:
			\[
				V_r(t_a,t_c,t_d)U'_r(t_a,t_b,t_d)=W_r(t_a,t_b,t_c,t_d)=U_r(t_a,t_b,t_c)V'_r(t_b,t_c,t_d).
			\]
			This corresponds to the following picture when collapsing a degree-2 vertex.
			\begin{center}
				\includegraphics[scale=0.8]{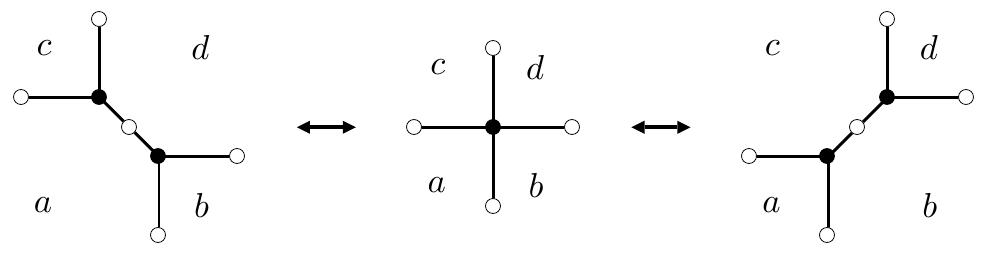}
			\end{center}
		\end{remark}
		
		\begin{remark}\label{rem:UVfactor}
			$U_r,U'_r,V_r$ and $V'_r$ defined in Definition \ref{defn:elementarymatrix} can also be factored.
			This factorization separates the coefficients $\overline{p}_a$'s and $p_d$'s from the variables $t_\alpha$'s.
			\begin{align*}
				U(t_a,t_b,t_c) &=
					\begin{pmatrix}
						1&0 \\
						0&\overline{p}_a
					\end{pmatrix}
					\begin{pmatrix}
						1&0 \\
						\frac{t_c}{t_b}&\frac{t_a}{t_b}
					\end{pmatrix}, \quad
				U'(t_a,t_b,t_d) =
					\begin{pmatrix}
						1&0 \\
						\frac{t_d}{t_b}&\frac{t_a}{t_b}
					\end{pmatrix}
					\begin{pmatrix}
						1&0 \\
						0&p_d
					\end{pmatrix},\\
				V(t_a,t_c,t_d) &=
					\begin{pmatrix}
						\overline{p}_a&0 \\
						0&1
					\end{pmatrix}
					\begin{pmatrix}
						\frac{t_c}{t_d}&\frac{t_a}{t_d} \\
						0&1
					\end{pmatrix},\quad
				V'(t_b,t_c,t_d) =
					\begin{pmatrix}
						\frac{t_c}{t_d}&\frac{t_b}{t_d} \\
						0&1 
					\end{pmatrix}
					\begin{pmatrix}
						p_d&0 \\
						0&1
					\end{pmatrix}.
			\end{align*}
		\end{remark}
		
		Since the entry $(i,j)$ of the network matrix is the partition function of weighted paths from row $i$ to row $j$, by the Lindstr\"{o}m-Gessel-Viennot theorem \cite{Lindstrom,GV}, the partition function of weighted non-intersecting paths from the rows $[r_{\rm min},-1]$ on the left to the rows $[r_{\rm min},-1]$ to the right is the principal minor of the network matrix corresponding to the rows/columns $[r_{\rm min},-1]$.
		Theorem \ref{thm:network} then gives the following.

		\begin{thm}[Network-matrix solution]\label{thm:networkmatrix}
			Let $M$ be the network matrix associated with the network $N$ of the graph $\G$. Then
			\[
				T_{i_0,j_0,k_0} =  Q^{-1} |M|_{r_{\rm min},\dots,-2,-1}^{r_{\rm min},\dots,-2,-1},
			\]
			where $|M|_{r_{\rm min},\dots,-2,-1}^{r_{\rm min},\dots,-2,-1}$ denotes the principal minor of $M$ corresponding to the rows $[r_{\rm min},-1]$ and columns $[r_{\rm min},-1]$, and $Q :=\displaystyle\prod_{\substack{\circ\text{---}\bullet\\b}\in \overline{M}_0}\overline{p}_b$.
		\end{thm}
		
		We may absorb the factor $Q^{-1}$ into the elementary network matrices, by defining the \defemph{modified elementary network matrices} to be as follows.
		\begin{align}
			\begin{aligned}
				\overline{U}_r(t_a,t_b,t_c) &:= (\overline{p}_a)^{-1} U_r(t_a,t_b,t_c),\quad &
				\overline{U}'_r(t_a,t_b,t_d) &:=  U'_r(t_a,t_b,t_d),\\
				\overline{V}_r(t_a,t_c,t_d) &:= (\overline{p}_a)^{-1} V_r(t_a,t_c,t_d),\quad &
				\overline{V}'_r(t_b,t_c,t_d) &:=  V'_r(t_b,t_c,t_d),\\
				\overline{W}_r(t_a,t_b,t_c,t_d) &:= (\overline{p}_a )^{-1}W_r(t_a,t_b,t_c,t_d).
			\end{aligned}
		\end{align}
		Then the \defemph{modified network matrix} associated with $N$ is defined as the product of all modified elementary network matrices according to a totally-order of the network chips as before.
		Theorem \ref{thm:networkmatrix} becomes the following.
		
		\begin{thm}[Modified-network-matrix solution]\label{thm:modifiednetworkmatrix}
			Let $\overline{M}$ be the modified network matrix associated with the network $N$ of the graph $\G$.
			Then
			\[
				T_{i_0,j_0,k_0} = |\overline{M}|_{r_{\rm min},\dots,-2,-1}^{r_{\rm min},\dots,-2,-1},
			\]
			where $|\overline{M}|_{r_{\rm min},\dots,-2,-1}^{r_{\rm min},\dots,-2,-1}$ is the minor of $\overline{M}$ corresponding to the rows $[r_{\rm min},-1]$ and columns $[r_{\rm min},-1]$.
		\end{thm}
		
		\begin{remark}[Flatness condition]
			It is worth pointing out that the condition
			\begin{align}
				\overline{U}'_{k+1}(t_\ell,t_c,t_u)\overline{V}_{k}(t_d,t_c,t_r) = \overline{V}'_k(t_d,t_\ell,t_c')\overline{U}_{k+1}(t_c',t_r,t_u)
				\label{eq:flatcond}
			\end{align}
			is equivalent to
			\[
				t'_c t_{c} = J_{i,j,k} t_\ell t_r + t_u t_d
			\]
			where $c=(i,j)$, $\ell=(i-1,j)$, $r=(i+1,j)$, $u=(i,j+1)$, $d=(i,j-1)$ and $k=\k(c)-1=\k(\ell)=\k(r)=\k(u)=\k(d)$.
			It corresponds to the following action on networks.
			\begin{center}
				\includegraphics[scale=0.8]{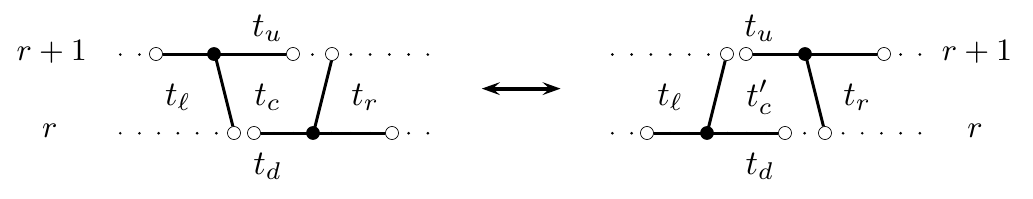}
			\end{center}
			In the graph $\G$, this is the urban renewal (Figure \ref{fig:urban}) at the face $c$ on the graph $\G$.
			\begin{center}
				\includegraphics[scale=0.8]{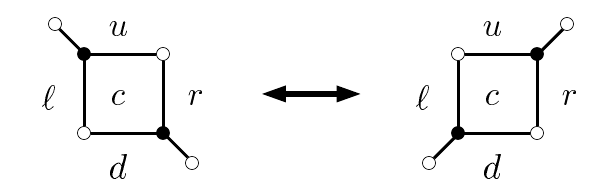}
			\end{center}
      On the stepped surface $\k$, this is a downward-mutation at $c$.
			Elementary network matrices therefore encode the octahedron recurrence \eqref{eq:tsyscoef} as the flatness condition of \eqref{eq:flatcond}.
		\end{remark}
		
		\subsection{Lozenge covering}
			In \cite{DFK13} the authors construct a lozenge covering from a stepped surface and  use it to construct a weighted network.
			Then the solution to the coefficient-free T-system is a partition function of non-intersecting paths in the network.
			
			We recall the procedure in \cite{DFK13} with some modification so that it fits in our setting.
			Starting from a stepped surface $\k$, a lozenge covering is constructed depending on the height of the four corners of a square: $(i,j)$, $(i+1,j)$, $(i,j+1)$ and $(i+1,j+1)$.
			The rule can be summarized in the following table.
			\begin{center}
				\includegraphics[scale=0.75]{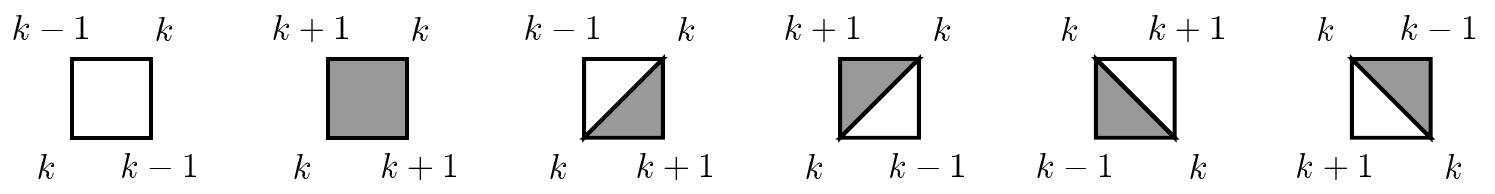}
			\end{center}
			We note that a choice of triangulating the squares in the first two cases makes a lozenge covering not unique.
			
			Next, we restrict the covering to the points $(i,j)$'s in $\mathring{F}\cup\partial F$.
			This gives a triangulation of a finite region.
			The next step is to group two triangles sharing a horizontal side together.
			As a result, we obtain ``elementary chips'' as follows.
			\begin{center}
				\includegraphics[scale=0.75]{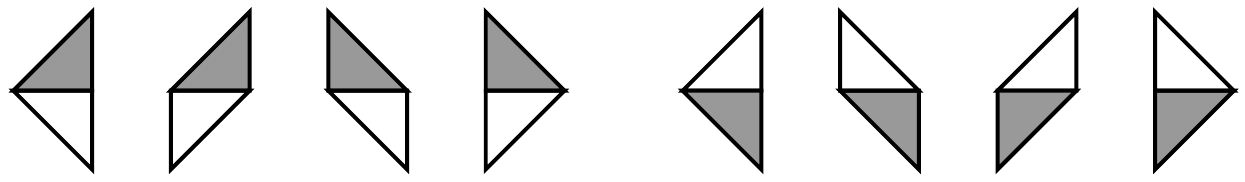}
			\end{center}
			They are classified into two types depending on whether the shaded triangle is on the top or the bottom of the lozenge.
			An ``elementary matrix'' is then associated with each type depending on the corners of the shaded triangle of a lozenge as follows.
			\begin{center}
				\includegraphics[scale=0.75]{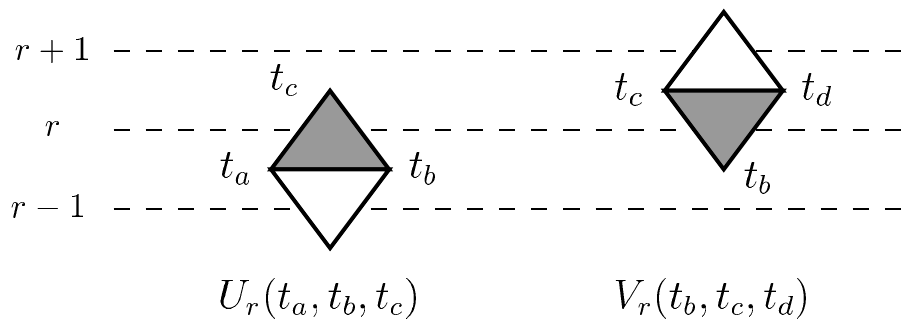}
			\end{center}
			We recall the notion of row defined in Definition \ref{defn:row}.
			The matrix $U_r(x,y,z)$ and $V_r(x,y,z)$ are square matrices of size $(r_{\rm max}-r_{\rm min}+1)\times (r_{\rm max}-r_{\rm min}+1)$ with entries
			\begin{align*}
				(U_r(x,y,z))_{\alpha,\beta}=
					\begin{cases}
						(U(x,y,z))_{\alpha-r+2,\beta-r+2}, &\text{if }\alpha,\beta\in\{r-1,r\},\\
						\delta_{\alpha,\beta}, &\text{otherwise},
					\end{cases}\\
        (V_r(x,y,z))_{\alpha,\beta}=
					\begin{cases}
						(V(x,y,z))_{\alpha-r+1,\beta-r+1}, &\text{if }\alpha,\beta\in\{r,r+1\},\\
						\delta_{\alpha,\beta}, &\text{otherwise},
					\end{cases}
			\end{align*}
			and similarly for $V_r(x,y,z)$.
			The $2\times 2$ matrices $U(x,y,z)$ and $V(x,y,z)$ are defined as
			\[
				U(x,y,z)=\begin{pmatrix}1&0\\\frac{z}{y}&\frac{x}{y}\end{pmatrix},\quad
				V(x,y,z)=\begin{pmatrix}\frac{y}{z}&\frac{x}{z}\\0&1\end{pmatrix}.
			\]
      We can immediately see that this is a specialization of the elementary network matrices in Definition \ref{defn:elementarymatrix} to the case when $c_{i,j}=1$ for all $(i,j)\in\Z^2$.
			With the same idea as the network-chip decomposition in Section \ref{subsec:networkmatrix}, the lozenge covering is decomposed.
			This gives an ordering of elementary network matrices.
			We then construct a ``network matrix'', which is a product of all elementary network matrices according to the order.
			The cluster variable $T_{i_0,j_0,k_0}$ is then expressed as a certain minor of the network matrix, similar to Theorem \ref{thm:networkmatrix} and Theorem \ref{thm:modifiednetworkmatrix}.
      
      \begin{remark}
        In the original definition of $U_r$ and $V_r$ in \cite{DFK13}, the index of $U$ is shifted by 1, i.e. $\mathcal{U}_r = U_{r+1}$ where $\mathcal{U}_r$ is the network matrix defined in \cite{DFK13}.
      \end{remark}
			
			In our construction, the existence of coefficients forces us to make a finer classification of the elementary network matrices as follows:
      \begin{center}
        \includegraphics[scale=0.75]{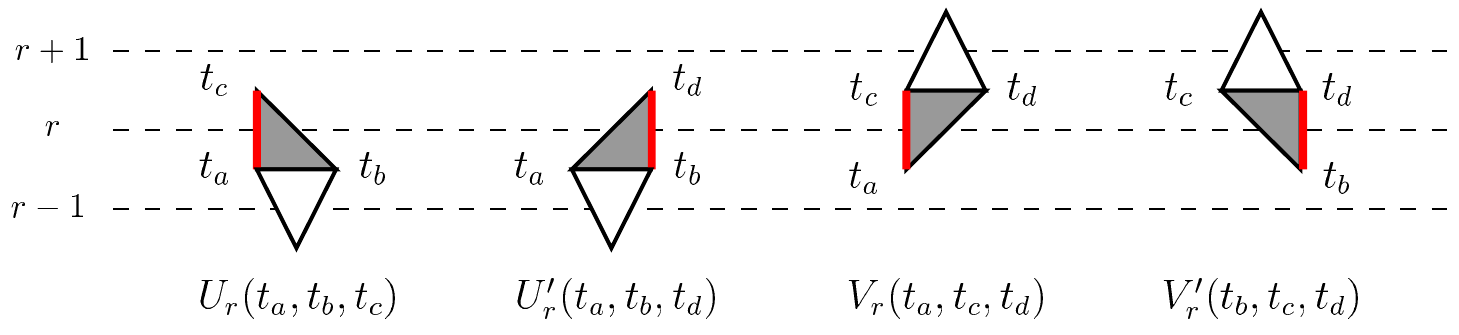}
      \end{center}
      From Remark \ref{rem:UVfactor}, we can think that the coefficients live on the vertical sides of shaded triangles. From Remark \ref{rem:collapsingvscommuting}, the matrix $W_r$ can be realized as a combination of two lozenges with the following choice of decomposition.
      \begin{center}
        \includegraphics[scale=0.75]{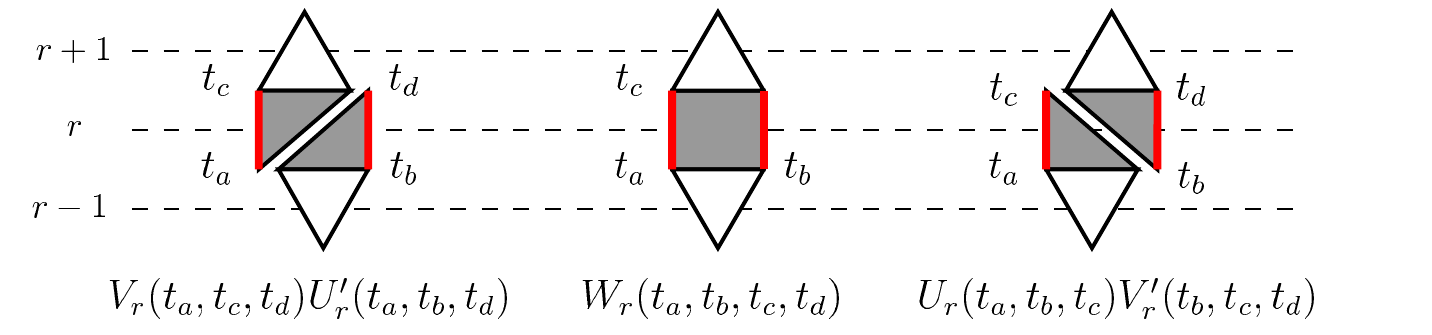}
      \end{center}
      
      The main reason that the lozenge covering has a rich connection to our story is because it is indeed a dual of our bipartite graph.
			This can be locally described as follows:
			\begin{center}
				\includegraphics[scale=0.75]{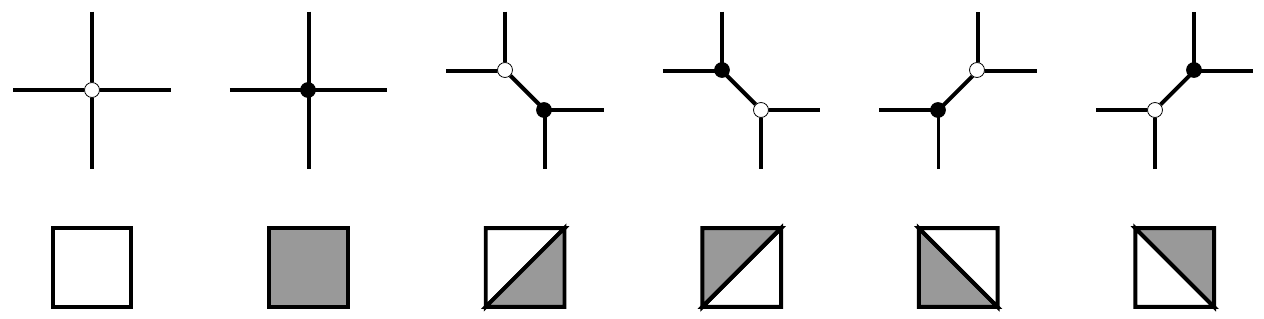}
			\end{center}
			A choice to triangulate a white square corresponds to a choice of decomposing a white vertex, and a choice to triangulate a black square corresponds to a choice of collapsing a degree-2 white vertex in Remark \ref{rem:collapsingvscommuting}.
			These can be illustrated by the following pictures:
			\begin{center}
				\includegraphics[scale=0.75]{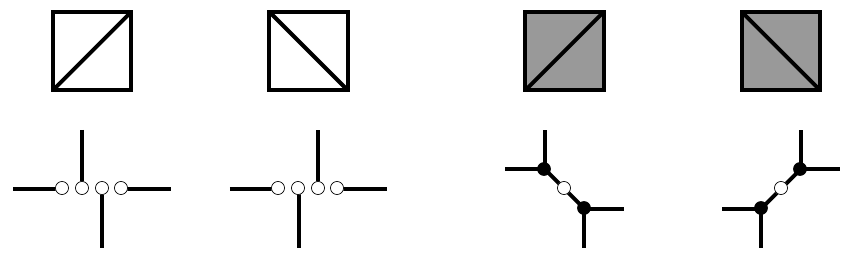}
			\end{center}
			
			We have provided various solutions to the T-system with principal coefficients.
			These solutions give combinatorial expressions of $T_{i_0,j_0,k_0}$ in terms of coefficients $c_{i,j}$'s and initial data $t_{i,j}$'s on $\k$ under the conditions:
			\[ k_0\geq \k(i_0,j_0) ,\quad\k(i,j)\geq \fund(i,j)\enspace\text{for }(i,j)\in\Z^2.\]
			We will discuss some other cases in Section \ref{sec_General}.


\section{Other coefficients}\label{sec_app}
In this section, we discuss a few examples of other choices of coefficients on T-systems: Speyer's octahedron recurrence \cite{Speyer}, generalized lambda-determinants \cite{DF13} and (higher) pentagram maps \cite{Glick,GSTV14}.
Almost all of them have their own explicit combinatorial solutions and are treated with different techniques.
Applying Theorem \ref{thm:sep} and Theorem \ref{thm:YasF} to our solutions, we get a partial solution to each of them when the initial data stepped surface is $\fund$.

In \cite{Speyer}, the author provides a perfect matching solution to the Speyer's octahedron recurrence, which is a partition function of perfect matchings of $G$.
In fact, our perfect matching solution is developed from the method used in the paper.
The author uses face-weight (same as Definition \ref{defn:faceweight}) which gives cluster variables and edge-weight (instead of our pairing-weight) which gives cluster coefficients.
So the main difference is that the edge-weight is specific to a choice of coefficients.

For the generalized lambda-determinant in \cite{DF13}, the author provides a network solution which is a partition function of non-intersecting paths in a weighted directed network.
This network is the same as the network discussed in Section \ref{sec_network}.
However, the weight used in \cite{DF13} is different to our weight due to the choice of coefficients.

The (higher) pentagram maps \cite{Schwartz,OST,Glick,GSTV14} can be realized as a Y-pattern, i.e. a dynamic on cluster coefficients, not cluster variables.
So it is not directly a T-system, but is instead a Y-pattern on the octahedron quiver.
In \cite{Glick}, the author gives a combinatorial solution to the pentagram map using alternating sign matrices.

\subsection{Speyer's octahedron recurrence}
	The \defemph{Speyer's octahedron recurrence}, defined with coefficients, is a recurrence relation \cite{Speyer} on the set of formal variables $\mathcal{T}^{(s)}=\{T^{(s)}_{i,j,k} \mid (i,j,k)\in\Zodd\}$ together with a set of extra variables, called coefficients, $\{A_{i,j},B_{i,j},C_{i,j},D_{i,j} \mid (i,j)\in\Z^2_{\rm even}\}$ satisfying
	\begin{align}
		T^{(s)}_{i,j,k-1}T^{(s)}_{i,j,k+1} =B_{i,j+k}D_{i,j-k} T^{(s)}_{i-1,j,k}T^{(s)}_{i+1,j,k} + A_{i+k,j}C_{i-k,j} T^{(s)}_{i,j-1,k}T^{(s)}_{i,j+1,k}
		\label{eq:speyer}
	\end{align}
	We note that in \cite{Speyer}, the condition on the index $(i,j,k)$ of $T^{(s)}_{i,j,k}$ is $i+j+k\equiv 0 \bmod 2$, and the coefficients are defined on $\Z^2_{\rm odd}$.
	With a shift in the indices, we make it coherent with our construction.
	
	Speyer's octahedron recurrence can also be interpreted \cite{Speyer} as a cluster algebra with coefficients.
	Its initial quiver is the octahedron quiver with initial cluster variables $T^{(s)}_{i,j,\fund(i,j)}$ similar to what we discussed in Section \ref{subsec:Tsys}.
	The only difference is the initial coefficients:
	\[
		y_{i,j}=
		\begin{cases}
			{B_{i,j}D_{i,j}}/{A_{i,j}C_{i,j}}, &i+j\equiv 0\bmod 2, \\
			{A_{i-1,j}C_{i+1,j}}/{B_{i,j-1}D_{i,j+1}}, &i+j\equiv 1\bmod 2,
		\end{cases}
	\]
	in the semifield $\P=\Trop(A_{i,j},B_{i,j},C_{i,j},D_{i,j} : (i,j)\in\Z^2_{\rm even})$.
	By Remark \ref{rem:GeometricFrozen}, it can be interpreted as a coefficient-free cluster algebra with frozen variables $\{A_{i,j},B_{i,j},C_{i,j},D_{i,j} \mid (i,j)\in\Z^2_{\rm even}\}$ with the quiver illustrated in Figure \ref{fig_Speyerquiver}.
	
	\begin{figure}
		\begin{center}\includegraphics[scale=0.8]{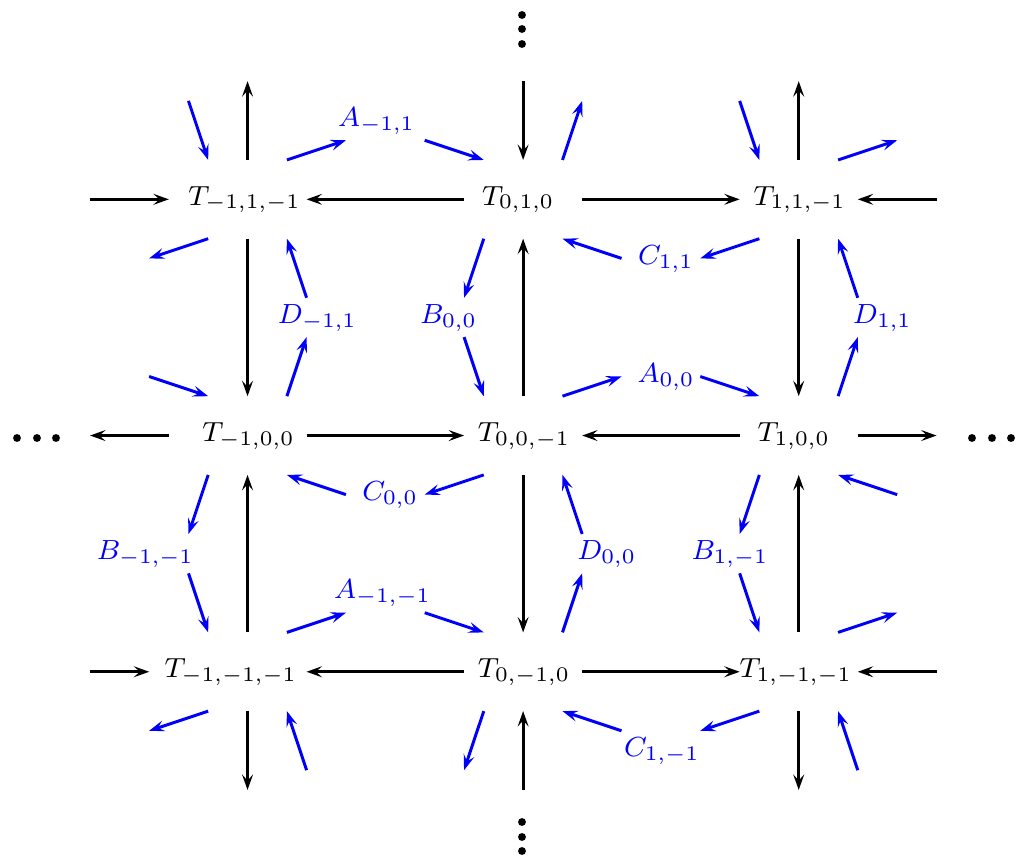}\end{center}
		\caption{
			A portion of the infinite quiver of Speyer's octahedron recurrence.
			The frozen variables and their incident arrows are in blue.
		}
		\label{fig_Speyerquiver}
	\end{figure}

	Since we will only consider the initial stepped surface $\fund$, we let $T^{(s)}_{i_0,j_0,k_0}$ denote its expression in terms of the initial data $t_{i,j}:=T^{(s)}_{i,j,\fund(i,j)}$ and the coefficients $A_{i,j},B_{i,j},C_{i,j},D_{i,j}$ for $(i,j)\in\Z^2$.
	For the T-system with principal coefficients, $T_{i_0,j_0,k_0}$ denotes its expression in terms of $t_{i,j}:=T_{i,j,\fund(i,j)}$ and $c_{i,j}$ for $(i,j)\in\Z^2$.
	In order to get $T^{(s)}_{i_0,j_0,k_0}$, we will have to specialize values of $t_{i,j}$ and $c_{i,j}$ in $T_{i_0,j_0,k_0}$ according to Theorem \ref{thm:sep}.
	Let $T_{i_0,j_0,k_0}(c_{i,j}=y_{i,j})$ denote the expression of $T_{i_0,j_0,k_0}$ where $t_{i,j}$ stayed untouched but $c_{i,j}$ is substituted by $y_{i,j}$.
	$T_{i_0,j_0,k_0}|_{\P}(t_{i,j}=1;c_{i,j}=y_{i,j})$ denotes the expression of $T_{i_0,j_0,k_0}$ where $t_{i,j}$ is set to $1$, $c_{i,j}$ is set to $y_{i,j}$, and then the whole expression is finally computed in $\P$.
	
	By the separation formula (Theorem \ref{thm:sep}), we get a solution to the Speyer's octahedron recurrence from the solution to the T-system with principal coefficients:
	\begin{align}
		T^{(s)}_{i_0,j_0,k_0} = \frac{T_{i_0,j_0,k_0}(c_{i,j}=y_{i,j})}{T_{i_0,j_0,k_0}|_{\P}(t_{i,j}=1;c_{i,j}=y_{i,j})}
		\label{eq:speyerfromprinc}
	\end{align}
	where $\P=\Trop(A_{i,j},B_{i,j},C_{i,j},D_{i,j} : (i,j)\in\Z^2_{\rm even})$ and
	\begin{align}
		y_{i,j}=
		\begin{cases}
			{B_{i,j}D_{i,j}}/{A_{i,j}C_{i,j}}, &i+j\equiv 0\bmod 2, \\
			{A_{i-1,j}C_{i+1,j}}/{B_{i,j-1}D_{i,j+1}}, &i+j\equiv 1\bmod 2.
		\end{cases}
		\label{eq:speyercoef}
	\end{align}

	We now compare our result to the solution in \cite{Speyer}.
	From our perfect matching solution (Theorem \ref{thm:main}), we then have
	\begin{align}
		T^{(s)}_{i_0,j_0,k_0} = \frac{\sum_{M\in\mathcal{M}} w_p(M)w_f(M)}{\bigoplus_{M\in\mathcal{M}} w_p(M)}
		\label{eq:mathchingsolforspeyer}
	\end{align}
	where $\mathcal{M}$ is the set of perfect matchings of the graph $G_{p,\fund}$.
	The denominator is a sum in $\P=\Trop(A_{i,j},B_{i,j},C_{i,j},D_{i,j} : (i,j)\in\Z^2_{\rm even})$, hence a monomial in $\{A_{i,j},B_{i,j},C_{i,j},D_{i,j}\mid (i,j)\in\Z^2_{\rm even}\}$.

	\begin{thm}[\cite{Speyer}]\label{thm:Speyer}
		For a point $p=(i_0,j_0,k_0)$ and an admissible initial stepped surface $\k$, we have
		\[
			T^{(s)}_{i_0,j_0,k_0} =  \sum_{M\in\mathcal{M}} w_s(M)w_f(M)
		\]
		where the sum runs over all the perfect matchings of $G=G_{p,\k}$. The weight $w_s$ is defined by $w_s(M):=\prod_{x\in M} w_s(x)$ and the weight $w_s(x)$ is defined for $x\in E(G)$ as follows:
		\begin{center}
			\includegraphics[scale=0.75]{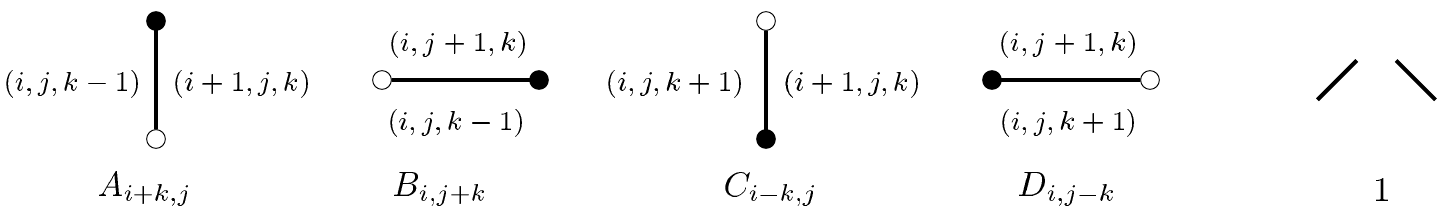}
		\end{center}
	\end{thm}

	Comparing \eqref{eq:mathchingsolforspeyer} to Theorem \ref{thm:Speyer}, we can write $w_s$ on a perfect matching on $G_{p,\fund}$ in terms of the pairing weight as follows:
	\begin{align}
		w_s(M) = \frac{ w_p(M)}{ \bigoplus_{M\in\mathcal{M}_\fund} w_p(M)}
	\end{align}
	for $M\in\mathcal{M}$, where the sum in the denominator is computed in the semifield $\P = \Trop(A_{i,j},B_{i,j},C_{i,j},D_{i,j} : (i,j)\in\Z^2_{\rm even})$.
	
	\begin{ex}
		Let $p=(0,0,1)$, then we get the graph $G$ and its two matchings as follows:
		\begin{center}
			\includegraphics[scale=0.6]{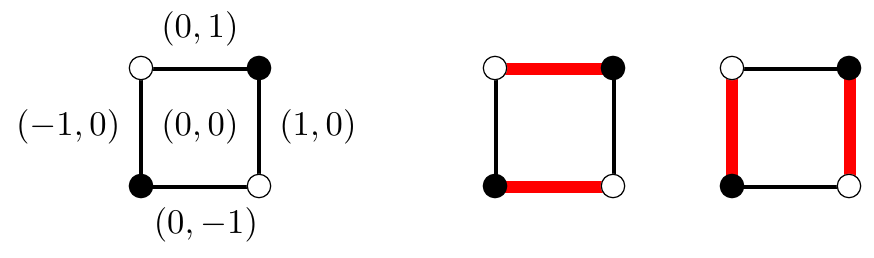}
		\end{center}
		Let $T_{i,j,\fund(i,j)} = t_{i,j}$ be the initial data. From Theorem \ref{thm:Speyer}, we have
		\[
			T^{(s)}_{0,0,1}= B_{0,0}D_{0,0}\,t_{-1,0}t_{0,0}^{-1}t_{1,0}+A_{0,0}C_{0,0}\,t_{0,-1}t_{0,0}^{-1}t_{0,1}.
		\]
		On the other hand, to apply the separation formula, we set $c_{0,0}=\frac{B_{0,0}D_{0,0}}{A_{0,0}C_{0,0}}$ as in \eqref{eq:speyercoef}.
		Equation \eqref{eq:speyerfromprinc} then gives
		\begin{align*}
			T^{(s)}_{0,0,1} &= \frac{c_{0,0}\,t_{-1,0}t_{0,0}^{-1}t_{1,0}+t_{0,-1}t_{0,0}^{-1}t_{0,1}}{c_{0,0}\oplus 1}\\
      &= B_{0,0}D_{0,0}\,t_{-1,0}t_{0,0}^{-1}t_{1,0}+A_{0,0}C_{0,0}\,t_{0,-1}t_{0,0}^{-1}t_{0,1}.
		\end{align*}
	\end{ex}

	\subsection{Lambda determinants}
	The generalized lambda-determinant \cite{DF13} can be considered as a recurrence relation on $\mathcal{T}^{(\lambda)}=\{T^{(\lambda)}_{i,j,k} \mid (i,j,k)\in\Zodd\}$ together with a set of coefficients $\{\lambda_i,\mu_i \mid i\in\Z\}$ satisfying
	\begin{align}
		T^{(\lambda)}_{i,j,k-1}T^{(\lambda)}_{i,j,k+1} =\lambda_i T^{(\lambda)}_{i-1,j,k}T^{(\lambda)}_{i+1,j,k} + \mu_j T^{(\lambda)}_{i,j-1,k}T^{(\lambda)}_{i,j+1,k}
		\label{eq:lambda}
	\end{align}
	for all $(i,j,k)\in\Zodd.$

	It can also be realized \cite{DF13} as a cluster algebra with coefficients.
	The quiver is the octahedron quiver, the initial cluster variables are $T_{i,j,\fund(i,j)}$ and the initial coefficients are
	\begin{align}
		y_{i,j}=
		\begin{cases}
			\lambda_i/\mu_j, &i+j\equiv 0\bmod 2,\\
			\mu_j/\lambda_i, &i+j\equiv 1\bmod 2,
		\end{cases}
		\label{eq:lambdadetcoef}
	\end{align}
	in $\P = \Trop( \lambda_i,\mu_i : i\in\Z )$.
	By Remark \ref{rem:GeometricFrozen}, we can also interpret it as a coefficient-free cluster algebra with frozen variables $\{\lambda_i,\mu_i \mid i\in\Z\}$ with the quiver illustrated in Figure \ref{fig_LambdaDet}.
	\begin{figure}%
		\begin{center}\includegraphics[scale=0.8]{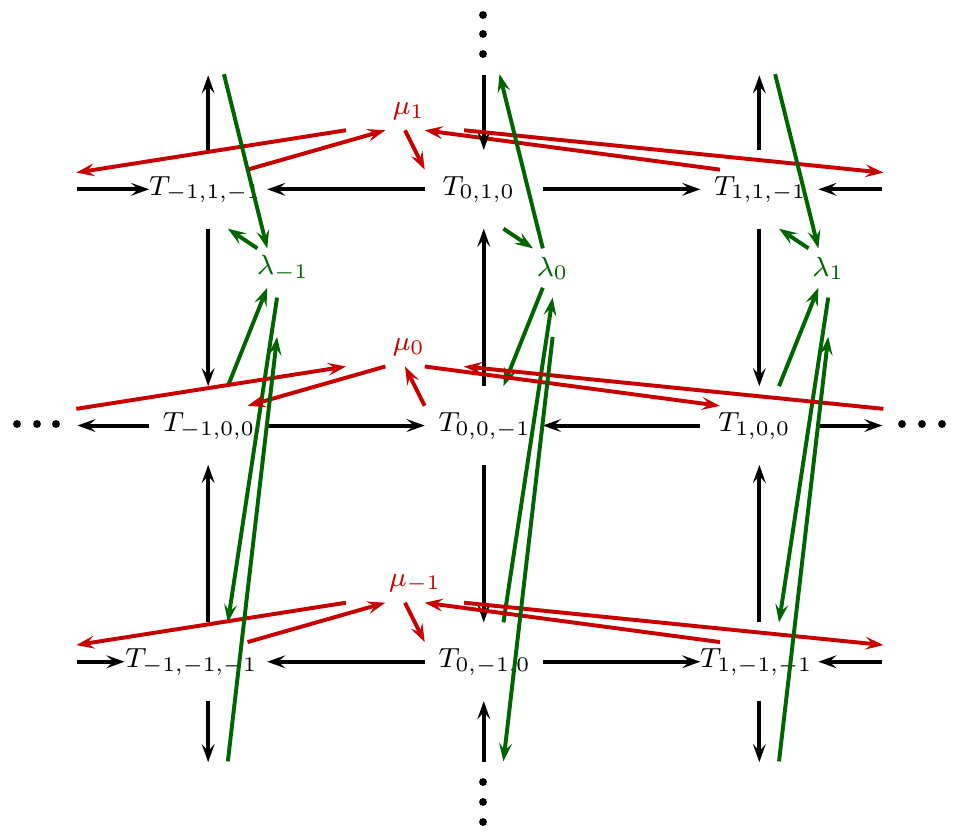}\end{center}
		\caption{A portion of the infinite quiver of the lambda-determinant.}%
		\label{fig_LambdaDet}%
	\end{figure}
	
	With the initial data stepped surface $\fund$, we let $T^{(\lambda)}_{i_0,j_0,k_0}$ be the expression in terms of the initial data $t_{i,j}:=T_{i,j,\fund(i,j)}$ and $\lambda_i,\mu_j$.
	By the separation formula (Theorem \ref{thm:sep}), we have
	\begin{align}
		T^{(\lambda)}_{i_0,j_0,k_0} = \frac{T_{i_0,j_0,k_0}(c_{i,j}=y_{i,j})}{T_{i_0,j_0,k_0}|_{\P}(t_{i,j}=1;c_{i,j}=y_{i,j})}
	\end{align}
	where $\P=\Trop(\lambda_i,\mu_i : i\in\Z)$ and
	\begin{align}
		y_{i,j}=
		\begin{cases}
			\lambda_i/\mu_j, &i+j\equiv 0\bmod 2,\\
			\mu_j/\lambda_i, &i+j\equiv 1\bmod 2.
		\end{cases}
		\label{eq:lambdacoef}
	\end{align}


\subsection{Pentagram maps}
	The pentagram map \cite{Schwartz,OST} is a discrete evolution on points in $\R\P^2$.
	It maps on a twisted polygon with $n$ vertices to give another twisted $n$-gon whose vertices are the intersections of the shortest diagonals of the original polygon.
	In \cite{Glick} the pentagram map evolution is interpreted as the mutation in a Y-pattern (cluster mutation on cluster coefficients).
	Using this interpretation, the authors of \cite{GSTV14} give a generalization of the pentagram map called \defemph{higher pentagram maps}.
	For a given integer $3\leq\kappa\leq n-1$, the higher pentagram map on a twisted $n-$gon produces a new polygon using $(\kappa-1)^{th}$-diagonals (connecting vertex $i$ to vertex $i+\kappa-1$) instead of the shortest diagonals (connecting vertex $i$ to vertex $i+2$) in the case when $\kappa=3$.
	The following is an example of one evolution of the higher pentagram map, $\kappa=4$, on a closed 9-gon.
	\begin{center}
		\includegraphics[scale=0.8]{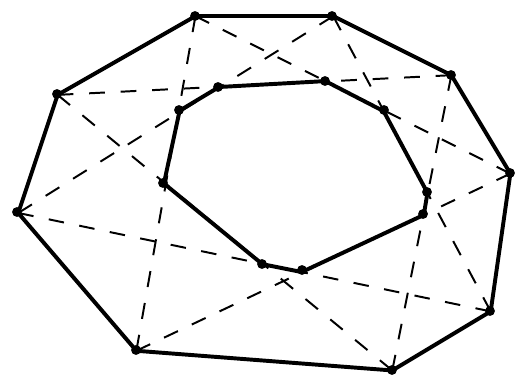}
	\end{center}
	
	For a twisted $n-$gon, one can define $2n$ variables $p_1,\dots,p_n,q_1,\dots,q_n \in \R$.
	Then the evolution \cite{Glick,GSTV14} of the higher pentagram map on these variables are as follows
	\begin{align*}
		q_i' = p_i^{-1}, \quad p_i' = q_i  \dfrac{(1+p_{i-r})(1+p_{i+r'})}{(1+p_{i-r-1}^{-1})(1+p_{i+r'+1}^{-1})},
	\end{align*}
	where $r=\lfloor\frac{\kappa-2}{2}\rfloor$ and $r'=\lceil\frac{\kappa-2}{2}\rceil$ and $p'_1,\dots,p'_n,q'_1,\dots,q'_n$ are the new variables associated with the new polygon produced by the higher pentagram map.
	We note that the variables originally defined in \cite{GSTV14} differ from the ones considered here by a change of variables and a shift in indices.
	See \cite{GSTV14} and \cite{KV14} for more details.
	
	The evolution of the variables $p_i,q_i$ can be also realized \cite{Glick,GSTV14} as the Y-pattern of a cluster algebra of rank $2n$ in the universal semifield $$\P=\mathbb{Q}_{sf}(p_1,\dots,p_n,q_1,\dots,q_n)$$ with the initial coefficient tuple $(p_1,\dots,p_n,q_1,\dots,q_n)$.
	The exchange matrix is $B=\begin{psmallmatrix}\mathbf{0}&C\\-C^T&\mathbf{0}\end{psmallmatrix}$ where $C=(c_{ij})$ is an $n\times n$ matrix defined by $c_{ij} = \delta_{i,j-1}-\delta_{i,j}-\delta_{i,j+1}+\delta_{i,j+1}$ (the indices are read modulo $n$).
	The quiver corresponding to $B$, the \defemph{generalized Glick's quiver}, is a bipartite graph with $2n$ vertices labeled by $p_1,\dots,p_n,q_1,\dots,q_n$ with four arrows adjacent to each $q_i$ as the following.
	\begin{center}
		\includegraphics[scale=0.7]{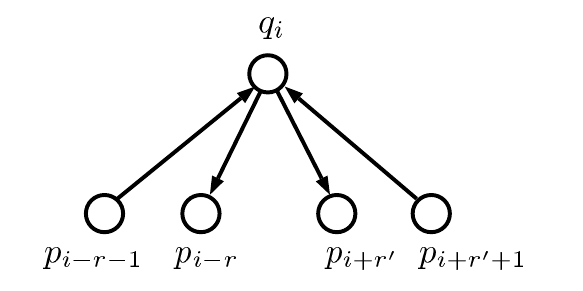}
	\end{center}
	Then the variables $p'_1,\dots,p'_n,q'_1,\dots,q'_n$ of the new polygon produced by the higher pentagram map is obtained by a composition of mutations at all $p_i$ for $i\in[1,n]$.
	
	Wrapping this finite quiver around a torus, we can then interpret it as the octahedron quiver with certain identification of vertices, along the two periods of the torus: $(i,j)\equiv(i+\kappa,j+2-\kappa)$ and $(i,j)\equiv(i+n,j-n)$ \cite{DFK13}.
	Hence we can realize the higher pentagram map as a Y-pattern on the octahedron quiver with periodic initial data $p_1,\dots,p_n,q_1,\dots,q_n$ in $\mathbb{Q}_{sf}(p_1,\dots,p_n,q_1,\dots,q_n)$, where only mutations at the vertex with exactly two incoming and two outgoing arrows are allowed.
	
	Let $\pi:\Z^2\rightarrow\{p_1,\dots,p_n,q_1,\dots,q_n\}$ be the map identifying the vertex of the octahedron quiver with the vertex of the generalized Glick's quiver.
  We pick $\pi$ by starting with $\pi(0,0)=p_n$ and then continuing by the following patterns:
  \[
    \begin{matrix}
    &q_{i-r'-1}&\\
    q_{i+r}&p_i&q_{i-r'}\\
    &q_{i+r+1}
    \end{matrix}
    \qquad\text{and}\qquad
    \begin{matrix}
    &p_{i-r-1}&\\
    p_{i+r'}&q_i&p_{i-r}\\
    &p_{i+r'+1}
    \end{matrix}
  \]
  where the indices are read modulo $n$. This choice of $\pi$ agrees with the map $\phi$ used for unfolding periodic two dimensional quivers in \cite[Section 12]{GP15} in the case when $S=\{(0,0),(\kappa-1,0),(r,1),(r+1,1)\}$ is a Y-pin for the higher pentagram map. The following is an example of the assignment by the map $\pi$ where the indices are read modulo $n$ and the center is $(0,0)$.
  \[
    \begin{matrix}
    &&&\vdots&&&\\
    &p_{n-2}&q_{n-r'-2}&p_{n-\kappa}&q_{n-\kappa-r'}&p_{n-2\kappa+2}&\\
    &q_{r-1}&p_{n-1}&q_{n-r'-1}&p_{n-\kappa+1}&q_{n-\kappa-r'+1}&\\
    \cdots&p_{\kappa-2}&q_{r}&p_n&q_{n-r'}&p_{n-\kappa+2}&\cdots\\
    &q_{\kappa+r-1}&p_{\kappa-1}&q_{r+1}&p_{1}&q_{n-r'+1}\\
    &p_{2\kappa-2}&q_{\kappa+r}&p_{\kappa}&q_{r+2}&p_{2}\\
    &&&\vdots&&&
    \end{matrix}
  \]

  Let $p^{(k)}_\ell$ and $q^{(k)}_\ell$ ($\ell\in[1,n]$) be the pentagram variables after the $k$-th iterate of the higher pentagram map.
  Then
  $$p^{(k)}_\ell = 1/q^{(k+1)}_\ell \text{ for all }\ell\in[1,n],$$
  and $q^{(k)}_\ell$ is the coefficient at a vertex $(i,j,k)$ on $\k$ such that
  $$\pi(i,j)=q_\ell \quad\text{and}\quad \k(i,j)=k=\k(i\pm 1,j\pm 1)+1.$$
  At the vertex $(i,j,k)$, the quiver associated with $\k$ will be as the following.
  \[
    \scalebox{0.8}{$
      \xymatrix@=1em{
        & (i,j+1,k-1)\ar[d] & \\
        (i-1,j,k-1) & (i,j,k)\ar[l]\ar[r] & (i+1,j,k-1) \\
        & (i,j-1,k-1)\ar[u] &
      }
    $}
  \]
	By Theorem \ref{thm:YasF} and Proposition \ref{prop:coefatvertex}, we then have
	\begin{align}\begin{aligned}
		q^{(k)}_\ell
		&= y_{(i,j),\k} \frac{ T_{i,j-1,k-1}T_{i,j+1,k-1}}{ T_{i-1,j,k-1}T_{i+1,j,k-1}}\bigg|_{\substack{T_{i,j,\fund(i,j)}=1\\c_{i,j}=\pi(i,j)}}\\
		&= \frac{I_{i,j,k-1}}{J_{i,j,k-1}} \frac{ T_{i,j-1,k-1}T_{i,j+1,k-1}}{ T_{i-1,j,k-1}T_{i+1,j,k-1}}\bigg|_{\substack{T_{i,j,\fund(i,j)}=1\\c_{i,j}=\pi(i,j)}}
		\label{eq:yijk}
	\end{aligned}\end{align}
	where $y_{(i,j),\k}$ is defined as in Proposition \ref{prop:coefatvertex}.
	This gives an expression of all the pentagram variables in terms of the solution to the T-system with principal coefficients.


\section{Conclusion and Discussion}\label{sec_General}
In this paper, we have defined the T-system with principal coefficients from cluster algebra aspect.
We obtain the octahedron recurrence with principal coefficients, which is a recurrence relation governing the T-system with principal coefficients.
Various explicit combinatorial solutions and their connection have been established.
This is for a special case when the point $p$ and the initial data stepped surface $\k$ satisfy the following conditions
	\begin{align}
		k_0 &\geq \k(i_0,j_0),\label{cond1}\\
		\k(i,j) &\geq \fund(i,j)\quad\text{for }(i,j)\in\Z^2.\label{cond2}
	\end{align}
These solutions to the T-system with principal coefficients allow us to solve any other systems having other choices of coefficients on the T-system as we seen in the previous section.
In particular, we are able to give a solution to the higher pentagram maps as a product of T-system variables and coefficients, see \eqref{eq:yijk}.

We notice a symmetry $\{i\leftrightarrow j,k\leftrightarrow -k-1\}$ of the T-system with principal coefficients \eqref{eq:tsyscoef}.
This symmetry basically switches the roles between $i$ and $j$ and reflects the system upside down.
So if we have a point $p$ and an initial data stepped surface $\k$ such that $k_0\leq\k(i_0,j_0)$ and $\k(i,j)\leq\fund(i,j)$ for $(i,j)\in\Z^2$, after applying the symmetry the system will satisfy the conditions \eqref{cond1} and \eqref{cond2}.
Furthermore, the condition \eqref{cond2} can be relaxed a little more.
Since the expression of $T_{i_0,j_0,k_0}$ depends only on the values of $\k(i,j)$ when $(i,j)\in\mathring{F}\cup\partial F$ (see \eqref{eq:opencloseface}), Condition \eqref{cond2} can be relaxed to
\begin{align*}
	\k(i,j) &\geq \fund(i,j)\quad\text{for }(i,j)\in\mathring{F}\cup\partial F. \label{newcond2}
\end{align*}
Nevertheless, an explicit combinatorial solution for arbitrary $p$ and $\k$ is still unknown.

Our general solution for the T-systems with principal coefficients may be applied to various problems related to the octahedron recurrence.
For instance, there are known connections between the T-systems and Bessenrodt-Stanley polynomials discussed in \cite{DF15}.
We expect the solutions to the T-systems with principal coefficients to provide generalizations of this family of polynomials.
It would also be interesting to apply our solutions to study the arctic curves of the octahedron equation with principal coefficients in the same spirit as \cite{DFS14}.
We expect the coefficients to act as additional probability for dimer configurations, which may give other shapes to the arctic curves.
Lastly, it would be interesting to investigate the quantum version of the T-systems with principal coefficients analog to \cite{DF11,DFK12}.


\end{document}